\documentclass[12pt]{amsart}
\usepackage[T1]{fontenc}
\usepackage{lmodern}
\usepackage{ifthen}
\usepackage{amsfonts}
\usepackage{amsxtra}
\usepackage{amssymb}
\usepackage{array}
\usepackage{eucal}
\usepackage[left=1.2in,right=1.2in,bottom=1.3in,top=1.3in]{geometry}
\usepackage{xcolor}
\definecolor{cite}{rgb}{0.50,0.00,1.00}
\definecolor{url}{rgb}{0.00,0.50,0.75}
\definecolor{link}{rgb}{0.00,0.00,0.50}
\usepackage[colorlinks,linkcolor=link,urlcolor=url,citecolor=cite,pagebackref,breaklinks]{hyperref}
\usepackage{colonequals}
\usepackage{mathtools}
\usepackage{mathrsfs}
\usepackage[all]{xy}
\usepackage[lite,abbrev,msc-links,alphabetic]{amsrefs}
\usepackage{enumerate}
\usepackage{enumitem}
\numberwithin{equation}{section}

\theoremstyle{plain}
\newtheorem{proposition}{Proposition}[subsection]

\newtheorem{corollary}[proposition]{Corollary}
\newtheorem{lem}[proposition]{Lemma}
\newtheorem{theorem}[proposition]{Theorem}

\theoremstyle{definition}
\newtheorem{definition}[proposition]{Definition}
\newtheorem{notation}[proposition]{Notation}

\theoremstyle{remark}
\newtheorem{example}[proposition]{Example}
\newtheorem{remark}[proposition]{Remark}

\renewcommand{\b}[1]{\mathbf{#1}}
\renewcommand{\c}[1]{\mathcal{#1}}
\renewcommand{\d}[1]{\mathbb{#1}}
\newcommand{\f}[1]{\mathfrak{#1}}
\renewcommand{\r}[1]{\mathrm{#1}}
\newcommand{\s}[1]{\mathscr{#1}}
\renewcommand{\sf}[1]{\mathsf{#1}}
\renewcommand{\(}{\left(}
\renewcommand{\)}{\right)}
\newcommand{\res}{\mathbin{|}}
\newcommand{\Sec}{\S}

\newcommand{\bD}{\b D}

\newcommand{\bL}{\b L}

\newcommand{\bR}{\b R}

\newcommand{\cA}{\c A}
\newcommand{\cB}{\c B}
\newcommand{\cC}{\c C}
\newcommand{\cD}{\c D}
\newcommand{\cE}{\c E}
\newcommand{\cF}{\c F}
\newcommand{\cG}{\c G}
\newcommand{\cH}{\c H}

\newcommand{\cX}{\c X}
\newcommand{\cY}{\c Y}

\newcommand{\dN}{\d N}

\newcommand{\dS}{\d S}

\newcommand{\dZ}{\d Z}

\newcommand{\fC}{\f C}

\newcommand{\fR}{\f R}

\newcommand{\fm}{\f m}

\newcommand{\rD}{\r D}
\newcommand{\rE}{\r E}

\newcommand{\rH}{\r H}

\newcommand{\rL}{\r L}

\newcommand{\rN}{\r N}

\newcommand{\rR}{\r R}

\newcommand{\ra}{\r a}
\newcommand{\rb}{\r b}
\newcommand{\rc}{\r c}

\newcommand{\rh}{\r h}

\newcommand{\rn}{\r n}

\newcommand{\sK}{\s K}

\newcommand{\sfK}{\sf K}
\newcommand{\sfL}{\sf L}

\newcommand{\sfp}{\sf p}
\newcommand{\sfq}{\sf q}

\newcommand{\all}{\r{all}}
\newcommand{\atimes}{\overset{\ra}\otimes}

\newcommand{\cart}{\r{cart}}
\newcommand{\Cat}{\c{C}\r{at}_\infty}
\newcommand{\cat}{\c{C}\r{at}_1}

\newcommand{\Chp}{\c{C}\r{hp}}
\newcommand{\Chpars}{\Chp^{\r{Ar}}_{\r{lft}/\dS}}

\newcommand{\colim}{\varinjlim}
\newcommand{\cons}{\r{cons}}

\newcommand{\cosk}{\r{cosk}}

\newcommand{\del}{\b{\Delta}}

\newcommand{\et}{\acute{\r{e}}\r{t}}

\newcommand{\Fibr}{\r{Fibr}}
\newcommand{\Fin}{\c{F}\r{in}_*}

\newcommand{\HOM}{\c{H}\r{om}}
\newcommand{\Hom}{\r{Hom}}
\newcommand{\id}{\r{id}}

\renewcommand{\lim}{\varprojlim}

\newcommand{\liset}{{\r{lis}\text{-}\acute{\r{e}}\r{t}}}
\newcommand{\ltor}{\Box\text{-}\tor}
\newcommand{\Mset}{\c{S}\r{et}_{\Delta}^+}

\newcommand{\PR}{\c{P}\r{r}}
\newcommand{\PRing}{\c{PR}\r{ing}}
\newcommand{\PRL}{\PR^{\rL}}

\newcommand{\PS}{\PR_{\r{st}}}
\newcommand{\PSL}{\PS^{\rL}}
\newcommand{\PSR}{\PS^{\rR}}
\newcommand{\PSLM}{\CAlg(\Cat)_{\r{pr,st,cl}}^\rL}
\newcommand{\PSLt}[1][]{\PR^{\r{L}}_{\r{st},\r{t}#1}}
\newcommand{\PSRt}[1][]{\PR^{\r{R}}_{\r{st},\r{t}#1}}

\newcommand{\Rind}{\c{R}\r{ind}}

\newcommand{\Ring}{\c{R}\r{ing}}

\newcommand{\Sch}{\c{S}\r{ch}}

\newcommand{\Schqcs}{\Sch^{\r{qc.sep}}}

\newcommand{\Set}{\c{S}\r{et}}
\newcommand{\St}{\mathnormal{St}}

\newcommand{\tor}{\r{tor}}
\newcommand{\trdeg}{\r{tr.deg}}
\newcommand{\Un}{\r{Un}}
\newcommand{\univ}{\r{univ}}
\newcommand{\Vosm}{\r{Vo}^{\r{sm}}}
\newcommand{\Voet}{\r{Vo}^{\et}}

\newcommand{\Chpar}[1]{\Chp^{\ifthenelse{\equal{#1}{}}{}{#1\text{-}}\r{Ar}}}
\newcommand{\Chpdm}[1]{\Chp^{\ifthenelse{\equal{#1}{}}{}{#1\text{-}}\r{DM}}}
\newcommand{\Chplft}[1]{\Chp^{\r{Ar}}_{\r{lft}/{#1}}}
\newcommand{\Chplmb}[1]{\Chp^{\r{LMB}}_{\r{lft}/{#1}}}
\newcommand{\EO}[4]{\prescript{#1}{#2}{\r{EO}}^{#3}_{#4}}
\newcommand{\Sset}[1][]{\Set_{#1\Delta}}
\newcommand{\TS}[3]{\prescript{#1}{}{#2}^{#3}}

\DeclareMathOperator{\CAlg}{CAlg}

\DeclareMathOperator{\codim}{codim}
\DeclareMathOperator{\Fun}{Fun}
\DeclareMathOperator{\Map}{Map}
\DeclareMathOperator{\Mod}{Mod}
\DeclareMathOperator{\Ob}{Ob}
\DeclareMathOperator{\Spec}{Spec}
\DeclareMathOperator{\Tr}{Tr}

\begin{document}

\title[Enhanced adic formalism for Artin stacks]
{Enhanced adic formalism and perverse t-structures for higher Artin stacks}

\author{Yifeng Liu}
\address{Department of Mathematics, Northwestern University, Evanston, IL 60208, United States}
\email{liuyf@math.northwestern.edu}

\author{Weizhe Zheng}
\address{Morningside Center of Mathematics, Academy of Mathematics and Systems Science, Chinese Academy of Sciences, Beijing 100190, China; University of the Chinese Academy of Sciences, Beijing
100049, China}
\email{wzheng@math.ac.cn}

\date{September 26, 2017}

\begin{abstract}
  In this sequel of \cites{LZ0,LZ1}, we develop an adic formalism for \'etale cohomology of Artin stacks and prove several desired properties including the base change theorem. In addition, we define perverse t-structures on Artin stacks for general perversity, extending Gabber's work on schemes. Our results generalize results of Laszlo and Olsson on adic formalism and middle perversity. We continue to work in the world of $\infty$-categories in the sense of Lurie, by enhancing all the derived categories, functors, and natural transformations to the level of $\infty$-categories.
\end{abstract}

\maketitle

\tableofcontents

\section*{Introduction}
\label{0}

In \cites{LZ0,LZ1}, we developed a theory of Grothendieck's six operations for \'etale cohomology of Artin stacks and prove several desired properties including the base change theorem. In the article, we develop the corresponding adic formalism and establish adic analogues of results in \cite{LZ1}. This extends all previous theories on the subject, including SGA~5 \cite{SGA5}, Deligne \cite{WeilII}, Ekedahl \cite{Ekedahl} (for schemes), Behrend \cite{Behrend} and Laszlo--Olsson \cite{LO2}. We prove, among other things, the base change theorem in derived categories, which was previous known only on the level of sheaves \cite{LO2} (and under other restrictions). Another limitation of the existing theories, including those for schemes, is the constructibility assumption. This assumption is not often met, for example, when considering morphisms between Artin stacks that are only locally of finite type. By contrast, the adic formalism developed in this article applies to unrestricted derived categories.

In addition, we define perverse t-structures on Artin stacks for general perversity, extending the work of Gabber \cite{Gabber} for schemes and the work of Laszlo and Olsson \cite{LO3} for the middle perversity.

As in our preceding article, the approach we are taking is different from all the previous theories. We work in the world of $\infty$-categories in the sense of Lurie \cites{Lu1,Lu2}, by enhancing all the derived categories, functors, and natural transforms to the level of $\infty$-categories. At this level, we may use some new machineries among which the most important ones are gluing objects, Adjoint Functor Theorem, $\infty$-categorical descent, all in \cites{Lu1,Lu2}, and some other techniques developed in \cite{LZ1}. In particular, we obtain several other special descent properties for the derived category of lisse-\'{e}tale sheaves.

\subsection{Six operations in adic formalism}
\label{0ss:result}

In this and the next sections, we will state our constructions and results only in the classical setting of Artin stacks on the level of usual derived categories (which are homotopy categories of the derived $\infty$-categories), among other simplification. We will provide the precise references of the complete results in later chapters, for higher Artin stack higher (and higher Deligne--Mumford stacks), stated on the level of stable $\infty$-categories. We refer the reader to \cite{LZ1}*{\Sec 0.1} for our convention on Artin stacks.

Let $\cX$ be an Artin stack and let $\lambda=(\Xi,\Lambda)$ be a ringed diagram, that is, a functor $\Lambda$ from a partially ordered set $\Xi$ to the category of unital commutative rings. Recall that for every $\xi\in\Xi$, $\rD_{\cart}(\cX_{\liset},\Lambda(\xi))$ is the full subcategory of $\rD(\cX_{\liset},\Lambda(\xi))$ spanned by complexes whose cohomology sheaves are all Cartesian. It has a natural $\infty$-categorical enhancement $\cD(\cX,\Lambda(\xi))$. In fact, we have a functor $\rN(\Xi)^{op}\to\Cat$ from the nerve of $\Xi^{op}$ to the $\infty$-category of $\infty$-categories sending $\xi$ to $\cD(\cX,\Lambda(\xi))$, with the transition functors being (derived) extension of scalars. We define
\[
\cD(\cX,\lambda)_\ra\coloneqq\varprojlim_{\rN(\Xi)^{op}}\cD(\cX,\Lambda(\xi))
\]
and let $\rD(\cX,\lambda)_\ra$ be its homotopy category. It is crucial that the limit is taken on the level of $\infty$-categories.

Let $f\colon\cY\to\cX$ be a morphism of Artin stacks. We then define operations:
\begin{align*}
f^{*\ra}&\colon\rD(\cX,\lambda)_\ra\to\rD(\cY,\lambda),\\
f_{*\ra}&\colon\rD(\cY,\lambda)_\ra\to\rD(\cX,\lambda),\\
-\atimes_\cX-&\colon\rD(\cX,\lambda)_\ra\times\rD(\cX,\lambda)_\ra\to\rD(\cX,\lambda)_\ra,\\
\HOM^\ra_\cX&\colon\rD(\cX,\lambda)_\ra^{op}\times\rD(\cX,\lambda)_\ra\to\rD(\cX,\lambda)_\ra.
\end{align*}
The pairs $(f^{*\ra},f_{*\ra})$ and $(-\atimes_\cX\sfK,\HOM^\ra_\cX(\sfK,-))$ for every $\sfK\in\rD(\cX,\lambda)_\ra$ are pairs of adjoint functors.

To state the other two operations, we fix a nonempty set $\Box$ of rational primes. Recall that a ring is \emph{$\Box$-torsion} \cite{SGA4}*{IX 1.1} if each element of it is killed by an integer that is a product of primes in $\Box$. An Artin stack $\cX$ is {\em $\Box$-coprime} if there exists a morphism $\cX\to \Spec\dZ[\Box^{-1}]$. If $\cX$ and $\cY$ are $\Box$-coprime, $f\colon \cY\to \cX$ is locally of finite type, and $\lambda$ is a $\Box$-torsion ringed diagram, then we have another pair of adjoint functors:
\begin{align*}
f_{!\ra}&\colon\rD(\cY,\lambda)_\ra\to\rD(\cX,\lambda)_\ra,\\
f^{!\ra}&\colon\rD(\cX,\lambda)_\ra\to\rD(\cY,\lambda)_\ra.\\
\end{align*}
The functors $f^{*\ra}$, $f_{!*}$ and $-\atimes_\cX-$ are naturally defined from the limit construction of $\rD(-,\lambda)_\ra$.

In \Sec\ref{1ss:relation}, we show that $\rD(\cX,\lambda)_\ra$ is canonically equivalent to a full subcategory of $\rD(\cX,\lambda)$ spanned by so-called \emph{adic complexes}, which admits a \emph{colocalization functor} $\fR_\cX\colon\rD(\cX,\lambda)\to\rD(\cX,\lambda)_\ra$. Moreover, $f^{*\ra}$, $f_{!*}$ and $-\atimes_\cX-$ are simply restrictions of $f^*$, $f_!$ and $-\otimes_\cX-$, respectively, as they preserve adic complexes. For the other three, we have $f_{*\ra}=\fR_\cX\circ f_*$, $f^{!\ra}=\fR_\cY\circ f^!$ and $\HOM^\ra_\cX=\fR_\cX\circ\HOM_\cX$. These operations satisfy the similar properties as in the non-adic version. Moreover, the exact category $\cD(\cX,\lambda)_\ra$ carries a \emph{usual t-structure} $(\cD^{\leq 0}(\cX,\lambda)_\ra,\cD^{\geq 0}(\cX,\lambda)_\ra)$. We refer the reader to \Sec\ref{1ss:adic_properties} and \Sec\ref{1ss:relation} for more details.

The adic formalism introduced above does \emph{not} assume the constructibility at the first place. In other words, we are free to talk
about adic complexes for any sheaves. In particular, in terms of Grothendieck's fonctions-faisceaux dictionary, we make sense of divergent
integrals on stacks over finite fields, those appear for example in \cite{FN}.

In \Sec\ref{2}, we introduce a special setup of the adic formalism, namely, the $\fm$-adic formalism on which there is a good notion of constructibility. Such formalism is enough for most applications. Let $\Lambda$ be a ring and $\fm\subseteq \Lambda$ be a principal ideal, satisfying the conditions in Definition \ref{2de:pring}. The typical example is that $\Lambda$ is a $1$-dimensional valuation ring and $\fm$ is a proper ideal. The pair $(\Lambda,\fm)$ corresponds to a ringed diagram $\Lambda_{\bullet}$ with the underlying category $\dN=\{0\to1\to2\to\cdots\}$ and $\Lambda_n=\Lambda/\fm^{n+1}$. Now we fix a pair $(\Lambda,\fm)$ as above such that $\Lambda$ is Noetherian and $\Lambda/\fm$ is $\Box$-torsion. Let $\dS$ be either a quasi-excellent finite-dimensional scheme or a regular scheme of dimension $\leq1$, that is $\Box$-coprime. Consider Artin stacks that are locally of finite type over $\dS$. In this setup, we define the intersection
\[
\rD(\cX,\Lambda_\bullet)_{\ra,\rc}\coloneqq\rD_{\cons}(\cX_{\liset},\Lambda_{\bullet})\cap\rD(\cX,\Lambda_{\bullet})_\ra
\subseteq\rD(\cX_{\liset},\Lambda_{\bullet})
\]
of constructible complexes and adic complexes as the category of constructible adic complexes. We assume that $\cX$ is \emph{locally $\Lambda/\fm$-bounded} (Definition \ref{2de:bounded}). Then we show in Corollary \ref{2co:constructible_adic_t} that the usual t-structure on $\rD(\cX,\Lambda_\bullet)_\ra$ restricts to a t-structure on the full subcategory $\rD(\cX,\Lambda_\bullet)_{\ra,\rc}$. Moreover, for a morphism $f\colon\cY\to\cX$ of Artin stacks (that are locally of finite type over $\dS$), the six operations mentioned previously restrict to the following refined ones:
\begin{align*}
f^{*\ra}&\colon\rD(\cX,\Lambda_\bullet)_{\ra,\rc}\to\rD(\cY,\Lambda_\bullet)_{\ra,\rc},\\
-\atimes_\cX-&\colon\rD^{(-)}(\cX,\Lambda_\bullet)_{\ra,\rc}\times\rD^{(-)}(\cX,\Lambda_\bullet)_{\ra,\rc}
\to\rD^{(-)}(\cX,\Lambda_\bullet)_{\ra,\rc},\\
\HOM^\ra_\cX&\colon\rD^{(-)}(\cX,\Lambda_\bullet)_{\ra,\rc}^{op}\times\rD^{(+)}(\cX,\Lambda_\bullet)_{\ra,\rc}
\to\rD^{(+)}(\cX,\Lambda_\bullet)_{\ra,\rc};
\end{align*}
if $\dS$ is locally finite-dimensional, then we have
\[
f^{!\ra}\colon\rD(\cX,\Lambda_\bullet)_{\ra,\rc}\to\rD(\cY,\Lambda_\bullet)_{\ra,\rc};
\]
if $f$ is of finite presentation \cite{LZ1}*{5.4.3}, then we have
\[
f_{!\ra}\colon\rD^{(-)}(\cY,\Lambda_\bullet)_{\ra,\rc}\to\rD^{(-)}(\cX,\Lambda_\bullet)_{\ra,\rc};
\]
if $f$ is quasi-compact and quasi-separated \cite{LZ1}*{5.4.3}, then we have
\[
f_{*\ra}\colon\rD^{(+)}(\cY,\Lambda_\bullet)_{\ra,\rc}\to\rD^{(+)}(\cX,\Lambda_\bullet)_{\ra,\rc}.
\]
See \Sec\ref{2ss:constructible_adic} for more and precise statements under various assumptions. In \Sec\ref{2ss:compatibility}, we show that our theory of constructible adic formalism coincides with Laszlo--Olsson \cite{LO3} under their assumptions.

\subsection{Perverse t-structures and hyperdescent}

In \Sec\ref{3}, we define the perverse t-structure, in both non-adic and adic settings, for general ``perversity'' for (higher) Artin stacks, while in all previous theory only middle perversity is considered \cite{LO3}. We introduce the notion of \emph{perversity smooth evaluation} $\sfp$ on an Artin stack $\cX$ (Definition \ref{3de:perversity_evaluation}) to be an assignment to each atlas $u\colon X\to \cX$ a weak perversity function $\sfp_u$ on $X$ in the sense of Gabber \cite{Gabber}, satisfying certain compatibility condition. In particular, when $\cX$ is a scheme, a perversity smooth evaluation is essentially same as a weak perversity function.

\begin{theorem}[(Adic) perverse t-structure, \Sec\ref{3ss:perverse_t} \& \Sec\ref{3ss:adic_perverse_t}]
Let $\cX$ be a $\Box$-coprime Artin stack equipped with a perversity smooth evaluation $\sfp$ and $\lambda$ an $\Box$-torsion ringed diagram.
\begin{enumerate}
  \item There is a unique up to equivalence t-structure $(\TS{\sfp}{\rD}{\leq0}(\cX,\lambda),\TS{\sfp}{\rD}{\geq0}(\cX,\lambda))$ on $\rD(\cX,\lambda)=\rD_{\cart}(\cX_{\liset},\lambda)$, called the \emph{perverse t-structure}, such that for every atlas $u\colon X\to\cX$, we have $u^*\TS{\sfp}{\rD}{\leq0}(\cX,\lambda)=\TS{\sfp_u}{\rD}{\leq0}(X,\lambda)$ and $u^*\TS{\sfp}{\rD}{\geq0}(\cX,\lambda)=\TS{\sfp_u}{\rD}{\geq0}(X,\lambda)$, where the corresponding t-structure on the scheme $X$ is defined by Gabber \cite{Gabber}.

  \item If $f\colon\cY\to\cX$ is a smooth morphism, then $f^*$ is perverse t-exact with respect to compatible perversity smooth evaluations $\sfp$ on $\cX$ and $\sfq$ on $\cY$.

  \item We have similar results in the adic setting, where $\TS{\sfp}{\rD}{\leq0}(\cX,\lambda)_\ra=\TS{\sfp}{\rD}{\leq0}(\cX,\lambda)\cap\rD(\cX,\lambda)_\ra$.

  \item Moreover, the classical description of the perverse t-structure via cohomology on stalks again holds (Remark \ref{3re:stalk} and Proposition \ref{3pr:adic_perverse_stalk}).
\end{enumerate}
\end{theorem}

In particular, when $\sfp=0$, we recover the usual t-structure in the non-adic case and obtain the similar usual t-structure in the adic case. When $\sfp$ is the middle perversity evaluation, we generalize the classical notion of middle perverse t-structure for schemes to Artin stacks, in both non-adic and adic cases.

In \Sec\ref{3ss:constructible_adic_perverse_t}, we show that in the $\fm$-adic formalism, under certain conditions on $(\Lambda,\fm)$ and the perversity smooth evaluation $\sfp$, the adic perverse t-structure restricts to the one on $\rD(\cX,\Lambda_{\bullet})_{\ra,\rc}$. In particular, when $\sfp$ is the middle perversity smooth evaluation (that is, the middle perversity function in the case of schemes), the corresponding (adic) perverse t-structure coincides with the one defined by Laszlo--Olsson \cite{LO3}, under their further restrictions on $(\Lambda,\fm)$ and $\cX$.

In \Sec\ref{4}, we prove several hyperdescent properties of derived $\infty$-categories and their adic version we have constructed. In particular, we have the following theorem, which is the incarnation\footnote{To deduce Theorem \ref{0th:descent}, we use the same argument in \cite{LZ1}*{6.2.14}.} on the level of usual derived categories of the main results in this chapter. For simplicity, we only state in the $\fm$-adic formalism.

\begin{theorem}\label{0th:descent}
Let (P) be a property that is either \emph{smooth}, \emph{proper}, or \emph{flat}. Let $f\colon \cY\to\cX$ be a morphism of Artin stacks and $y\colon\cY^+_0\to\cY$ be a (P) surjective morphism. Let $\cY^+_{\bullet}$ be a (P) hypercovering of $\cY$ (Definition \ref{4de:covering}) with the morphism $y_n\colon\cY^+_n\to\cY^+_{-1}\coloneqq\cY$. Put $f_n\coloneqq f\circ y_n\colon\cY^+_n\to\cX$ for $n\geq 0$. Fix a pair $(\Lambda,\fm)$ as in the $\fm$-adic formalism.
\begin{enumerate}
  \item If (P) is proper or flat, then we assume that $\cY$ is locally $\Lambda/\fm$-bounded. For every complex $\sfK\in\rD^{\geq0}(\cY,\Lambda_\bullet)_\ra$, we have a convergent spectral
      sequence
      \[
      \rE_1^{p,q}=\rH^q((f_p)_{*\ra}(y_p)^{*\ra}\sfK)\Rightarrow \rH^{p+q}f_{*\ra}\sfK.
      \]

  \item Suppose that (P) is smooth. If $\cX$ is $\Box$-coprime; $\Lambda/\fm$ is $\Box$-torsion; and $f$ is locally of finite type, then for every complex $\sfK\in\rD^{\leq0}(\cY,\Lambda_\bullet)_\ra$, we have a convergent spectral sequence
      \[
      \tilde{\rE}_1^{p,q}=\rH^q ((f_{-p})_{!\ra}(y_{-p})^{!\ra}\sfK)\Rightarrow \rH^{p+q}f_{!\ra}\sfK.
      \]
\end{enumerate}
\end{theorem}

Finally, we would like to emphasize that all conventions and notation from \cite{LZ1}, especially those in \Sec0.5 there, will be continually adopted in the current article, unless otherwise specified.

\subsubsection*{Acknowledgments}

We thank Ofer~Gabber, Luc~Illusie, Aise~Johan~de~Jong, Jo\"el~Riou,
Shenghao~Sun, and Xinwen~Zhu for useful conversations. Part of this work was
done during visits of the first author to the Morningside Center of
Mathematics, Chinese Academy of Sciences, in Beijing for several times. He
thanks the Center for its hospitality. The first author was partially
supported by NSF grant DMS--1302000. The second author was partially
supported by China's Recruitment Program of Global Experts; National Natural
Science Foundation of China Grants 11321101, 11621061, 11688101; National
Center for Mathematics and Interdisciplinary Sciences and Hua~Loo-Keng Key
Laboratory of Mathematics, Chinese Academy of Sciences.

\section{The adic formalism}
\label{1}

In this chapter, we provide the adic formalism for Grothendieck's six operations. In \Sec\ref{1ss:limit}, we provide our adic formalism by constructing two enhanced operation maps via the limit construction. In \Sec\ref{1ss:adic_properties}, we study several properties of the enhanced operation maps we constructed previously. In \Sec\ref{1ss:relation}, we study the relation between the limit construction and so-called adic complexes. In \Sec\ref{1ss:adic_dualizing}, we construct adic dualizing complexes and study biduality properties.

\subsection{The limit construction}
\label{1ss:limit}

Recall from \cite{LZ1} that for higher Artin stacks, we construct the \emph{first enhanced operation map}
\begin{align*}
\EO{}{\Chpar{}}{\r{I}}{}\colon((\Chpar{})^{op}\times\rN(\Rind)^{op})^\amalg\to\Cat,
\end{align*}
and the \emph{second enhanced operation map}
\begin{align*}
\EO{}{\Chpar{}_\Box}{\r{II}}{}\colon \delta^*_{2,\{2\}}(((\Chpar{}_\Box)^{op}\times\rN(\Rind_{\ltor})^{op})^{\amalg,op})^\cart_{F,\all}\to\Cat.
\end{align*}
Their restrictions to the common domain $((\Chpar{}_\Box)^{op}\times\rN(\Rind_{\ltor})^{op})^\amalg$ are equivalent. In particular, for every object $X$ of $\Chpar{}$ and every object $\lambda=(\Xi,\Lambda)$ of $\Rind$, we obtain a diagram $\Xi^{op}\to\PSL$ given by $\xi\mapsto\cD(X,\Lambda(\xi))$ with the transition map given by extension of scalars.

\begin{definition}
We define the \emph{adic derived $\infty$-category} of $\lambda$-modules on $X$ to be
\[
\cD(X,\lambda)_\ra\coloneqq\varprojlim_{\rN(\Xi)^{op}}\cD(X,\Lambda(\xi)).
\]
\end{definition}

The goal of this section is to make the above definition functorial in a homotopy coherent way. Namely, we will construct the \emph{first enhanced adic operation map}
\begin{align}\label{1eq:premonoidal_artin_adic}
\EO{\ra}{\Chpar{}}{\r{I}}{}\colon((\Chpar{})^{op}\times\rN(\Rind)^{op})^\amalg\to\Cat,
\end{align}
and the \emph{second enhanced adic operation map}
\begin{align}\label{1eq:operation_artin_adic}
\EO{\ra}{\Chpar{}_\Box}{\r{II}}{}\colon \delta^*_{2,\{2\}}(((\Chpar{}_\Box)^{op}\times\rN(\Rind_{\ltor})^{op})^{\amalg,op})^\cart_{F,\all}\to\Cat,
\end{align}
such that their values on $(X,\lambda)$ are both (equivalent to) $\cD(X,\lambda)_\ra$.

By definition, there is a tautological functor $\Rind\to\cat$ sending $(\Xi,\Lambda)$ to $\Xi$. Applying Grothendieck's construction, we obtain an op-fibration $\pi\colon\Rind^\univ\to\Rind$. More precisely, $\Rind^\univ$ is an ordinary category whose objects are pairs $((\Xi,\Lambda),\xi)$ where $(\Xi,\Lambda)$ is an object of $\Rind$ and $\xi$ is an object of $\Xi$, and a morphism from $((\Xi,\Lambda),\xi)$ to $((\Xi',\Lambda'),\xi')$ is a morphism $(\Gamma,\gamma)\colon(\Xi,\Lambda)\to(\Xi',\Lambda')$ of $\Rind$ such that $\Gamma(\xi)$ admits an arrow to $\xi'$. We have another functor $\sigma\colon\Rind^\univ\to\Rind$ sending $((\Xi,\Lambda),\xi)$ to $(\ast,\Lambda(\xi))$. We have two natural inclusion
\begin{align*}
j_0&\colon\rN(\Rind)^{op}\to\rN(\Rind)^{op}\diamond_{\rN(\Rind)^{op}}\rN(\Rind^\univ)^{op},\\
j_1&\colon\rN(\Rind^\univ)^{op}\to\rN(\Rind)^{op}\diamond_{\rN(\Rind)^{op}}\rN(\Rind^\univ)^{op}
\end{align*}
of simplicial sets.

To construct \eqref{1eq:premonoidal_artin_adic}, we start from the map
\[
\EO{\sigma}{\Chpar{}}{\r{I}}{}\colon((\Chpar{})^{op}\times\rN(\Rind^\univ)^{op})^\amalg\to\Cat
\]
as the composition of
\[
(\id_{(\Chpar{})^{op}}\times\rN(\sigma)^{op})^\amalg\colon
((\Chpar{})^{op}\times\rN(\Rind^\univ)^{op})^\amalg\to((\Chpar{})^{op}\times\rN(\Rind)^{op})^\amalg
\]
and $\EO{}{\Chpar{}}{\r{I}}{}$. Taking the right Kan extension of $\EO{\sigma}{\Chpar{}}{\r{I}}{}$ along the inclusion
\[
((\Chpar{})^{op}\times\rN(\Rind^\univ)^{op})^\amalg\hookrightarrow
((\Chpar{})^{op}\times\rN(\Rind)^{op}\diamond_{\rN(\Rind)^{op}}\rN(\Rind^\univ)^{op})^\amalg
\]
induced by $j_1$, and restricting to $((\Chpar{})^{op}\times\rN(\Rind)^{op})^\amalg$ via $j_0$, we obtain the desired map $\EO{\ra}{\Chpar{}}{\r{I}}{}$ \eqref{1eq:premonoidal_artin_adic}.

The construction of \eqref{1eq:operation_artin_adic} is similar. We have the map
\[
\EO{\sigma}{\Chpar{}_\Box}{\r{II}}{}\colon
\delta^*_{2,\{2\}}(((\Chpar{}_\Box)^{op}\times\rN(\Rind_{\ltor}^\univ)^{op})^{\amalg,op})^\cart_{F,\all}\to\Cat,
\]
where $\Rind_{\ltor}^\univ=\Rind^\univ\times_{\Rind}\Rind_{\ltor}$ in which the first functor in the fiber product is $\pi$. Taking the right Kan extension of $\EO{\sigma}{\Chpar{}_\Box}{\r{II}}{}$ along the inclusion
\begin{multline*}
\delta^*_{2,\{2\}}(((\Chpar{}_\Box)^{op}\times\rN(\Rind_{\ltor}^\univ)^{op})^{\amalg,op})^\cart_{F,\all}\\
\hookrightarrow\delta^*_{2,\{2\}}(((\Chpar{}_\Box)^{op}\times
\rN(\Rind_{\ltor})^{op}\diamond_{\rN(\Rind_{\ltor})^{op}}\rN(\Rind_{\ltor}^\univ)^{op})^{\amalg,op})^\cart_{F,\all}
\end{multline*}
induced by $j_1$, and restricting to $\delta^*_{2,\{2\}}(((\Chpar{}_\Box)^{op}\times\rN(\Rind_{\ltor})^{op})^{\amalg,op})^\cart_{F,\all}$ via $j_0$, we obtain the desired map $\EO{\ra}{\Chpar{}_\Box}{\r{II}}{}$ \eqref{1eq:operation_artin_adic}.

By the similar process, we obtain enhanced adic operation maps for higher Deligne--Mumford stacks:
\begin{align*}
\EO{\ra}{\Chpdm{}}{\r{I}}{}\colon((\Chpdm{})^{op}\times\rN(\Rind)^{op})^\amalg\to\Cat,
\end{align*}
and a map
\begin{align*}
\EO{\ra}{\Chpdm{}}{\r{II}}{}\colon \delta^*_{2,\{2\}}(((\Chpdm{})^{op}\times\rN(\Rind_\tor)^{op})^{\amalg,op})^\cart_{F,\all}\to\Cat,
\end{align*}
satisfying the obvious compatibility properties with higher Artin stacks.

\subsection{Properties of enhanced adic operations}
\label{1ss:adic_properties}

In this section, we study properties of the two enhanced adic operation maps constructed previously, in a way parallel to the non-adic ones in \cite{LZ1}.

To simplify notation, we will only discuss properties for higher Artin stacks, that is, the two maps \eqref{1eq:premonoidal_artin_adic} and \eqref{1eq:operation_artin_adic}. We will leave the analogous discussion for higher DM stacks to readers.

\begin{proposition}\label{1pr:monoidal}
We have
\begin{description}
  \item[(P0)] (Monoidal symmetry) The functor $\EO{\ra}{\Chpar{}}{\r{I}}{}$ is a weak Cartesian structure \cite{LZ1}*{1.5.6}, and the induced functor $\EO{\ra}{\Chpar{}}\otimes{}\coloneqq(\EO{\ra}{\Chpar{}}{\r{I}}{})^\otimes$ factorizes through $\PSLM$.

  \item[(P1)] (Disjointness) The map $\EO{\ra}{\Chpar{}}\otimes{}$ sends small coproducts to products.

  \item[(P2)] (Compatibility) The restrictions of $\EO{\ra}{\Chpar{}}{\r{I}}{}$ and $\EO{\ra}{\Chpar{}_\Box}{\r{II}}{}$ to the common domain $((\Chpar{}_\Box)^{op}\times\rN(\Rind_{\ltor})^{op})^\amalg$ are equivalent functors.
\end{description}
\end{proposition}

\begin{proof}
By construction, the value of $\EO{\ra}{\Chpar{}}{\r{I}}{}$ on an object $((X_1,\lambda_1),\dots,(X_n,\lambda_n))$ in the target is an $\infty$-category equivalent to
\[
\prod_{i=1}^n\cD(X_i,\lambda_i)_\ra=\prod_{i=1}^n\varprojlim_{\Xi_i^{op}}\cD(X_i,\Lambda_i(\xi))
\]
if $\lambda_i=(\Xi_i,\Lambda_i)$. We also note that the inclusion functor $\PSLM\to\CAlg(\Cat)$ preserves small limits. Therefore, (P0) and (P1) follow immediately. (P2) is clear from the construction.
\end{proof}

Before discussing the other properties, we introduce more notation. Similar to the non-adic case, we have the map
\begin{align}\label{1eq:operation1}
\EO{\ra}{\Chpar{}_\Box}{*}{!}\colon \delta^*_{2,\{2\}}\rN(\Chpar{}_\Box)^\cart_{F,\all}\times\rN(\Rind_{\ltor})^{op}\to\PSL
\end{align}
induced from \eqref{1eq:operation_artin_adic}.

Evaluating \eqref{1eq:premonoidal_artin_adic} at the object $\langle 1\rangle\in\Fin$, we obtain the map
\begin{align}\label{1eq:upperstar_adic}
\EO{\ra}{\Chpar{}}{*}{}\colon \rN(\Chpar{})^{op}\times\rN(\Rind)^{op}\to\PSL.
\end{align}
Note that this is equivalent to the map by restricting \eqref{1eq:operation1} to the second direction, on $\rN(\Chpar{})^{op}\times\rN(\Rind_{\ltor})^{op}$. Taking right adjoints, we obtain the map
\begin{align}\label{1eq:lowerstar_adic}
\EO{\ra}{\Chpar{}}{}{*}\colon \rN(\Chpar{})\times\rN(\Rind)\to\PSR.
\end{align}
Restricting \eqref{1eq:operation1} to the first direction, we obtain the map
\begin{align}\label{1eq:lowersh_adic}
\EO{\ra}{\Chpar{}_\Box}{}{!}\colon \rN(\Chpar{}_\Box)_F\times\rN(\Rind_{\ltor})^{op}\to\PSL.
\end{align}
Again by taking right adjoints, we obtain the map
\begin{align}\label{1eq:uppersh_adic}
\EO{\ra}{\Chpar{}_\Box}{!}{}\colon \rN(\Chpar{}_\Box)^{op}_F\times\rN(\Rind_{\ltor})\to\PSR.
\end{align}

More concretely, we have the following enhance adic operations:
\begin{description}
  \item[1L] $f^{*\ra}\colon\cD(X,\lambda)_\ra\to\cD(Y,\lambda)_\ra$, obtained by applying \eqref{1eq:upperstar_adic} to a morphism $f\colon Y\to X$ in $\Chpar{}$ and an object $\lambda=(\Xi,\lambda)\in\Rind$. It coincides with the limit of functors $f_\xi^*\colon\cD(X,\Lambda(\xi))\to\cD(Y,\Lambda(\xi))$ over $\Xi^{op}$, and underlies a monoidal functor $f^{*\otimes\ra}\colon\cD(X,\lambda)^\otimes_\ra\to\cD(Y,\lambda)^\otimes_\ra$ obtained from $\EO{\ra}{\Chpar{}}\otimes{}$.

  \item[1R] $f_{*\ra}\colon\cD(Y,\lambda)_\ra\to\cD(X,\lambda)_\ra$, obtained by applying \eqref{1eq:lowerstar_adic} to a morphism $f\colon Y\to X$ in $\Chpar{}$ and an object $\lambda\in\Rind$. It is right adjoint to $f^{*\ra}$.

  \item[2L] $f_{!\ra}\colon\cD(Y,\lambda)_\ra\to\cD(X,\lambda)_\ra$, obtained by applying \eqref{1eq:lowersh_adic} to a morphism $f\colon Y\to X$ in $\Chpar{}_\Box$ and an object $\lambda=(\Xi,\Lambda)\in\Rind_{\ltor}$. It coincides with the limit of functors $f_{\xi!}\colon\cD(Y,\Lambda(\xi))\to\cD(X,\Lambda(\xi))$ over $\Xi^{op}$.

  \item[2R] $f^{!\ra}\colon\cD(X,\lambda)_\ra\to\cD(Y,\lambda)_\ra$, obtained by applying \eqref{1eq:uppersh_adic} to a morphism $f\colon Y\to X$ in $\Chpar{}_\Box$ and an object $\lambda\in\Rind_{\ltor}$. It is right adjoint to $f_{!\ra}$.

  \item[3L] $-\atimes_X-\colon\cD(X,\lambda)_\ra\times\cD(X,\lambda)_\ra\to\cD(X,\lambda)_\ra$, the symmetric tensor product obtained from Proposition \ref{1pr:monoidal} (P0) for every object $(X,\lambda)$ of $\Chpar{}\times\rN(\Rind)$.

  \item[3R] $\HOM^\ra_X\colon\cD(X,\lambda)_\ra^{op}\times\cD(X,\lambda)_\ra\to\cD(X,\lambda)_\ra$, induced from $-\atimes_X-$ in a same way as $\HOM_X$ was induced from $-\otimes_X-$ in \cite{LZ1}*{\Sec6.2}. In particular, for every object $\sfK\in\cD(X,\lambda)_\ra$, we have a pair of adjoint functors $(-\atimes_X\sfK,\HOM^\ra_X(\sfK,-))$.

  \item[4L] $\pi^{*\ra}\colon\cD(X,\lambda)_\ra\to\cD(X,\lambda')_\ra$, obtained by applying \eqref{1eq:upperstar_adic} to an object $X\in\Chpar{}$ and a morphism $\pi\colon\lambda'\to\lambda$ of $\Rind$. It is symmetric monoidal.

  \item[4R] $\pi_{*\ra}\colon\cD(X,\lambda')_\ra\to\cD(X,\lambda)_\ra$, which is a right adjoint of $\pi^*$.
\end{description}

\begin{proposition}\label{1pr:p34}
Let $f\colon Y\to X$ be a morphism of $\Chpar{}$ and $\lambda$ an object of $\Rind$.
\begin{description}
  \item[(P3)] (Conservativeness) If $f$ is surjective, then $f^{*\ra}\colon\cD(X,\lambda)_\ra\to\cD(Y,\lambda)_\ra$ is conservative.

  \item[(P4)] (Descent) Suppose that $f$ is smooth surjective. Then $(f,\id_\lambda)$ is of universal $\EO{\ra}{\Chpar{}}\otimes{}$-descent. If $X$ belongs to $\Chpar{}_\Box$ and and $\lambda$ belongs to $\Rind_{\ltor}$, then $(f,\id_\lambda)$ is of universal $\EO{\ra}{\Chpar{}}{}{!}$-codescent. See \cite{LZ1}*{3.1.1} for the definition of (co)descent.
\end{description}
\end{proposition}

\begin{proof}
(P3) follows from the construction and the fact that
\[
\varprojlim_{\rN(\Xi)^{op}}\cD(X,\Lambda(\xi))\to\varprojlim_{\rN(\Xi)^{op}}\cD(Y,\Lambda(\xi))
\]
is conservative if each functor $\cD(X,\Lambda(\xi))\to\cD(Y,\Lambda(\xi))$ is, where $\lambda=(\Xi,\Lambda)$. The latter is true as $f$ is surjective.

Now we consider (P4). The universal descent property for $\EO{\ra}{\Chpar{}}\otimes{}$ follows from the construction, the same property in the non-adic case, and (the dual version of) \cite{Lu1}*{4.3.2.9}. The universal codescent property for $\EO{\ra}{\Chpar{}}{}{!}$ follows from the construction, the same property in the non-adic case, and \cite{Lu2}*{4.7.5.19}. Note that condition (c) in \cite{Lu2}*{4.7.5.19} is fulfilled by the Poincar\'{e} duality \cite{LZ1}*{6.2.9}.
\end{proof}

\begin{proposition}[\textbf{(P5)} Smooth Base Change]
Let
\[
\xymatrix{
W \ar[d]_-{q} \ar[r]^-{g} & Z \ar[d]^-{p} \\
Y \ar[r]^-{f} & X   }
\]
be a Cartesian diagram in $\Chpar{}_\Box$ where $p$ is smooth. Then for every object $\lambda$ of $\Rind_{\ltor}$, the following square
\[
\xymatrix{
\cD(W,\lambda)_\ra   & \cD(Z,\lambda)_\ra \ar[l]_-{g^{*\ra}}  \\
\cD(Y,\lambda)_\ra \ar[u]^-{q^{*\ra}} & \cD(X,\lambda)_\ra \ar[l]_-{f^{*\ra}}\ar[u]_-{p^{*\ra}}   }
\]
is right adjointable.
\end{proposition}

\begin{proof}
It follows from the construction, the same property in the non-adic case, and \cite{LZ1}*{4.3.7}.
\end{proof}

Now we consider the usual t-structure on $\cD(X,\lambda)$ for an object $(X,\lambda)\in\Chpar{}\times\rN(\Rind)$. Recall from \cite{Lu2}*{1.4.4.12} that, for a presentable stable $\infty$-category $\cD$, a t-structure\footnote{We use a \emph{cohomological} indexing convention, which is different from \cite{Lu2}*{1.2.1.4}.} is \emph{accessible} if the full subcategory $\cD^{\leq0}$ is presentable. For a scheme $X\in\Schqcs$, the usual t-structure on $\cD(X,\lambda)$ is accessible by \cite{Lu2}*{1.3.5.21}. For a higher Artin stack $X$, the usual t-structure on $\cD(X,\lambda)$ is accessible by construction \cite{LZ1}*{4.3.7} (Part (3) of (P6)).

Suppose $\lambda=(\Xi,\Lambda)$. For $n\in\dZ$, we let $\cD^{\leq n}(X,\lambda)_\ra$ be the full subcategory of $\cD(X,\lambda)_\ra$ spanned by objects $\sfK=(\sfK_\xi)_{\xi\in\Xi}$ with $\sfK_\xi\in\cD^{\leq n}(X,\Lambda(\xi))$. Put
\[
\cD^{\geq n}(X,\lambda)_\ra\coloneqq\cD^{\leq n-1}(X,\lambda)^{\perp}_\ra
\]
as a full subcategory of $\cD(X,\lambda)_\ra$. By \cite{LZ1}*{3.1.4}, we have an equivalence
\[
\cD^{\leq n}(X,\lambda)_\ra\simeq\varprojlim_{\rN(\Xi)^{op}}\cD^{\leq n}(Y,\Lambda(\xi)).
\]
Here, we have used the fact that transition functors, which are (derived) extension of scalars, are left exact. In particular, $\cD^{\leq n}(X,\lambda)_\ra$ is presentable; the inclusion $\cD^{\leq n}(X,\lambda)_\ra\subseteq\cD(X,\lambda)$ preserves all small colimits; and $\cD^{\leq n}(X,\lambda)_\ra$ is closed under extension. By \cite{Lu2}*{1.4.4.11 (1)}, the pair $(\cD^{\leq n}(X,\lambda)_\ra,\cD^{\geq n}(X,\lambda)_\ra)$ define an accessible t-structure, called the \emph{usual t-structure}, on $\cD(X,\lambda)_\ra$. We have truncation functors
\[
\tau_\ra^{\leq n}\colon\cD(X,\lambda)_\ra\to\cD^{\leq n}(X,\lambda)_\ra,\quad
\tau_\ra^{\geq n}\colon\cD(X,\lambda)_\ra\to\cD^{\geq n}(X,\lambda)_\ra
\]
for every $n\in\dZ$. Properties (P6) and (P7) will be studied in the next section, after we reveal a relation between $\cD(X,\lambda)_\ra$ and $\cD(X,\lambda)$.

\subsection{Relation with adic complexes}
\label{1ss:relation}

In this section, we define a natural full subcategory $\cD(X,\lambda)'_\ra$ of $\cD(X,\lambda)$ consisting of \emph{adic complexes} and show that there is a canonical equivalence $\cD(X,\lambda)'_\ra\simeq\cD(X,\lambda)_\ra$ of $\infty$-categories.

Let $X$ be an object of $\Chpar{}$, and $\lambda=(\Xi,\Lambda)$ an object of $\Rind$. For every morphism $\varphi\colon \xi\to\xi'$ in $\Xi$, there is a commutative diagram in $\Rind$ of the form
\[
\xymatrix{
(\Xi,\Lambda) \ar@{=}[d] & (\Xi_{/\xi},\Lambda_{/\xi})
\ar[d]^-{i_{\varphi}} \ar[l]_-{i_{\xi}} \ar[r]^-{p_{\xi}}
& (\{\xi\},\Lambda(\xi)) \ar[d]^-{\tilde{\varphi}} \\
(\Xi,\Lambda) & (\Xi_{/\xi'},\Lambda_{/\xi'}) \ar[l]_-{i_{\xi'}} \ar[r]^-{p_{\xi'}} & (\{\xi'\},\Lambda(\xi')),   }
\]
which induces the following diagram in $\PRL$:
\begin{align}\label{1eq:adic}
\xymatrix{
\cD(X,\lambda) \ar[r]^-{i_{\xi}^*} & \cD(X,\lambda_{/\xi})
& \cD(X,\Lambda(\xi)) \ar[l]_-{p_{\xi}^*} \\
\cD(X,\lambda) \ar[r]^-{i_{\xi'}^*} \ar@{=}[u] & \cD(X,\lambda_{/\xi'})
\ar[u]_-{i_{\varphi}^*}   & \cD(X,\Lambda(\xi')) \ar[l]_-{p_{\xi'}^*} \ar[u]_-{\tilde{\varphi}^*},  }
\end{align}
where $\lambda_{/\xi}\coloneqq(\Xi_{/\xi},\Lambda_{/\xi})$. Let $p_{\xi*}$ (resp.\ $p_{\xi'*}$) be a right adjoint of $p_{\xi}^*$ (resp.\ $p_{\xi'}^*$) and let $\alpha_{\varphi}\colon\tilde{\varphi}^*p_{\xi'*}\to p_{\xi*}i_{\varphi}^*$ be the natural transformation.

\begin{definition}[Adic complex]\label{1de:adic_object}
We say that an element $\sfK\in\cD(X,\lambda)$ is an \emph{adic complex} if the natural morphism
\[
\alpha_{\varphi}(i_{\xi'}^*\sfK)\colon\tilde{\varphi}^*p_{\xi'*}i_{\xi'}^*\sfK\to p_{\xi*}i_{\varphi}^*i_{\xi'}^*\sfK
\]
is an equivalence for every morphism $\varphi\colon \xi\to\xi'$ in $\Xi$. The target of $\alpha_{\varphi}(i_{\xi'}^*\sfK)$ is equivalent to
$p_{\xi*}i_{\xi}^*\sfK$. It is clear that adic complexes are stable under equivalence. Denote by
\[
\cD(X,\lambda)'_\ra\subseteq\cD(X,\lambda)
\]
the full subcategory spanned by adic complexes.
\end{definition}

\begin{lem}\label{1le:evaluation}
Let $f\colon Y\to X$ be a morphism in $\Chpar{}$. If $\sfK$ is an adic complex in $\cD(X,\lambda)$, then $f^*\sfK$ is also an adic complex in $\cD(Y,\lambda)$. If $f$ is surjective, then the converse holds as well.
\end{lem}

\begin{proof}
The first statement follows if we can show that the following diagram
\begin{align}\label{1eq:evaluation}
\xymatrix{
\cD(X,\lambda_{/\xi}) \ar[d]_-{f^*} & \cD(X,\Lambda(\xi)) \ar[l]_-{p_{\xi}^*}\ar[d]^-{f^*} \\
\cD(Y,\lambda_{/\xi})  & \cD(Y,\Lambda(\xi)) \ar[l]_-{p_{\xi}^*} \\
}
\end{align}
is right adjointable. By the construction of $\EO{}{\Chpar{}}{\r{I}}{}$ and \cite{LZ1}*{4.3.7}, we may assume that $f$ is a morphism in $\Schqcs$. Then the following diagram
\[
\xymatrix{
\Mod(X^{\Xi_{/\xi}}_{\et},\Lambda_{/\xi})  \ar[d]_-{f^*} & \Mod(X_{\et},\Lambda(\xi)) \ar[l]_-{p_{\xi}^*}\ar[d]^-{f^*}  \\
\Mod(Y^{\Xi_{/\xi}}_{\et},\Lambda_{/\xi})  & \Mod(Y_{\et},\Lambda(\xi)) \ar[l]_-{p_{\xi}^*}
}
\]
has a right adjoint, which is
\[
\xymatrix{
\Mod(X^{\Xi_{/\xi}}_{\et},\Lambda_{/\xi})  \ar[r]^-{s_{\xi}^*}\ar[d]_-{f^*} & \Mod(X_{\et},\Lambda(\xi))\ar[d]^-{f^*} \\
\Mod(Y^{\Xi_{/\xi}}_{\et},\Lambda_{/\xi})  \ar[r]^-{s_{\xi}^*} & \Mod(Y_{\et},\Lambda(\xi))
}
\]
where $s_\xi\colon\{\xi\}\to\Xi_{/\xi}$ is the inclusion map. Thus, \eqref{1eq:evaluation} is right adjointable.

The second statement follows from the first one and property (P3) for $\EO{}{\Chpar{}}{\r{I}}{}$.
\end{proof}

In general, if $\lambda=(\Xi,\Lambda)$ is an object of $\Rind$ and $\xi\in\Xi$, then we have successive inclusions
\[
e_{\xi}\colon(\{\xi\},\Lambda(\xi))\xrightarrow{s_{\xi}}(\Xi_{/\xi},\Lambda_{/\xi})\xrightarrow{i_{\xi}}(\Xi,\Lambda)
\]
which induce the \emph{evaluation functor (at $\xi$)}
\[
e_{\xi}^*\colon\cD(X,\lambda)\to\cD(X,\Lambda(\xi))
\]
for a higher Artin stack $X$. By Lemma \ref{1le:evaluation}, $e_{\xi}^*$ and $p_{\xi*}\circ i_{\xi}^*$ are equivalent. For brevity, we sometimes also write $\sfK_\xi$ for $e_{\xi}^*\sfK$ for an object $\sfK\in\cD(X,\lambda)$.

The functor
\[
\prod_{\xi\in\Xi}e_{\xi}^*\colon\cD(X,\lambda)\to\prod_{\xi\in\Xi}\cD(X,\Lambda(\xi))
\]
is conservative. This is obvious when $X$ is in $\Schqcs$. The general case follows, because simplicial limits of conservative functors are
conservative.

\begin{lem}\label{1le:adic}
Suppose that $\Xi$ admits a final object $\xi$. Then the image of the natural map
$p_{\xi}^*\colon\cD(X,\Lambda(\xi))\to\cD(X,\Lambda)$ is contained in $\cD(X,\Lambda)'_\ra$. Moreover, the induced map
$p_{\xi}^*\colon\cD(X,\Lambda(\xi))\to\cD(X,\Lambda)'_\ra$ is an equivalence of $\infty$-categories.
\end{lem}

\begin{proof}
The first assertion follows from Definition \ref{1de:adic_object} and the natural isomorphism between $p_{\xi'*}$ and $s_{\xi'}^*$ as in \eqref{1eq:adic} for an arbitrary object $\xi'$ of $\Xi$.

For the second assertion, we only need to show that for every adic complex $\sfK\in\cD(X,\lambda)'_\ra$, the adjunction map $p_{\xi}^*p_{\xi*}\sfK\to\sfK$ is an equivalence. Since the functor $\prod_{\xi'\in\Xi}e_{\xi'}^*$ is conservative, this is equivalent to showing that the map $\beta\colon e_{\xi'}^*p_{\xi}^*p_{\xi*}\sfK\to e_{\xi'}^*\sfK$ is an equivalence for every object $\xi'\in\Xi$. Let $\varphi$ be the map $\xi'\to\xi$. Since $\sfK$ is an adic complex, the composite
\[
{\tilde\varphi}^*p_{\xi*}\sfK\xrightarrow{\alpha}p_{\xi'*}p_{\xi'}^*{\tilde\varphi}^*p_{\xi*}\sfK\simeq p_{\xi'*}i_\varphi^*p_\xi^*p_{\xi*}\sfK\xrightarrow{\beta}p_{\xi'*}i_\varphi^*\sfK
\]
is an equivalence, where we adopt the notation in \eqref{1eq:adic}. Moreover, we have shown that $\alpha$ is an equivalence as $p_{\xi'*}\simeq s_{\xi}^*$. Therefore, $\beta$ is an equivalence.
\end{proof}

\begin{proposition}\label{1pr:adic_inclusion}
The inclusion $\cD(X,\lambda)'_\ra\to\cD(X,\lambda)$ is a morphism in $\PRL$.
\end{proposition}

\begin{proof}
By definition, the inclusion $\cD(X,\lambda)'_\ra\subseteq\cD(X,\lambda)$ fits into the following diagram
\[
\xymatrix{
\cD(X,\lambda)'_\ra \ar[rr]\ar[d] && \prod_{\xi\in\Xi}\cD(X,\lambda_{/\xi})'_\ra \ar[d] \\
\cD(X,\lambda) \ar[rr]^-{\prod_{\xi\in\Xi}i_{\xi}^*} &&
\prod_{\xi\in\Xi}\cD(X,\lambda_{/\xi}), }
\]
which is a pullback diagram in $\Cat$ by Lemma \ref{1le:pullback} below. By Lemma \ref{1le:adic}, $p_{\xi*}$ commutes with small colimits, $\cD(X,\lambda_{/\xi})'_\ra$ is presentable and the inclusion into $\cD(X,\lambda_{/\xi})$ preserves small colimits. Therefore, the right vertical arrow is a morphism in $\PRL$ as $\Xi$ is small. Moreover, the functor $\prod_{\xi\in\Xi}i_{\xi}^*$ preserves small colimits since each $i_{\xi}^*$ does and $\Xi$ is small. Therefore, the inclusion $\cD(X,\lambda)'_\ra\to\cD(X,\lambda)$ is a morphism in $\PRL$, because the inclusion $\PRL\subseteq\Cat$ preserves small limits.
\end{proof}

\begin{lem}\label{1le:pullback}
Let $\cD$ be a full subcategory of an $\infty$-category $\cC$ and $f\colon\cD\to\cC$ be the inclusion. Then the pullback of $f$ in the category $\Sset$ by any functor $g\colon \cC'\to \cC$ with source in $\Cat$ is a pullback in $\Cat$.
\end{lem}

\begin{proof}
This follows immediately from \cite{LZ1}*{3.1.4} applied to the pullback of $\id_\cC$ by $g$.
\end{proof}

Next, we will construct a natural functor
\begin{align}\label{1eq:limit}
\cD(X,\lambda)'_\ra\to\cD(X,\lambda)_\ra=\varprojlim_{\rN(\Xi)^{op}}\cD(X,\Lambda(\xi))
\end{align}
and show that it is an equivalence. It is clear that for every object $\xi\in\Xi$, $i_\xi^*$ sends $\cD(X,\lambda)'_\ra$ to $\cD(X,\lambda_{/\xi})'_\ra$; and for every morphism $\varphi\colon\xi\to\xi'$ in $\Xi$, $i_\varphi^*$ sends $\cD(X,\lambda_{/\xi'})'_\ra$ to $\cD(X,\lambda_{/\xi})'_\ra$. Therefore, the diagram \eqref{1eq:adic} induces a functor
\[
\cD(X,\lambda)'_\ra\to\varprojlim_{\rN(\Xi)^{op}}\cD(X,\lambda_{/\xi})'_\ra.
\]
By Lemma \ref{1le:adic}, the right-hand side is equivalent to
\[
\varprojlim_{\rN(\Xi)^{op}}\cD(X,\Lambda(\xi))=\cD(X,\lambda)_\ra.
\]
Thus, we obtain the desired functor \eqref{1eq:limit}.

\begin{theorem}\label{1th:limit}
For every objects $X$ of $\Chpar{}$ and $\lambda=(\Xi,\Lambda)$ of $\Rind$, the functor
\[
\cD(X,\lambda)'_\ra\to\cD(X,\lambda)_\ra=\varprojlim_{\rN(\Xi)^{op}}\cD(X,\Lambda(\xi))
\]
\eqref{1eq:adic} is an equivalence of $\infty$-categories.
\end{theorem}

We need some preparation before the proof. Let $X$ be an object of $\Schqcs$. For simplicity, we will write $X$ for $X_{\et}$ as well. By definition, $\cD(X,\lambda)'_\ra$ is a full subcategory of $\cD(X,\lambda)=\cD(X^\Xi,\Lambda)=\cD(\Mod(X^\Xi,\Lambda))$. For every object $\xi$ of $\Xi$, we have an \emph{evaluation functor}
\[
e_\xi^*\colon\Mod(X^\Xi,\Lambda)\to\Mod(X,\Lambda(\xi))
\]
at $\xi$ on the level of Abelian categories. It is exact and admits a (right exact) left adjoint functor
\begin{align}\label{1eq:limit1}
e_{\xi!}\colon\Mod(X,\Lambda(\xi))\to\Mod(X^\Xi,\Lambda).
\end{align}
Moreover, we define a truncation functor
\begin{align}\label{1eq:limit2}
t_{\leq\xi}\colon\Mod(X^\Xi,\Lambda)\to\Mod(X^\Xi,\Lambda)
\end{align}
such that for a $\Lambda$-module $F_{\bullet}\in\Mod(X^\Xi,\Lambda)$, we have
\[
(t_{\leq\xi}F_{\bullet})_{\xi'}
=\begin{cases}
F_{\xi'} & \text{if }\xi'\leq\xi,\\
0 & \text{otherwise.}
\end{cases}
\]
It is exact and admits a right adjoint.

\begin{proof}
By Lemma \ref{1le:evaluation}, \cite{LZ1}*{3.1.4}, property (P4) for $\EO{}{\Chpar{}}{\r{I}}{}$ and Proposition \ref{1pr:p34}, we may assume $X\in\Schqcs$.

We first study the functor
\[
\alpha\colon\cD(X,\lambda)'_\ra\to\varprojlim_{\rN(\Xi)^{op}}\cD(X,\Lambda(\xi))
\]
from the point of view of coCartesian fibrations. First, we have a functor $\Delta^1\times\rN(\Xi)\to\Cat$ sending
$\Delta^1\times(\varphi\colon\xi\to\xi')$ to the square
\[
\xymatrix{
\cD(X^{\Xi_{/\xi}},\Lambda_{/\xi}) \ar[r]^-{p_{\xi*}} \ar[d]_-{i_{\varphi*}} & \cD(X,\Lambda(\xi)) \ar[d]^-{\tilde{\varphi}_*} \\
\cD(X^{\Xi_{/\xi'}},\Lambda_{/\xi'}) \ar[r]^-{p_{\xi'*}}  & \cD(X,\Lambda(\xi')). }
\]
This corresponds to a projectively fibrant simplicial functor $\cF\colon\fC[\rN(D)]\to\Mset$, where $D=[1]\times\Xi$. Let $\phi_D\colon\fC[\rN(D)]\to D$ be the canonical equivalence of simplicial categories and put
\[
\cF'=(\Fibr^D \circ \St^+_{\phi_D^{op}}\circ\Un^+_{\rN(D)^{op}})\cF\colon D\to \Mset.
\]
We write $\cF'$ in the form $\cF'\colon[1]\to(\Mset)^\Xi$. Applying the marked unstraightening functor $\Un^+_{\phi}$ for the weak equivalence of simplicial categories $\phi\colon\fC[\rN(\Xi)^{op}]\to\Xi^{op}$, we obtain a morphism $\tilde\alpha\colon F_1\to F_2$ of Cartesian fibrations in the category $(\Mset)_{/\rN(\Xi)^{op}}$. Moreover, by \cite{Lu1}*{5.2.2.5}, both $F_1$ and $F_2$ are \emph{coCartesian} fibrations as well, but $\tilde\alpha$ does \emph{not} send coCartesian edges to coCartesian ones in general. By a similar argument, we have a map
\[
\cD(X^\Xi,\Lambda)\to\Map^{\r{coCart}}_{\rN(\Xi)^{op}}(\rN(\Xi)^{op},F_1)\coloneqq\Map^\flat_{\rN(\Xi)^{op}}((\rN(\Xi)^{op})^\sharp,(F_1,\cE)),
\]
where $\cE$ is the set of coCartesian edges of $F_1$. Composing with the obvious inclusion $\Map^{\r{coCart}}_{\rN(\Xi)^{op}}(\rN(\Xi)^{op},F_1)\subseteq\Map_{\rN(\Xi)^{op}}(\rN(\Xi)^{op},F_1)$ and $\Map_{\rN(\Xi)^{op}}(\rN(\Xi)^{op},\tilde\alpha)$, we obtain a map
\[
\alpha'\colon\cD(X^\Xi,\Lambda)\to\Map_{\rN(\Xi)^{op}}(\rN(\Xi)^{op},F_2).
\]
We have the equivalence
\[
\Map^{\r{coCart}}_{\rN(\Xi)^{op}}(\rN(\Xi)^{op},F_2)\simeq\lim_{\Xi^{op}}\cD(X,\Lambda(\xi))
\]
by \cite{Lu1}*{3.3.3.2}, and the following pullback diagram
\[
\xymatrix{
\cD(X^\Xi,\Lambda)'_\ra \ar[r]^-{\alpha} \ar[d] & \Map^{\r{coCart}}_{\rN(\Xi)^{op}}(\rN(\Xi)^{op},F_2) \ar[d] \\
\cD(X^\Xi,\Lambda) \ar[r]^-{\alpha'} & \Map_{\rN(\Xi)^{op}}(\rN(\Xi)^{op},F_2) }
\]
by the definition of adic complexes, where vertical arrows are inclusions. We also note that $\alpha'$ commutes with small colimits by \cite{Lu1}*{5.1.2.2}. Thus, the goal is to show that $\alpha$ is an equivalence.

To construct an inverse $\beta$ of $\alpha$, we use $\del_{/\Xi}$: the category of simplices of $\Xi$. Then all $n$-cells
of $\rN(\del_{/\Xi})$ are degenerate for $n\geq2$. Define a functor
\[
\beta'\colon\rN(\del_{/\Xi}^{op})\to\Fun(\Map_{\rN(\Xi)^{op}}(\rN(\Xi)^{op},F_2),\cD(X^\Xi,\Lambda))
\]
sending a typical subcategory $\xi\to(\xi\to\xi')\leftarrow\xi'$ of $\del_{/\Xi}$ to
\[
\xymatrix{
\rL e_{\xi!}\circ \epsilon_\xi & t_{\leq\xi}\circ\rL e_{\xi'!}\circ\epsilon_{\xi'} \ar[l]\ar[r]
&  \rL e_{\xi'!}\circ \epsilon_{\xi'}, }
\]
where $\epsilon_\xi\colon\Map_{\rN(\Xi)^{op}}(\rN(\Xi)^{op},F_2)\to\cD(X,\Lambda(\xi))$ is the restriction functor to the fiber at $\xi$. The functor $\Fun(\alpha',\cD(X^\Xi,\Lambda))\circ\beta'$ extends to a functor
$\rN(\del_{/\Xi}^{op})^{\triangleright}\to\Fun(\cD(X^\Xi,\Lambda),\cD(X^\Xi,\Lambda))$ carrying $(\xi\to(\xi\to\xi')\leftarrow\xi')^\triangleleft$ to
\[
\xymatrix{
\rL e_{\xi!}\circ \epsilon_{\xi}\circ \alpha' \ar[rd] & t_{\leq\xi}\circ\rL e_{\xi'!}\circ\epsilon_{\xi'}\circ \alpha'
\ar[l]\ar[r]\ar[d] &  \rL e_{\xi'!}\circ \epsilon_{\xi'}\circ\alpha'\ar[ld] \\ & \r{id} }
\]
which induces a natural transformation
\[
(\colim\beta')\circ\alpha'\simeq\colim(\Fun(\alpha',\cD(X^\Xi,\Lambda))\circ\beta')\to\r{id}.
\]
Now we put
\[
\beta\coloneqq\colim\beta'\res\Map_{\rN(\Xi)^{op}}^{\r{coCart}}(\rN(\Xi)^{op},F_2).
\]
It is easy to check  that $\beta$ takes values in $\cD(X^\Xi,\Lambda)_\ra$.

We show that the induced natural transformation $\beta\circ\alpha\to\r{id}$ is an equivalence. Pick up an object $\sfK$ of $\cD(X^\Xi,\Lambda)'_\ra$. We need to show that the diagram
\[
\beta^{\triangleright}_{\sfK}\colon\rN(\del_{/\Xi}^{op})^{\triangleright}\to\cD(X^{\Xi},\Lambda),
\]
depicted as
\[
\xymatrix{
\rL e_{\xi!}\sfK_\xi \ar[dr] & t_{\leq\xi}\rL e_{\xi'!} \sfK_{\xi'} \ar[l]\ar[r]\ar[d] &
\rL e_{\xi'!} \sfK_{\xi'}\ar[dl] \\ & \sfK }
\]
is a colimit diagram. We only need to check this after applying $e_{\xi_0}^*$ for every $\xi_0\in\Xi$, since $e_{\xi_0}^*$ commutes with colimits. The composite functor $e_{\xi_0}^*\circ\beta^{\triangleright}_{\sfK}$ has value (equivalent to) $\sfK_{\xi_0}$ (resp.\ $0$) on the cone point, vertices $\{\xi\}$ and $(\xi\to\xi')$ of $\del_{/\Xi}$ for $\xi\geq \xi_0$ (resp.\ otherwise), with all morphisms being either identities on $\sfK_{\xi_0}$ or $0$, or the zero morphism $0\to\sfK_{\xi_0}$. It is clear that $e_{\xi_0}^*\circ\beta^{\triangleright}_{\sfK}$ induces an equivalence $\colim(e_{\xi_0}^*\circ\beta^{\triangleright}_{\sfK}\res\rN(\del_{/\Xi}^{op}))\simeq\sfK_{\xi_0}$ in $\cD(X,\Lambda(\xi_0))$.

For the other direction, that is, a natural equivalence $\alpha\circ\beta\to\r{id}$, we note that the functor $\Fun_{\rN(\Xi)^{op}}(\rN(\Xi)^{op},F_2),\alpha')\circ\beta'$ also extends to a functor
\[
\rN(\del_{/\Xi}^{op})^{\triangleright}\to\Fun(\Map_{\rN(\Xi)^{op}}(\rN(\Xi)^{op},F_2),\Map_{\rN(\Xi)^{op}}(\rN(\Xi)^{op},F_2))
\]
carrying $(\xi\to(\xi\to\xi')\leftarrow\xi')^\triangleleft$ to
\[
\xymatrix{
\alpha'\circ\rL e_{\xi!}\circ \epsilon_\xi \ar[rd] & \alpha'\circ t_{\leq\xi}\circ\rL e_{\xi'!}\circ\epsilon_{\xi'}
\ar[l]\ar[r]\ar[d] & \alpha'\circ \rL e_{\xi'!}\circ \epsilon_{\xi'}\ar[ld] \\ & \r{id} }
\]
which induces a natural transformation
\[
\alpha'\circ(\colim\beta')\simeq\colim(\Fun_{\rN(\Xi)^{op}}(\rN(\Xi)^{op},F_2),\alpha')\circ\beta')\to\id,
\]
where the equivalence of two functors is due to the fact that $\alpha'$ commutes with colimits. Restricting to $\Map^{\r{coCart}}_{\rN(\Xi)^{op}}(\rN(\Xi)^{op},F_2)$, one obtains a natural transformation $\alpha\circ\beta\to\id$ which is an equivalence by an argument similar to the previous one. Therefore, $\alpha$ is an equivalence and the proposition follows.
\end{proof}

\begin{remark}\label{1re:limit}
By Theorem \ref{1th:limit}, in what follows, we will identify $\cD(X,\lambda)'_\ra$ with $\cD(X,\lambda)_\ra$. In particular, we will regard $\cD(X,\lambda)_\ra$ as a full subcategory of $\cD(X,\lambda)$.
\begin{enumerate}
  \item By Proposition \ref{1pr:adic_inclusion}, the inclusion functor $\cD(X,\lambda)_\ra\to\cD(X,\lambda)$ admits a right adjoint, which we denote by $\fR_X\colon\cD(X,\lambda)\to\cD(X,\lambda)_\ra$. It is a colocalization functor \cite{Lu1}*{\Sec5.2.7}.

  \item Let $f\colon Y\to X$ be a morphism of $\Chpar{}$. The functor $f^*\colon\cD(X,\lambda)\to\cD(Y,\lambda)$ preserves adic complexes, and the induced functor $f^*\colon\cD(X,\lambda)_\ra\to\cD(Y,\lambda)_\ra$ coincides with $f^{*\ra}$ up to equivalence. The functor $f_{*\ra}$ is equivalent to the composition of the inclusion $\cD(Y,\lambda)_\ra\to\cD(Y,\lambda)$, $f_*\colon\cD(Y,\lambda)\to\cD(X,\lambda)$ and the functor $\fR_X$.

  \item Let $f\colon Y\to X$ be a locally of finite type morphism of $\Chpar{}_\Box$, and suppose $\lambda\in\Rind_{\ltor}$. The functor $f_!\colon\cD(Y,\lambda)\to\cD(X,\lambda)$ preserves adic complexes, and the induced functor $f_!\colon\cD(Y,\lambda)_\ra\to\cD(X,\lambda)_\ra$ coincides with $f_{!\ra}$ up to equivalence. The functor $f^{!\ra}$ is equivalent to the composition of the inclusion $\cD(X,\lambda)_\ra\to\cD(X,\lambda)$, $f^!\colon\cD(X,\lambda)\to\cD(Y,\lambda)$ and the functor $\fR_Y$.

  \item The functor $-\otimes_X-\colon\cD(X,\lambda)\times\cD(X,\lambda)\to\cD(X,\lambda)$ preserves adic complexes, and the induced functor
      $-\otimes_X-\colon\cD(X,\lambda)_\ra\times\cD(X,\lambda)_\ra\to\cD(X,\lambda)_\ra$ coincides with $-\atimes_X-$ up to equivalence. The functor $\HOM^\ra_X$ is equivalent to the composition of the inclusion $\cD(X,\lambda)_\ra^{op}\times\cD(X,\lambda)_\ra\to\cD(X,\lambda)^{op}\times\cD(X,\lambda)$, $\HOM_X$ and $\fR_X$.

  \item Let $\pi\colon\lambda'\to\lambda$ be a morphism of $\Rind$. The functor $\pi^*\colon\cD(X,\lambda)\to\cD(X,\lambda')$ preserves adic complexes, and the induced functor $\pi^*\colon\cD(X,\lambda)_\ra\to\cD(X,\lambda')_\ra$ coincides with $\pi^{*\ra}$ up to equivalence. The functor $\pi_{*\ra}$ is equivalent to the composition of the inclusion $\cD(X,\lambda')_\ra\to\cD(X,\lambda')$, $\pi_*$ and $\fR_X$.

  \item We have $\cD^{\leq n}(X,\lambda)_\ra=\cD^{\leq n}(X,\lambda)\cap\cD(X,\lambda)_\ra$ for every $n\in\dZ$.
\end{enumerate}
\end{remark}

\begin{remark}[P6]\label{1re:p6}
We have the following remarks concerning the above t-structure.
\begin{enumerate}
  \item The usual t-structure on $\cD(X,\lambda)_\ra$ is accessible. Moreover, the intersection $\bigcap_{n}\cD^{\leq-n}(X,\lambda)$ consists of zero objects\footnote{This is call \emph{weakly left complete} in Definition \ref{3de:complete}. Unlike the non-adic case, we do not know in general whether $\cD(X,\lambda)_\ra$ is right complete or even weakly right complete.}.

  \item The constant sheaf $\lambda_X\in\cD(X,\lambda)$ is an adic complex and belongs to the heart
      \[
      \cD^\heartsuit(X,\lambda)_\ra\coloneqq\cD^{\leq 0}(X,\lambda)_\ra\cap\cD^{\geq 0}(X,\lambda)_\ra
      \]
      by Remark \ref{1re:limit} (6).

  \item The functor $-\langle d\rangle\colon\cD(X,\lambda)\to\cD(X,\lambda)$ from \cite{LZ1}*{Input II} restricts to a functor
      \[
      -\langle d\rangle\colon\cD(X,\lambda)_\ra\to\cD(X,\lambda)_\ra
      \]
      for every integer $d$.

  \item The functors $f^{*\ra}$, $-\atimes_X-$, $\pi^{*\ra}$ are all left t-exact (that is, preserve $\cD^{\leq n}$). The functors $f_{*\ra}$, $\HOM^\ra_X$, $\pi_{*\ra}$ are all right t-exact (that is, preserve $\cD^{\geq n}$).

  \item It follows from \cite{LZ1}*{6.2.15} that $f_{!\ra}[2d]$ is left t-exact, hence $f^{!\ra}[-2d]$ is right t-exact.
\end{enumerate}
\end{remark}

\begin{theorem}[\textbf{(P7)} Poincar\'{e} duality]\label{1th:poincare}
Let $f\colon Y\to X$ be a morphism of $\Chpar{}_\Box$ that is flat and locally of finite presentation. Let $\lambda$ be an object of $\Rind_{\ltor}$. Then
\begin{enumerate}
  \item There is a functorial (in the sense of \cite{LZ1}*{4.1.6}) trace map
      \[
      \Tr_f\colon\tau_\ra^{\geq0}f_{!\ra}\lambda_Y\langle d\rangle\to\lambda_X
      \]
      in the heart $\cD^\heartsuit(X,\lambda)_\ra$ for every integer $d\geq\dim^+(f)$.

  \item If $f$ is moreover smooth, the induced natural transformation
      \[
      u_f\colon f_{!\ra}\circ f^{*\ra}\langle\dim f\rangle\to\id_X
      \]
      is a counit transformation, so that the induced map
      \[
      f^{*\ra}\langle\dim f\rangle\to f^{!\ra}\colon \cD(X,\lambda)_\ra\to\cD(Y,\lambda)_\ra
      \]
      is a natural equivalence of functors.
\end{enumerate}
\end{theorem}

\begin{proof}
For (1), we note that $f_{!\ra}\lambda_Y\langle d\rangle=f_!\lambda_Y\langle d\rangle\in\cD^{\leq 0}(X,\lambda)$ by part (1) of (P7) in \cite{LZ1}*{\Sec4.1}. Thus, by definition, $f_{!\ra}\lambda_Y\langle d\rangle\in\cD^{\leq 0}(X,\lambda)_\ra$. Note that we have a trace map $f_!\lambda_Y\to\lambda_X$ in the non-adic case. Applying $\tau^{\geq0}$ (for the adic one), we obtain the desired trace map
\[
\Tr_f\colon\tau_\ra^{\geq0}f_{!\ra}\lambda_Y\langle d\rangle=\tau_\ra^{\geq0}f_!\lambda_Y\langle d\rangle\to\tau_\ra^{\geq0}\lambda_X=\lambda_X
\]
which is a map in $\cD^\heartsuit(X,\lambda)_\ra$. The functoriality is automatic.

For (2), by the Poincar\'{e} duality $f^*\langle\dim f\rangle\simeq f^!$ in the non-adic case, $f^!$ preserves adic complexes hence $f^{!\ra}=f^!\res\cD(X,\lambda)_\ra$. Then it follows from the corresponding argument in the non-adic case.
\end{proof}

The following theorem, which contains some other properties of enhanced adic operations, holds with the same proof in the non-adic case.

\begin{theorem}\label{1th:properties}
We have
\begin{enumerate}
  \item The \emph{K\"{u}nneth Formula} \cite{LZ1}*{6.2.1} holds in the adic case.

  \item The \emph{Base Change} \cite{LZ1}*{6.2.2} holds in the adic case.

  \item The \emph{Projection Formula} \cite{LZ1}*{6.2.3} holds in the adic case.

  \item \cite{LZ1}*{6.2.4, 6.2.5, 6.2.8, 6.2.11, 6.2.13, 6.2.14, 6.2.15} hold in the adic case.
\end{enumerate}
\end{theorem}

\subsection{Adic dualizing complexes}
\label{1ss:adic_dualizing}

In this section, we construct adic dualizing complexes and study the biduality properties in the adic case.

Let $X$ be an object of $\Chpar{}$, and $\lambda=(\Xi,\Lambda)$ an object of $\Rind$. Let $\Omega$ be an object of $\cD(X,\lambda)$ (resp.\
$\cD(X,\lambda)_\ra$). By adjunction of the pair of functors $-\otimes\sfK\coloneqq-\otimes_X\sfK$ and $\HOM(\sfK,-)\coloneqq\HOM_X(\sfK,-)$ (resp.\ $-\atimes\sfK\coloneqq-\atimes_X\sfK$ and $\HOM^\ra(\sfK,-)\coloneqq\HOM^\ra_X(\sfK,-)$), we have a natural transformation
\begin{align}\label{1eq:biduality1}
\delta_\Omega&\colon\id\to\rh\HOM(\rh\HOM(-,\Omega),\Omega)\\\label{1eq:biduality2}
\text{resp.\ }\delta^\ra_\Omega&\colon\id\to\rh\HOM^\ra(\rh\HOM^\ra(-,\Omega),\Omega)
\end{align}
between endofunctors of $\rh\cD(X,\lambda)$ (resp.\ $\rh\cD(X,\lambda)_\ra$), which is called the \emph{biduality transformation}\footnote{In fact, $\delta_\Omega$ can be enhanced to a natural transformation $\tilde\delta_\Omega\colon\id\to\HOM(\HOM(-,\Omega),\Omega)$ between endofunctors of $\cD(X,\lambda)$, that is, $\rh\tilde\delta_\Omega=\delta_{\Omega}$; and similar for the adic case. We omit the details here since we do not need such enhancement in what follows.}.

In the remaining of this section, we fix a $\Box$-coprime base scheme $\dS$ that is a disjoint union of \emph{excellent} schemes\footnote{A scheme is \emph{excellent} if it is quasi-compact and admits a Zariski open cover by spectra of excellent rings \cite{EGAIV}*{7.8.2}.}, \emph{endowed with a global dimension function}. Let $\Rind_{\Box\text{-}\r{dual}}$ be the full subcategory of $\Rind_{\ltor}$ spanned by ringed diagrams $\Lambda\colon\Xi^{op}\to\Ring$ such that $\Lambda(\xi)$ is a ($\Box$-torsion) Gorenstein ring of dimension 0 for every object $\xi$ of $\Xi$.

\begin{definition}[Potential dualizing complex]\label{1de:dualizing_complex}
Let $\lambda=(\Xi,\Lambda)$ be an object of $\Rind_{\Box\text{-}\r{dual}}$. For an object $f\colon X\to\dS$ of $\Chpars$ with $X$ in $\Schqcs$, we say that an object $\Omega\in\cD(X,\lambda)$ is a \emph{pinned/potential dualizing complex} (on $X$) if
\begin{enumerate}
  \item $\Omega$ is an adic complex, and

  \item for every object $\xi$ of $\Xi$, $\Omega_\xi=e_\xi^*\Omega\in\cD(X,\Lambda(\xi))$ is a pinned/potential dualizing complex.
\end{enumerate}
For a general object $f\colon X\to\dS$ of $\Chpars$, we say that an object $\Omega\in\cD(X,\lambda)$ is a \emph{pinned/potential dualizing complex} if for every atlas $u\colon X_0\to X$ with $X_0$ in $\Schqcs$, $u^!\Omega$ is a pinned/potential dualizing complex on $X_0$.
\end{definition}

\begin{proposition}\label{1pr:dualizing_complex}
Let $f\colon X\to\dS$ be an object of $\Chpars$ and $\lambda$ an object of $\Rind_{\Box\text{-}\r{dual}}$. The full subcategory of $\cD(X,\lambda)$ spanned by all pinned/potential dualizing complexes is equivalent to the nerve of an ordinary category
consisting of only one object $\Omega$ with
\[
\Hom(\Omega,\Omega)=\(\lim_{\xi\in\Xi}\Lambda(\xi)\)^{\pi_0(X)}.
\]
Moreover, pinned/potential dualizing complexes are constructible and compatible under extension of scalars.
\end{proposition}

In the proof, we will use the following observation which is essentially \cite{Lu1}*{A.3.2.27}. Let $\cC\colon K^{\triangleleft}\to\Cat$ be a functor that is a limit diagram. Let $X,Y$ be two objects in the limit $\infty$-category $\cC_{-\infty}$ and write $X_k,Y_k$ the natural images in $\cC_k$ for every vertex $k$ of $K$. Then $\Map_{\cC_{-\infty}}(X,Y)$ is naturally the homotopy limit (in the $\infty$-category $\cH$ of spaces) of a diagram $K\to\cH$ sending $k$ to $\Map_{\cC_k}(X_k,Y_k)$.

\begin{proof}
We first consider the case where $\Xi=*$ is a singleton.

In this case, if $X$ is in $\Schqcs$, then the proposition is proved in \cite{Ill} (see \cite{LZ1}*{6.5.3}). We also note that if
$\Omega_\dS$ is a pinned dualizing complex on $\dS$, then $f^!\Omega_\dS$ is a pinned dualizing complex on $X$. We prove by induction on $k$ that for an object $f\colon X\to\dS$ of $\Chpars$ with $X$ in $\Chpar{k}$,
\begin{enumerate}
  \item For any two pinned dualizing complexes $\Omega$ and $\Omega'$, $\Map_{\cD(X,\Lambda)}(\Omega,\Omega')$ is discrete\footnote{More precisely, it means that $\Map_{\cD(X,\Lambda)}(\Omega,\Omega')$ is equivalent to a discrete set in $\cH$.};

  \item There is a unique distinguished equivalence $o\colon\Omega\to\Omega'$ such that for every atlas $u\colon X_0\to X$ with $X_0$ in $\Schqcs$, $u^!o$ is the one preserving pinning.
\end{enumerate}
It is clear that once the equivalence $o$ in (2) exists, it is compatible under $f^!$ for every smooth morphism $f$. Choose an atlas $u\colon Y\to X$ (with $Y$ in $\Chpar{(k-1)}$). Since $u$ is of universal $\EO{}{\Chpar{}_\Box}{!}{}$-descent, both (1) and (2) follow from the induction hypothesis, the above observation, and the fact that limit of $k$-truncated spaces is $k$-truncated (which follows from \cite{Lu1}*{5.5.6.5}).

Then we show that $\Map_{\cD(X,\Lambda)}(\Omega,\Omega)\simeq\pi_0\Map_{\cD(X,\Lambda)}(\Omega,\Omega)$ is isomorphic to $\Lambda^{\pi_0(X)}$. Without loss of generality, we assume that $X$ is connected. Choose an atlas $u=\coprod_Iu_i\colon \coprod_IY_i\to X$ with $Y_i$ in $\Schqcs$ that is connected. We have the following commutative diagram
\[
\xymatrix{
\Lambda \ar[r]^-{\alpha} \ar@{=}[d] & \pi_0\Map_{\cD(X,\Lambda)}(\Omega,\Omega) \ar[d]^-{\beta} \\
\Lambda \ar[r] &
\bigoplus_I\pi_0\Map_{\cD(Y_i,\Lambda)}(u_i^!\Omega,u_i^!\Omega). }
\]
Since $u^!$ is conservative, we know that $\beta$ is injective. Since $\Lambda\to\pi_0\Map_{\cD(Y_i,\Lambda)}(u_i^!\Omega,u_i^!\Omega)$ is an isomorphism for every $i\in I$, we know that $\alpha$ is injective. If we write elements of $\bigoplus_I\pi_0\Map_{\cD(Y_i,\Lambda)}(u_i^!\Omega,u_i^!\Omega)$ in the coordinate form $(\dots,\lambda_i,\dots)$ with respect to the basis consisting of distinguished equivalences, then the image of $u^!$ must belong to the diagonal since $X$ is connected. Therefore, $\alpha$ is an isomorphism. The fact that pinned dualizing complexes are constructible and compatible under extension of scalars follows from the case of schemes.

We then consider the case of general coefficient $\lambda=(\Xi,\Lambda)$. We start by constructing a pinned dualizing complex $\Omega_{\dS,\lambda}$ on the base scheme $\dS$. Recall that $\del_{/\Xi}$ is the category of simplices of $\Xi$, whose $n$-simplices are degenerate for $n\geq2$. For every object $\xi$ of $\Xi$, denote by $\Omega_{\dS,\xi}$ the pinned dualizing complex in $\cD(\dS,\Lambda(\xi))$. Recall the functors $e_{\xi!}$ \eqref{1eq:limit1} and $t_{\leq\xi}$ \eqref{1eq:limit2}. Define a functor $\delta\colon\rN(\del_{/\Xi})\to\cD(\dS,\lambda)$ sending a typical subcategory $\xi\leftarrow(\xi\to\xi')\rightarrow\xi'$ of $\del_{/\Xi}$ to
\[
\xymatrix{
\rL e_{\xi!}\Omega_{\dS,\xi} & \rL e_{\xi!}(\Omega_{\dS,\xi'}\overset{\rL}{\otimes}_{\Lambda(\xi')}\Lambda(\xi))
\simeq t_{\leq\xi}\rL e_{\xi'!}\Omega_{\dS,\xi'} \ar[l]\ar[r] &  \rL e_{\xi'!}\Omega_{\dS,\xi'} }
\]
in which the left arrow is given by the distinguished equivalence $\Omega_{\dS,\xi'}\overset{\rL}{\otimes}_{\Lambda(\xi')}\Lambda(\xi)\xrightarrow{\sim}\Omega_{\dS,\xi}$. It is easy to see that
$\Omega_{\dS,\lambda}\coloneqq\lim\delta$, viewed as an element in $\cD(\dS,\lambda)$, satisfies the two requirements in Definition
\ref{1de:dualizing_complex}, hence is a pinned dualizing complex. For an object $f\colon X\to\dS$ of $\Chpars$, put $\Omega_{f,\lambda}=f^!\Omega_{\dS,\lambda}$. Then it is a pinned dualizing complex on $X$. The rest of the proposition follows from the fact that $\Omega_{f,\lambda}$ is adic, Theorem \ref{1th:limit}, the observation before the proof, and the same assertion when $\Xi$ is a singleton.
\end{proof}

\begin{definition}\label{1de:dualizing}
We introduce the following dualizing functors:
\begin{align*}
\cD=\cD_X\coloneqq\HOM_X(-,\Omega_{X,\lambda})&\colon\cD(X,\lambda)^{op}\to\cD(X,\lambda), \\
\cD^\ra=\cD^\ra_X\coloneqq\HOM^\ra_X(-,\Omega_{X,\lambda})&\colon\cD(X,\lambda)_\ra^{op}\to\cD(X,\lambda)_\ra.
\end{align*}
Put $\rD=\rh\cD$ and $\rD^\ra=\rh\cD^\ra$.
\end{definition}

\begin{proposition}
Let $(X,\lambda)$ be an object of $\Chpar{}\times\rN(\Rind)$. Let $\sfK\in\cD(X,\lambda)_\ra$ be an object such that $\delta_{\Omega_{X,\Lambda(\xi)}}(e_\xi^*\sfK)$ is an equivalence for every object $\xi$ of $\Xi$, where $\delta$ is the biduality transformation \eqref{1eq:biduality1}. Then $\delta^\ra_{\Omega_{X,\lambda}}(\sfK)$ is an equivalence as well, where $\delta^\ra$ is the biduality transformation \eqref{1eq:biduality2}.
\end{proposition}

\begin{proof}
We need to show that the natural morphism $\sfK\to\rD^\ra\rD^\ra\sfK$ is an isomorphism (in the homotopy category $\rh\cD(X,\lambda)_\ra$). By definition, we have
\begin{align*}
\rD^\ra\rD^\ra\sfK
&=\rh\HOM^\ra(\sfK,\rh\HOM^\ra(\sfK,\Omega_{X,\lambda}))\\
&\simeq\rh\fR_X\rh\HOM(\sfK,\rh\fR_X\rh\HOM(\sfK,\Omega_{X,\lambda})) \\
&\simeq\rh\fR_X\rh\HOM(\sfK,\rh\HOM(\sfK,\Omega_{X,\lambda})).
\end{align*}
It suffices to show that the map $\delta_{\Omega_{X,\lambda}}(\sfK)\colon\sfK\to\rh\HOM(\sfK,\rh\HOM(\sfK,\Omega_{X,\lambda}))$ is an equivalence. In fact, since $\sfK$ is adic, we have
\begin{align*}
e_\xi^*\rh\HOM(\sfK,\rh\HOM(\sfK,\Omega_{X,\lambda}))
&\simeq\rh\HOM(e_\xi^*\sfK,\rh\HOM(e_\xi^*\sfK,e_\xi^*\Omega_{X,\lambda}))\\
&\simeq\rh\HOM(e_\xi^*\sfK,\rh\HOM(e_\xi^*\sfK,\Omega_{X,\Lambda(\xi)}))
\end{align*}
for every object $\xi\in\Xi$ by Lemma \ref{1le:adic_hom} below, which is equivalent to $e_\xi^*\sfK$ by the assumption.
\end{proof}

\begin{lem}\label{1le:adic_hom}
Let $\lambda=(\Xi,\Lambda)$ be an object of $\Rind$, $\xi$ an object of $\Xi$, and $\sK$ an object of $\cD(X,\lambda)_\ra$. Then the following diagram
\[
\xymatrix{
\cD(X,\lambda) \ar[d]_-{e_\xi^*} && \cD(X,\lambda) \ar[d]^-{e_\xi^*}\ar[ll]_-{-\otimes_X\sfK} \\
\cD(X,\Lambda(\xi)) && \cD(X,\Lambda(\xi)) \ar[ll]_-{-\otimes_X e_\xi^*\sfK} }
\]
is right adjointable and its transpose is left adjointable. In other words, the natural maps $e_{\xi!}(\sfL\otimes_X e_\xi^*\sfK)\to(e_{\xi!}\sfL)\otimes_X\sfK$ and $e_\xi^*\HOM_X(\sfK,\sfL')\to\HOM(e_\xi^*\sfK,e_\xi^*\sfL')$ are equivalences for objects $\sfL$ of $\cD(X,\Lambda(\xi))$ and $\sfL'$ of $\cD(X,\lambda)$.
\end{lem}

\begin{proof}
By \cite{LZ1}*{6.2.7}, we may assume that $\xi$ is the final object of $\Xi$. In this case, $e_\xi^*$ can be identified with $\pi_*$, where $\pi\colon(\Xi,\Lambda)\to(\{\xi\},\Lambda(\xi))$ is the projection. Since $\sfK$ is adic, the morphism $\pi^*e_\xi^*\sfK\to\sfK$ is an equivalence. A left adjoint of the transpose of the above diagram is then given by the diagram
\[
\xymatrix{
\cD(X,\lambda) \ar[d]_-{-\otimes_X\sfK} && \cD(X,\Lambda(\xi)) \ar[ll]_-{\pi^*}\ar[d]^-{-\otimes_X e_\xi^*\sfK}  \\
\cD(X,\lambda) && \cD(X,\Lambda(\xi)) \ar[ll]_-{\pi^*}.
}
\]
The lemma follows by adjunction.
\end{proof}

\section{The $\fm$-adic formalism and constructibility}
\label{2}

In this chapter, we make a finer study of the adic formalism for a special kind of ringed diagrams, which we call the $\fm$-adic formalism. It includes the most common application, namely, the $\ell$-adic one. We start in \Sec\ref{2ss:madic} by introducing such $\fm$-adic formalism. In \Sec\ref{2ss:finiteness_condition}, we introduce the finiteness condition under which the $\fm$-adic formalism behaves in a very nice way. In \Sec\ref{2ss:constructible_adic}, we study the constructible adic complexes and their behaviour under the six operations. The last section \Sec\ref{2ss:compatibility} is dedicated to proving the compatibility between our theory and Laszlo--Olsson \cites{LO2,LO3} under their restrictions.

\subsection{The $\fm$-adic formalism}
\label{2ss:madic}

\begin{definition}\label{2de:pring}
Define a category $\PRing$ as follows. The objects are pairs $(\Lambda,\fm)$, where $\Lambda$ is a (small) ring and $\fm\subseteq\Lambda$ is a principal ideal, such that
\begin{itemize}
  \item $\fm$ is generated by an element that is not a zero divisor;

  \item the natural homomorphism $\Lambda\to\lim_{n}\Lambda_n$ is an isomorphism, where $\Lambda_n=\Lambda/\fm^{n+1}$ ($n\in\dN$).
\end{itemize}
A morphism from $(\Lambda',\fm')$ to $(\Lambda,\fm)$ is a ring homomorphism $\phi\colon\Lambda\to\Lambda'$ satisfying $\phi(\fm)\subseteq\fm'$. We denote by $\PRing_{\tor}$ (resp.\ $\PRing_{\ltor}$) the full subcategory of $\PRing$ spanned by $(\Lambda,\fm)$ such that $(\dN,\Lambda_\bullet)$ belongs to $\Rind_\tor$ (resp.\ $\Rind_{\ltor}$).

We have a natural functor $\PRing\to\Fun([1],\Rind)$ sending $(\Lambda,\fm)$ to $(\dN,\Lambda_\bullet)\xrightarrow{\pi}(\ast,\Lambda)$. In what follows, we simply write $\Lambda_\bullet$ for the ringed diagram $(\dN,\Lambda_\bullet)$.
\end{definition}

Let $(\Lambda,\fm)$ be an object of $\PRing$. Let $X\in\Chpar{}$ be a higher Artin stack. We have a pair of adjoint functors
\[
\rL\pi^*\colon\cD(X,\Lambda)\to\cD(X,\Lambda_\bullet),\quad\rR\pi_*\colon\cD(X,\Lambda_\bullet)\to\cD(X,\Lambda).
\]
Note that our notation here is different from those in \cite{LZ1} as we add $\rL$ and $\rR$ for the ``derived'' functors, since later we will consider $(\pi^*,\pi_*)$ on the level of Abelian categories. As $\pi$ is perfect in the sense of \cite{LZ1}*{2.2.8}, the functor $\rL\pi^*$ admits a left adjoint \cite{LZ1}*{6.2.6} and in particular preserves small limits.

\begin{definition}[Normalized complex]
A complex $\sfK\in\cD(X,\Lambda_\bullet)$ is said \emph{normalized} if the cofiber\footnote{The underlying object in the ordinary triangulated category is a cone \cite{Lu2}*{1.1.2.11}.} \cite{Lu2}*{1.1.1.6} of the adjunction map $\rL\pi^*\rR\pi_*\sfK\to\sfK$ is $0$. We denote by
$\cD(X,\Lambda_\bullet)_\rn$ the full subcategory of $\cD(X,\Lambda_\bullet)$ spanned by normalized complexes.
\end{definition}

The subcategory $\cD(X,\Lambda_\bullet)_\rn\subseteq\cD(X,\Lambda_\bullet)$ is a stable subcategory stable under small limits. Note that $\cD(X,\Lambda)=\cD(X,\Lambda)_\ra$, so that the image of $\rL\pi^*$ is contained in $\cD(X,\Lambda_\bullet)_\ra$. In particular, we have
$\cD(X,\Lambda_\bullet)_\rn\subseteq\cD(X,\Lambda_\bullet)_\ra$. For the other direction, we have the following result. We define
$\cD(X,\Lambda_\bullet)^{(+)}_\ra=\cD(X,\Lambda_\bullet)_\ra\cap\cD^{(+)}(X,\Lambda_\bullet)$\footnote{We deliberately not denote this intersection by $\cD^{(+)}(X,\Lambda_\bullet)_\ra$ as $\cD(X,\Lambda_\bullet)_\ra$ carries a usual t-structure itself, and we do not know whether $\cD^{(+)}(X,\Lambda_\bullet)_\ra=\cD(X,\Lambda_\bullet)^{(+)}_\ra$ always holds. However, see Remark \ref{2re:bound}.}.

\begin{lem}\label{2le:adic_normalize}
We have $\cD(X,\Lambda_{\bullet})_\ra^{(+)}\subseteq\cD(X,\Lambda_\bullet)_\rn$.
\end{lem}

\begin{proof}
The proof is similar to \cite{Zh}*{4.13}.
\end{proof}

As the operations $\otimes$, $f^*$, $f_!$ preserve adic complexes, they preserve normalized complexes in $\cD^{(+)}(-,\Lambda_\bullet)$. Next we examine effects of $\HOM$, $f_*$, $f^!$ on normalized complexes, which imply that the restrictions of $\HOM^\ra$, $f_{*\ra}$, $f^{!\ra}$ to $\cD(-,\Lambda_\bullet)_\ra^{(+)}$ coincide with $\HOM$, $f_*$, $f^!$, respectively.

\begin{proposition}\label{2pr:hom_normalize}
Let $X$ be a higher Artin (resp.\ higher Deligne--Mumford) stack and let $(\Lambda,\fm)$ be an object of $\PRing_{\ltor}$ (resp.\ $\PRing$). For $\sfK,\sfL\in\cD(X,\Lambda_\bullet)_\rn$, and more generally for $\sfK,\sfL$ in the essential image of $\rL\pi^*$, the complex $\HOM_X(\sfK,\sfL)$ is adic. In particular, up to equivalence, $\HOM_X$ restricts to the functor
\[
\HOM_X^\ra\colon(\cD(X,\Lambda_\bullet)^{(+)}_\ra)^{op}\times\cD(X,\Lambda_\bullet)^{(+)}_\ra\to\cD(X,\Lambda_\bullet)_\ra.
\]
\end{proposition}

\begin{proof}
By the Poincar\'e duality, we may reduce to the case of schemes. Then it is essentially proved in \cite{Zh}*{4.18}.
\end{proof}

\begin{proposition}\label{2pr:star_pushforward_normalize}
Let $f\colon Y\to X$ be a morphism of higher Artin stacks and let $(\Lambda,\fm)$ be an object of $\PRing$. Then $f_*\colon\cD(Y,\Lambda_\bullet)\to\cD(X,\Lambda_\bullet)$ preserves normalized complexes. In particular, up to equivalence, $f_*$ restricts to the functor
\[
f_{*\ra}\colon\cD(Y,\Lambda_\bullet)^{(+)}_\ra\to\cD(X,\Lambda_\bullet)^{(+)}_\ra.
\]
\end{proposition}

\begin{proof}
This follows from the fact that $f_*$ commutes with $\rL\pi^*$ \cite{LZ1}*{6.2.6}.
\end{proof}

\begin{proposition}\label{2pr:shrink_pullback_normalize}
Let $f\colon Y\to X$ be a morphism locally of finite type in $\Chpar{}_\Box$ and let $(\Lambda,\fm)$ be an object of $\PRing_{\ltor}$. Then $f^!\colon\cD(X,\Lambda_\bullet)\to \cD(Y,\Lambda_\bullet)$ preserves normalized complexes. In particular, up to equivalence, $f^!$ restricts to the functor
\[
f^{!\ra}\colon\cD(X,\Lambda_\bullet)^{(+)}_\ra\to\cD(X,\Lambda_\bullet)^{(+)}_\ra.
\]
\end{proposition}

\begin{proof}
By the Poincar\'e duality applied to atlases, we can reduce the proposition to the case of a closed immersion of schemes, which follows from the fact that $f^!$ commutes with $\rL\pi^*$ \cite{LZ1}*{6.2.7}.
\end{proof}

The truncation functors $\tau^{\le n}$, $\tau^{\ge n}$ do not preserve normalized complexes in general. In the rest of this section, we study the effects of the truncation functors on normalized complexes.

Let $\cA$ be an Abelian category. An object $M_\bullet$ in $\Fun(\dN^{op},\cA)$ is called \emph{essentially null} if for each $n\in\dN$, there is an element $r\in\dN$ such that $M_{r+n}\to M_n$ is the zero morphism. If $\cA$ admits sequential limits, then we have a left exact functor $\lim\colon\Fun(\dN^{op},\cA)\to\cA$. Given a topos $X$, we have a pair of adjoint functors
\[
\pi^*\colon \Mod(X,\Lambda)\to\Mod(X^\dN,\Lambda_\bullet),\quad\pi_*\colon\Mod(X^\dN,\Lambda_\bullet)\to\Mod(X,\Lambda)
\]
induced by the morphism $\pi\colon(\dN,\Lambda_\bullet)\to(\ast,\Lambda)$. Then we have
\[
\pi_*=\lim\circ\nu,
\]
where $\nu\colon\Mod(X^\dN,\Lambda_\bullet)\to\Fun(\dN^{op},\Mod(X,\Lambda))$ is the obvious forgetful functor, which is exact.

\begin{lem}\label{2le:essentially_null}
Let $F_{\bullet}$ be a module in $\Mod(X^\dN,\Lambda_\bullet)$ such that $\nu F_\bullet$ is essentially null. Then $\rR^n\pi_*F_\bullet=0$ for
all $n\geq0$.
\end{lem}

\begin{proof}
Note that $\rR^n\pi_*F_\bullet$ is the sheaf associated to the presheaf $(U\mapsto\rH^n(U^\dN,F_\bullet))$, where $U$ runs over objects of $X$.
Let $a\colon(U^\dN,\Lambda_\bullet)\to(\ast^\dN,\Lambda_\bullet)$ be the morphism of ringed topoi. Since $\rR^q a_*F_\bullet$ is essentially null for all $q$, we have $\rR\Gamma(U^\dN,F_\bullet)\simeq\rR\lim\rR a_*F_\bullet=0$.
\end{proof}

Suppose that $X$ is a higher Artin stack and let $(\Lambda,\fm)$ be an object of $\PRing$. Let $\cD_0(X,\Lambda_\bullet)$ be the full subcategory of $\cD(X,\Lambda_\bullet)$ spanned by complexes whose cohomology sheaves are all essentially null. Put $\cD^{(+)}_0(X,\Lambda_\bullet)=\cD^{(+)}(X,\Lambda_\bullet)\cap\cD_0(X,\Lambda_\bullet)$. Both are stable subcategories.

\begin{lem}\label{2le:pi_vanish}
For $\sfK\in\cD^{(+)}_0(X,\Lambda_\bullet)$, we have $\rR\pi_*\sfK=0$.
\end{lem}

\begin{proof}
Let $f_0\colon X_0\to X$ be a smooth atlas and let $X_\bullet$ be a \v{C}ech nerve of $f_0$. Then we have $\sfK\simeq \lim f_{n*}f_n^*\sfK$ by \cite{LZ1}*{6.2.13 (1)}, where $f_n\colon X_n\to X$ is the induced morphism. Therefore, we have
\[
\rR\pi_*\sfK\simeq \lim \rR\pi_*f_{n*}f_n^*\sfK\simeq\lim f_{n*}\rR\pi_*f_n^*\sfK.
\]
Thus it suffices to show the lemma for each $X_n$. By induction, we may assume $X\in\Schqcs$ and $\sfK\in\cD^+(X,\Lambda_\bullet)\cap\cD_0(X,\Lambda_\bullet)$. Then the statement follows from Lemma \ref{2le:essentially_null}.
\end{proof}

\begin{definition}
A complex $\sfK\in\cD(X,\Lambda_\bullet)$ is called \emph{essentially normalized} if the cofiber of the adjunction map $\rL\pi^*\rR\pi_*\sfK\to\sfK$ is in $\cD^{(+)}_0(X,\Lambda_\bullet)$. We denote by $\cD(X,\Lambda_\bullet)_{\r{en}}$ the full subcategory of
$\cD(X,\Lambda_\bullet)$ spanned by essentially normalized complexes, which is a stable subcategory.
\end{definition}

\begin{lem}\label{2le:essentially_normalize}
The image of the functor $\rL\pi^*\circ\rR\pi_*\res\cD(X,\Lambda_\bullet)_{\r{en}}$ is contained in $\cD(X,\Lambda_\bullet)_\rn$. Moreover, the induced functor
\[
\rL\pi^*\circ\rR\pi_*\colon\cD(X,\Lambda_\bullet)_{\r{en}}\to\cD(X,\Lambda_\bullet)_\rn
\]
is right adjoint to the obvious inclusion $\cD(X,\Lambda_\bullet)_\rn\subseteq\cD(X,\Lambda_\bullet)_{\r{en}}$.
\end{lem}

\begin{proof}
For the fist assertion, we need to show that $\rL\pi^*\rR\pi_*\rL\pi^*\rR\pi_*\sfK\to\rL\pi^*\rR\pi_*\sfK$ is an equivalence for $\sfK\in\cD(X,\Lambda_\bullet)_{\r{en}}$. By definition, the cofiber of $\rL\pi^*\rR\pi_*\sfK\to\sfK$ is contained in
$\cD^{(+)}_0(X,\Lambda_\bullet)$. The assertion then follows from Lemma \ref{2le:pi_vanish}.

For the second assertion, we need to show that the natural transformation $\rL\pi^*\circ\rR\pi_*\to\id$ induces a homotopy equivalence (that is, an equivalence in $\cH$)
\[
\Map_{\cD(X,\Lambda_\bullet)_\rn}(\sfK,\rL\pi^*\rR\pi_*\sfL)\to\Map_{\cD(X,\Lambda_\bullet)_{\r{en}}}(\sfK,\sfL),
\]
for every object $\sfK$ (resp.\ $\sfL$) of $\cD(X,\Lambda_\bullet)_\rn$ (resp.\ $\cD(X,\Lambda_\bullet)_{\r{en}}$). By definition, the cofiber
$\sfL'$ of $\rL\pi^*\rR\pi_*\sfL\to\sfL$ is in $\cD^{(+)}_0(X,\Lambda_\bullet)$, and $\sfK$ is equivalent to $\rL\pi^*\rR\pi_*\sfK$. Therefore, the assertion follows from the fact that
\[
\Map_{\cD(X,\Lambda_\bullet)_{\r{en}}}(\rL\pi^*\rR\pi_*\sfK,\sfL')\simeq\Map_{\cD(X,\Lambda)}(\rR\pi_*\sfK,\rR\pi_*\sfL')\simeq \{\ast\}.
\]
Here in the second equivalence, we have used the fact $\rR\pi_*\sfL'=0$, which follows from Lemma \ref{2le:pi_vanish}.
\end{proof}

\begin{lem}\label{2le:hom_vanish}
For $\sfK\in\cD(X,\Lambda_\bullet)_\ra\cap\cD^{(-)}(X,\Lambda_\bullet)$ and $\sfL\in\cD^{(+)}_0(X,\Lambda_\bullet)$, we have $\HOM_X(\sfK,\sfL)\in\cD^{(+)}_0(X,\Lambda_\bullet)$ and $\Hom(\sfK,\sfL)=0$.
\end{lem}

\begin{proof}
The proof is similar to \cite{Zh}*{4.19}.
\end{proof}

\begin{lem}\label{2le:normalize_truncation}
For every $\sfK\in\cD(X,\Lambda_\bullet)_\rn$ and $n\in\dZ$, we have $\tau^{\geq n}\sfK\in\cD(X,\Lambda_\bullet)_{\r{en}}$. Moreover, the
fiber of the adjunction map $\rL\pi^*\rR\pi_*\tau^{\geq n}\sfK\to\tau^{\geq n}\sfK$ is concentrated in degree $n-1$ and belongs to $\cD_0(X,\Lambda_\bullet)$.
\end{lem}

\begin{proof}
This is essentially proved in \cite{Zh}*{4.14}. Let us recall the arguments. We know that the fiber of the map $\rL\pi^*\tau^{\ge n}\rR\pi_*\sfK\to\tau^{\ge n}\rL\pi^*\rR\pi_*\sfK$ is concentrated in degree $n-1$ and belongs to $\cD_0(X,\Lambda_\bullet)$. Note that we have $\rL\pi^*\rR\pi_*\sfK\simeq\sfK$ as $\sfK\in\cD(X,\Lambda_\bullet)_\rn$. So the fiber of the map $a\colon\rL\pi^*\tau^{\ge n}\rR\pi_*\sfK\to\tau^{\ge n}\sfK$ is concentrated in degree $n-1$ and belongs to $\cD_0(X,\Lambda_\bullet)$.

Consider the diagram
\[
\xymatrix{
\rL\pi^*\rR\pi_*\rL\pi^*\tau^{\ge n}\rR\pi_*\sfK \ar[d]_{\rL\pi^*\rR\pi_*a}\ar[r]^-b&\rL\pi^*\tau^{\ge n} \rR\pi_* \sfK\ar[d]^a\\
\rL\pi^*\rR\pi_*\tau^{\ge n} \sfK\ar[r]^-c &\tau^{\ge n}\sfK.}
\]
By Lemma \ref{2le:essentially_null}, $\rL\pi^*\rR\pi_*a$ is an equivalence. By Proposition \ref{2le:adic_normalize}, $b$ is an equivalence. Therefore, the fiber of $c$ is equivalent to the fiber of $a$.
\end{proof}

\begin{proposition}\label{2pr:essentially_normalize_t}
Let $\Mod(X,\Lambda_{\bullet})_{\r{en}}$ be the full subcategory of $\cD(X,\Lambda_\bullet)_{\r{en}}$ spanned by complexes that are
concentrated at degree $0$. For $n\in\d Z$, put
\[
\cD^{\leq n}(X,\Lambda_\bullet)_{\r{en}}=\cD^{\leq n}(X,\Lambda_\bullet)\cap\cD(X,\Lambda_\bullet)_{\r{en}},\quad
\cD^{\geq n}(X,\Lambda_\bullet)_{\r{en}}=\cD^{\geq n}(X,\Lambda_\bullet)\cap\cD(X,\Lambda_\bullet)_{\r{en}}.
\]
Then
\begin{enumerate}
  \item $(\cD^{\leq 0}(X,\Lambda_\bullet)_{\r{en}},\cD^{\geq 0}(X,\Lambda_\bullet)_{\r{en}})$ defines a t-structure on $\cD(X,\Lambda_{\bullet})_{\r{en}}$ whose heart is $\Mod(X,\Lambda_\bullet)_{\r{en}}$; and

  \item $\Mod(X,\Lambda_\bullet)_{\r{en}}$ is (equivalent to the nerve of) a full subcategory of $\Mod(X,\Lambda_\bullet)$, closed under kernels, cokernels and extensions.
\end{enumerate}
\end{proposition}

\begin{proof}
For (1), we only need to show that $\tau^{\leq 0}$ and $\tau^{\geq 0}$ preserve the full subcategory $\cD(X,\Lambda_\bullet)_{\r{en}}$. Since
$\cD(X,\Lambda_\bullet)_{\r{en}}$ is a stable full subcategory, we only need to prove this for $\tau^{\geq 0}$, that is, the cofiber of the
adjunction map $\rL\pi^*\rR\pi_*\tau^{\geq0}\sfK\to\tau^{\geq0}\sfK$ is in $\cD^{(+)}_0(X,\Lambda_\bullet)$ for every object $\sfK$ of
$\cD(X,\Lambda_\bullet)_{\r{en}}$. Consider the diagram
\[
\xymatrix{
\rL\pi^*\rR\pi_*\tau^{\ge 0} \rL\pi^*\rR\pi_*\sfK \ar[r]^-b\ar[d]_{\rL\pi^*\rR\pi_*a} & \tau^{\ge 0}\rL\pi^*\rR\pi_*\sfK\ar[d]^a\\
\rL\pi^*\rR\pi_*\tau^{\ge 0}\sfK \ar[r]^-c & \tau^{\ge 0}\sfK.}
\]
By definition, the cofiber of $\rL\pi^*\rR\pi_*\sfK\to\sfK$ is in $\cD^{(+)}_0(X,\Lambda_\bullet)$, so that the cofiber of $a$ is in
$\cD^{(+)}_0(X,\Lambda_\bullet)$. It follows that $\rL\pi^*\rR\pi_*a$ is an equivalence, by Lemma \ref{2le:pi_vanish}. By Lemma
\ref{2le:essentially_normalize}, we have $\rL\pi^*\rR\pi_*\sfK\in\cD(X,\Lambda_\bullet)_\rn$. Thus, by Lemma \ref{2le:normalize_truncation}, the cofiber of $b$ is in $\cD^{(+)}_0(X,\Lambda_\bullet)$. Therefore, by the octahedral axiom, the cofiber of $c$ is in $\cD^{(+)}_0(X,\Lambda_\bullet)$ as well.

For (2), it follows from (1).
\end{proof}

\begin{corollary}\label{2co:normalize_t}
The essential image of
\[
\rL\pi^*\circ\rR\pi_*\res\cD^{\geq n}(X,\Lambda_\bullet)_{\r{en}}\colon\cD^{\geq n}(X,\Lambda_\bullet)_{\r{en}}\to\cD(X,\Lambda_\bullet)_\rn
\]
is right perpendicular to the full subcategory $\cD(X,\Lambda_\bullet)_\rn\cap\cD^{<n}(X,\Lambda_\bullet)$ of $\cD(X,\Lambda_\bullet)_\rn$.
\end{corollary}

\subsection{Finiteness conditions}
\label{2ss:finiteness_condition}

Let $X\in\Chpar{}$ be a higher Artin stack and $(\Lambda,\fm)$ an object of $\PRing$. Recall that we have the full subcategory
\[
\cD(X,\Lambda_\bullet)_\rn\subseteq\cD(X,\Lambda_\bullet)_\ra.
\]

\begin{definition}\label{2de:admissible}
The pair $(X,(\Lambda,\fm))$ is said \emph{admissible} if $\cD(X,\Lambda_\bullet)_\ra\subseteq\cD(X,\Lambda_\bullet)_\rn$ (so that
$\cD(X,\Lambda_\bullet)_\ra=\cD(X,\Lambda_\bullet)_\rn$), that is, for every $\sfK\in\cD(X,\Lambda_\bullet)_\ra$, the adjunction map $\rL\pi^*\rR\pi_*\sfK\to\sfK$ is an equivalence.
\end{definition}

\begin{proposition}\label{2pr:adic_t}
Let $(X,(\Lambda,\fm))$ be an admissible pair. Then we have
\begin{enumerate}
  \item $\fR_X\res\cD(X,\Lambda_\bullet)_{\r{en}}\simeq\rL\pi^*\circ\rR\pi_*$, where $\fR_X$ is the colocalization functor in Remark \ref{1re:limit} (1); and

  \item $\cD^{\geq n}(X,\Lambda_\bullet)_\ra$ is the essential image of $\rL\pi^*\circ\rR\pi_*\res\cD^{\geq n}(X,\Lambda_\bullet)_{\r{en}}$.
\end{enumerate}
\end{proposition}

\begin{proof}
Part (1) follows from Lemma \ref{2le:essentially_normalize}.

For (2), we denote by $\cD'$ the essential image of $\rL\pi^*\circ\rR\pi_*\res\cD^{\geq n}(X,\Lambda_\bullet)_{\r{en}}$. Then $\cD'\subseteq\cD^{\geq n}(X,\Lambda_\bullet)_\ra$ by Corollary \ref{2co:normalize_t}. Now we take $\sfK\in\cD^{\ge n}(X,\Lambda_\bullet)_\ra$. As $(X,(\Lambda,\fm))$ is an admissible pair, $\sfK$ is in $\cD(X,\Lambda_\bullet)_\rn$, that is, $\rL\pi^*\rR\pi_*\sfK\simeq\sfK$. Since $\rL\pi^*\colon\cD(X,\Lambda)\to\cD(X,\Lambda_\bullet)_\ra$ is left t-exact, we know that $\rR\pi_*\res\cD(X,\Lambda_\bullet)_\ra\colon\cD(X,\Lambda_\bullet)_\ra\to\cD(X,\Lambda)$ is right t-exact. Thus $\sfK\simeq\rL\pi^*\rR\pi_*\sfK\simeq\rL\pi^*\tau^{\ge n}\rR\pi_*\sfK$. As the fiber of $\rL\pi^*\tau^{\ge n}\rR\pi_*\sfK\to\tau^{\ge n}\rL\pi^*\rR\pi_*\sfK$ is concentrated in degree $n-1$ and belongs to $\cD_0(X,\Lambda_\bullet)$, so is the fiber of $\sfK\to\tau^{\ge n}\sfK$. Therefore, we have $\sfK\simeq\rL\pi^*\rR\pi_*\sfK\simeq\rL\pi^*\rR\pi_*\tau^{\ge n}\sfK$ by Lemma \ref{2le:pi_vanish}. In other words, we have $\sfK\in\cD'$.
\end{proof}

\begin{remark}\label{2re:t}
Let $\Mod(X,\Lambda_{\bullet})'_{\r{en}}$ be the full subcategory of $\Mod(X,\Lambda_{\bullet})_{\r{en}}$ (introduced in Proposition \ref{2pr:essentially_normalize_t}) spanned by complexes $\sfK$ such that $\rH^i\bL\pi^*\bR\pi_*\sfK=0$ for $i>0$, which is an exact category. Let $\Mod(X,\Lambda_{\bullet})_0$ be the full subcategory of $\Mod(X,\Lambda_{\bullet})_{\r{en}}$ spanned by essentially null modules, which is closed under sub-objects, quotients and extensions. If $(X,(\Lambda,\fm))$ is an admissible pair, then the projection functor
\[
\cD^{\heartsuit}(X,\Lambda_{\bullet})_\ra\to\Mod(X,\Lambda_{\bullet})'_{\r{en}}/\Mod(X,\Lambda_{\bullet})_0
\]
from the heart of $\cD(X,\Lambda_{\bullet})_\ra$ with respect to the usual t-structure to the full subcategory of $\Mod(X,\Lambda_{\bullet})_{\r{en}}/\Mod(X,\Lambda_{\bullet})_0$ spanned by the image of $\Mod(X,\Lambda_{\bullet})'_{\r{en}}$ is an equivalence of categories. In fact, the functor $\bL\pi^*\circ\bR\pi_*\res\Mod(X,\Lambda_{\bullet})'_{\r{en}}$ induces a quasi-inverse.
\end{remark}

\begin{remark}\label{2re:bound}
Let $(X,(\Lambda,\fm))$ be an admissible pair. Proposition \ref{2pr:adic_t} implies that $\cD^{(+)}(X,\Lambda_\bullet)_\ra=\cD(X,\Lambda_\bullet)^{(+)}_\ra$.
\end{remark}

\begin{proposition}\label{2pr:right_exact}
Let $f\colon Y\to X$ be a morphism of higher Artin stacks. Let $(\Lambda,\fm)$ be an object of $\PRing$ such that $(X,(\Lambda,\fm))$ is admissible. Then the functor
\[
f^{*\ra}\colon\cD(X,\Lambda_\bullet)_\ra\to\cD(Y,\Lambda_\bullet)_\ra
\]
is t-exact with respect to the usual t-structures.
\end{proposition}

\begin{proof}
By Remark \ref{1re:p6} (4), we only need to show the right t-exactness of $f^{*\ra}$. Take $\sfK\in\cD^{\geq n}(X,\Lambda_\bullet)_\ra$ and $\sfL\in\cD^{\le n-1}(Y,\Lambda_\bullet)_\ra$. Consider the fiber sequence
\[
\tau^{\le n-1}\sfK\to\sfK\to\tau^{\ge n}\sfK
\]
in $\cD(X,\Lambda_\bullet)$. It induces a fiber sequence
\[
f^*\tau^{\le n-1}\sfK\to f^*\sfK=f^{*\ra}\sfK\to f^*\tau^{\ge n}\sfK
\]
in $\cD(Y,\Lambda_\bullet)$. As $(X,(\Lambda,\fm))$ is admissible, $\sfK$ belongs to $\cD(X,\Lambda_\bullet)_\rn$. By Lemma \ref{2le:normalize_truncation}, we have $\tau^{\le n-1}\sfK\in\cD(X,\Lambda_\bullet)_{\r{en}}$. As $0=\tau^{\le n-1}_\ra\sfK\simeq\fR_X\tau^{\le n-1}\sfK$, which is isomorphic to $\rL\pi^*\rR\pi_*\tau^{\le n-1}\sfK$ by Proposition \ref{2pr:adic_t}, we know that $\tau^{\le n-1}\sfK\in\cD^{(+)}_0(X,\Lambda_\bullet)$ by Lemma \ref{2le:normalize_truncation}. Thus we have $f^*\tau^{\le n-1}\sfK\in\cD^{(+)}_0(Y,\Lambda_\bullet)$, and $\Hom(\sfL,f^*\tau^{\le n-1}\sfK)=0$ by Lemma \ref{2le:hom_vanish}. It follows that $\Hom(\sfL,f^{*\ra}\sfK)=0$. Therefore, $f^{*\ra}\sfK\in\cD^{\geq n}(Y,\Lambda_\bullet)_\ra$. The proposition is proved.
\end{proof}

Let $f\colon Y\to X$ be a smooth surjective morphism in $\Chpar{}_\Box$ (resp.\ $\Chpdm{}$), and $(\Lambda,\fm)$ an object of $\PRing_{\ltor}$ (resp.\ $\PRing$). By the Poincar\'e duality, if $(Y,(\Lambda,\fm))$ is admissible, then $(X,(\Lambda,\fm))$ is locally admissible. This applies in particular to the case where $Y$ is an algebraic space. In this case, admissibility is related to the following finiteness condition on cohomological dimension.

\begin{definition}\label{2de:bounded}
Let $X$ be a higher Artin (resp.\ higher Deligne--Mumford) stack and $R$ a ring. We say that $X$ is \emph{locally $R$-bounded}, if there exists an atlas (resp.\ \'{e}tale atlas) $\coprod_{i\in I}X_i\to X$ with $X_i$ algebraic spaces such that for every $i\in I$, and every scheme $U$ \'{e}tale and of finite presentation over $X_i$, we have
\[
\r{cd}_R(U)\coloneqq\max\{n\res \r{H}^n(U,F)\neq 0\text{ for some }F\in\Mod(U,R)\}<\infty.
\]
\end{definition}

\begin{proposition}\label{2pr:five_condition}
Let $X$ be an algebraic space and $(\Lambda,\fm)$ an object of $\PRing$. Consider the following conditions:
\begin{enumerate}
  \item The pair $(X,(\Lambda,\fm))$ is admissible.

  \item For every $\sfK\in\cD(X,\Lambda_\bullet)_\ra$, we have
      \[
      \rR\pi_*(F_\bullet\overset{\rL}\otimes_{\Lambda_\bullet}\sfK)=0,
      \]
      where $F_\bullet\in\Mod(X^\dN_{\et},\Lambda_\bullet)\simeq\cD^{\heartsuit}(X,\Lambda_\bullet)$ is the module
      \[
      \cdots\xrightarrow{0}\Lambda/\fm\xrightarrow{0}\cdots\xrightarrow{0}\Lambda/\fm.
      \]

  \item We have $\rR\pi_*\sfK=0$ for every $\sfK\in\cD_0(X,\Lambda_\bullet)$.

  \item There exists an \'etale cover $\coprod_{i\in I} X_i\to X$ by algebraic spaces such that, for every $i\in I$, the cohomological dimension of $\pi_*\colon\Mod(X_{i,\et}^\dN,\Lambda_\bullet)\to\Mod(X_{i,\et},\Lambda)$ is finite.

  \item The algebraic space $X$ is locally $(\Lambda/\fm)$-bounded.
\end{enumerate}
We have $(5)\Rightarrow(4)\Rightarrow(3)\Rightarrow(2)\Leftrightarrow(1)$.
\end{proposition}

\begin{proof}
$(5)\Rightarrow(4)$: By the \'{e}tale base change, we can assume $\r{cd}_{\Lambda/\fm}(U)=N<\infty$ for every scheme $U$ \'{e}tale and of finite type over $X$. Since for $n\in\dN$, every $\Lambda_n=\Lambda/\fm^{n+1}$-module is a successive extension of $\Lambda/\fm$-modules, we have $\r{cd}_{\Lambda_n}(U)=N$. For a sheaf $F_\bullet\in\Mod(X_{\et}^\dN,\Lambda_\bullet)$, $\rR^i\pi_*F_\bullet$ is the sheaf associated to the presheaf $U\mapsto\rH^i(U_{\et}^\dN,F_\bullet)$. Thus, from the exact sequence
\[
\xymatrix{
0 \ar[r] & \rR^1\lim_n\rH^{i-1}(U_{\et},F_n) \ar[r] & \rH^i(U_{\et}^\dN,F_\bullet)
\ar[r] & \lim_n\rH^i(U_{\et},F_n) \ar[r] & 0, }
\]
we know that $\rR^i\pi_*F_\bullet=0$ for $i>N+1$.

$(4)\Rightarrow(3)$: We can assume $X$ quasi-compact. Then this follows from Lemma \ref{2le:pi_vanish} and the following standard observation: If $f\colon\cB\to \cA$ is a left exact additive functor of Grothendieck Abelian categories such that $\rR^if=0$ for $i>d$ for some integer $d\geq0$, then $\rR f$ sends $\cD^{\leq n}(\cB)$ to $\cD^{\leq n+d}(\cA)$. In fact, let $X$ be an element of $\cD^{\le n}(\cB)$. By \cite{KS}*{14.3.4}, we can compute $\rR f X$ by $fY$, where $Y$ is an \emph{arbitrary} resolution of $X$ with $f$-acyclic components. We can take $Y$ to be the image of $\tau^{\le n+d+1}$ of a homotopically injective resolution with injective components (fibrant replacement) of $X$. Then $Y$ is a complex concentrated in degree $\leq n+d+1$. This shows that $\rR f$ sends $\cD^{\le n}(\cB)$ to $\cD^{\le n+d+1}(\cA)$. It follows that $\rR f$ sends $\cD^{\le n}(\cB)$ to $\cD^{\le n+d}(\cA)$ by truncation.

$(3)\Rightarrow(2)$: In fact, for every $\sfK\in\cD(X,\Lambda_\bullet)$, we have $F_\bullet\overset{\rL}\otimes_{\Lambda_{\bullet}}\sfK\in\cD_0(X,\Lambda_{\bullet})$.

$(2)\Leftrightarrow(1)$: For $\sfK\in\cD(X,\Lambda_\bullet)_\ra$, we need to show that $\rR\pi_*(F_\bullet\overset{\rL}\otimes_{\Lambda_\bullet}\sfK)=0$ if and only if the adjunction map $\rL\pi^*\rR\pi_*\sfK\to\sfK$ is an equivalence. Since $\prod_{n\in\dN}e_n^*$ is conservative, the latter is equivalent to the condition that the morphism
\[
\epsilon\colon\Lambda_n\overset{\rL}{\otimes}_{\Lambda}\rR\pi_*\sfK\to\sfK_n\coloneqq e_n^*\sfK
\]
is an isomorphism in $\cD(X,\Lambda_n)$ for every $n\in\dN$. The morphism $\epsilon$ can be decomposed as
\begin{align*}
\Lambda_n\overset{\rL}{\otimes}_\Lambda\rR\pi_*\sfK
&\xrightarrow{\alpha} \rR\pi_*(\rL\pi^*\Lambda_n\overset{\rL}{\otimes}_{\Lambda_\bullet}\sfK)
\xrightarrow{\beta}\rR\pi_*(\pi^*\Lambda_n\overset{\rL}{\otimes}_{\Lambda_\bullet}\sfK) \\
&\xrightarrow{\gamma}\rR(\pi_{\ge n})_*(\pi_{\ge n}^*\Lambda_n\overset{\rL}{\otimes}_{\Lambda_{\bullet,\ge n}}\sfK_{\ge n})
\xrightarrow{\delta}\rR(\pi_{\ge n})_*(e_{n*}\sfK_n)\simeq\sfK_n,
\end{align*}
where $\pi_{\ge n}\colon(\dN_{\ge n},\Lambda_{\bullet,\ge n})\to(\ast,\Lambda)$ and $e_n\colon (\{n\},\Lambda)\to(\dN_{\ge n},\Lambda_{\bullet,\ge n})$ are obvious morphisms. Here, $\dN_{\ge n}\subseteq \dN$ is the full subcategory spanned by integers $\ge n$. We show that $\alpha$, $\gamma$ and $\delta$ are all isomorphisms.

By assumption, $\fm$ is generated by an element $\lambda$ that is not a zero divisor. Thus we have a finite free resolution $[\Lambda\xrightarrow{\times\lambda^{n+1}}\Lambda]$ of $\Lambda_n$ as a $\Lambda$-module. Therefore, $\rL\pi^*\Lambda_n$ is represented by the
complex of $\Lambda_\bullet$-modules $[\Lambda_\bullet\xrightarrow{\times\lambda^{n+1}}\Lambda_\bullet]$ (in degrees $-1$ and $0$). This implies that $\rL\pi^*\Lambda_n\overset{\rL}{\otimes}_{\Lambda_\bullet}\sfK$ is represented by the mapping cone of
$\sfK\xrightarrow{\times\lambda^{n+1}}\sfK$, which is a fibrant object. Then $\Lambda_n\overset{\rL}{\otimes}_{\Lambda}\rR\pi_*\sfK$ and
$\rR\pi_*(\rL\pi^*\Lambda_n\overset{\rL}{\otimes}_{\Lambda_\bullet}\sfK)$ are both represented by
$\pi_*\sfK'\xrightarrow{\times\lambda^{n+1}}\pi_*\sfK'$, where $\sfK'$ is a fibrant replacement of $\sfK$, and $\alpha$ is represented by the identity. Thus, $\alpha$ is an isomorphism.

For $\gamma$, consider the diagram
\[
\xymatrix{
(\dN_{\ge n},\Lambda_{\bullet,\ge n})\ar[r]^-{j}\ar[d] &(\dN,\Lambda_\bullet)\ar[d]\\
(\dN_{\ge n},\lambda_{\ge n}) \ar[r]^-{j'}\ar[rd]_{\pi'_{\ge n}} & (\dN,\lambda)\ar[d]^{\pi'}\\ &\Lambda,}
\]
where $\lambda$ is the constant ring with value $\Lambda$. By the cofinality of $\dN_{\ge n}$ in $\dN$, the natural transformation $\pi'_*\to(\pi'_{\ge n})_*\circ {j'}^*$ is an isomorphism. Since $j'_*$ admits an \emph{exact} left adjoint, it follows that $\rR\pi'_*\to\rR(\pi'_{\ge n})_*\circ{j'}^*$ is an isomorphism. Thus the natural transformation $\rR\pi_*\to\rR\pi_{\ge n}\circ j^*$ is an isomorphism. Therefore, $\gamma$ is an isomorphism.

The morphism $\delta$ is induced by the morphism
\[
\pi_{\ge n}^*\Lambda_n\overset{\rL}{\otimes}_{\Lambda_{\bullet,\ge n}}\sfK_{\ge n}\to e_{n*}e_n^*(\pi_{\ge n}^*\Lambda_n\overset{\rL}{\otimes}_{\Lambda_{\bullet,\ge n}}\sfK_{\ge n})\simeq e_{n*}\sfK_n
\]
which is an isomorphism since $\sfK$ is adic.

As summary, $\epsilon$ is an isomorphism if and only if $\beta$ is. By the above resolution of $\Lambda$, the cone of $\rL\pi^*\Lambda_n\to\Lambda_n$ is $G^n_\bullet[-2]$, where $G^n_m\coloneqq\Lambda/\fm^{\min(m,n)+1}$ and the transition maps are multiplication by $\lambda$, so that $G^0_\bullet =F_\bullet$. Thus, if $\beta$ is an isomorphism for $n=0$, then $\rR\pi_*(F_\bullet\overset{\rL}\otimes_{\Lambda_\bullet}\sfK)=0$. For $n\ge 1$, $G^n_\bullet$ is an extension of $F_\bullet$ by $G^{n-1}_{\bullet+1}$. Thus, if $\rR\pi_*(F_\bullet\overset{\rL}\otimes_{\Lambda_\bullet}\sfK)=0$, then, by the
above reason, $\beta$ is an isomorphism for all $n\in\dN$.
\end{proof}

\subsection{Constructible adic complexes}
\label{2ss:constructible_adic}

In this section, we fix an object $(\Lambda,\fm)$ of $\PRing$ such that $\Lambda/\fm^{n+1}$ is Noetherian for all $n$.

For a higher Artin stack $X\in\Chpar{}$, we put
\begin{align*}
\cD(X,\Lambda_\bullet)_{\ra,\rc}&\coloneqq\cD(X,\Lambda_\bullet)_\ra\cap \cD_\cons(X,\Lambda_\bullet),\\
\cD(X,\Lambda_\bullet)_{\ra,\rc}^{(+)}&\coloneqq\cD(X,\Lambda_\bullet)_\ra\cap\cD_\cons^{(+)}(X,\Lambda_\bullet),\\
\cD(X,\Lambda_\bullet)_{\ra,\rc}^{(-)}&\coloneqq\cD(X,\Lambda_\bullet)_\ra\cap\cD_\cons^{(-)}(X,\Lambda_\bullet),\\
\cD(X,\Lambda_\bullet)_{\ra,\rc}^{(\rb)}&\coloneqq\cD(X,\Lambda_\bullet)_\ra\cap\cD_\cons^{(\rb)}(X,\Lambda_\bullet).
\end{align*}
Note that we always have $\cD(X,\Lambda_\bullet)_{\ra,\rc}^{(-)}=\cD^{(-)}(X,\Lambda_\bullet)_\ra\cap\cD_\cons(X,\Lambda_\bullet)$. By Remark \ref{2re:bound}, if $(X,(\Lambda,\fm))$ is an admissible pair, then we have
\begin{align*}
\cD(X,\Lambda_\bullet)_{\ra,\rc}^{(+)}&=\cD^{(+)}(X,\Lambda_\bullet)_\ra\cap\cD_\cons(X,\Lambda_\bullet),\\
\cD(X,\Lambda_\bullet)_{\ra,\rc}^{(\rb)}&=\cD^{(\rb)}(X,\Lambda_\bullet)_\ra\cap\cD_\cons(X,\Lambda_\bullet)
\end{align*}
as well.

The following proposition is an immediate consequence of the above definitions and \cite{LZ1}*{6.4.4}.

\begin{proposition}\label{2le:constructible_adic}
Let $f\colon Y\to X$ be a morphism of higher Artin stacks. Then $f^*$ and $-\otimes_X-$ restrict to the following functors:
\begin{description}
  \item[1L'] $f^{*\ra}\colon\cD(X,\Lambda_\bullet)_{\ra,\rc}\to\cD(Y,\Lambda_\bullet)_{\ra,\rc}$.

  \item[3L'] $-\atimes_X-\colon\cD(X,\Lambda_\bullet)^{(-)}_{\ra,\rc}\times\cD(X,\Lambda_\bullet)^{(-)}_{\ra,\rc}
      \to\cD(X,\Lambda_\bullet)^{(-)}_{\ra,\rc}$.
\end{description}
In particular, we have a symmetric monoidal subcategory $(\cD(X,\Lambda_\bullet)^{(-)}_{\ra,\rc})^\otimes$ of $\cD(X,\Lambda_\bullet)^\otimes_\ra$.
\end{proposition}

As in \cite{LZ1}*{\Sec6.4}, to state the results for the other operations, we work in a relative setting. Let $\dS$ be a $\Box$-coprime higher Artin stack. Assume that there exists an atlas $S\to\dS$, where $S$ is either a quasi-excellent scheme or a regular scheme of dimension $\le 1$. Combining \cite{LZ1}*{6.4.4, 6.4.5} and Propositions \ref{2pr:hom_normalize}, \ref{2pr:star_pushforward_normalize}, \ref{2pr:shrink_pullback_normalize}, we have the following two propositions.

\begin{proposition}\label{2pr:constructible_adic}
Suppose $(\Lambda,\fm)\in\PRing_{\ltor}$. Let $f\colon Y\to X$ be a morphism of $\Chpars$. Then $f_!$, $f_*$, $f^!$, $\HOM_X$ restrict
to the following functors:
\begin{description}
  \item[2L'] $f_{!\ra}\colon\cD(Y,\Lambda_\bullet)^{(-)}_{\ra,\rc}\to\cD(X,\Lambda_\bullet)^{(-)}_{\ra,\rc}$ if $f$ is of finite presentation (see \cite{LZ1}*{5.4.3} for the definition), and $f_{!\ra}\colon\cD(Y,\Lambda_\bullet)_{\ra,\rc}\to\cD(X,\Lambda_\bullet)_{\ra,\rc}$ if $f$ is of finite presentation and $0$-Artin.

  \item[1R'] $f_{*\ra}\colon\cD(Y,\Lambda_\bullet)^{(+)}_{\ra,\rc}\to\cD(X,\Lambda_\bullet)^{(+)}_{\ra,\rc}$ if $f$ is quasi-compact and quasi-separated (see \cite{LZ1}*{5.4.3} for the definition).

  \item[2R'] $f^{!\ra}\colon\cD(X,\Lambda_\bullet)^{(+)}_{\ra,\rc}\to\cD(Y,\Lambda_\bullet)^{(+)}_{\ra,\rc}$.

  \item[3R'] $\HOM^\ra_X\colon(\cD(X,\Lambda_\bullet)^{(\rb)}_{\ra,\rc})^{op}\times\cD(X,\Lambda_\bullet)^{(+)}_{\ra,\rc}
      \to\cD(X,\Lambda_\bullet)^{(+)}_{\ra,\rc}$.
\end{description}
\end{proposition}

\begin{proposition}\label{2pr:un_constructible_adic}
Suppose $(\Lambda,\fm)\in\PRing_{\ltor}$. Let $f\colon Y\to X$ be a morphism of $\Chpars$ such that both $(X,(\Lambda,\fm))$ and $(Y,(\Lambda,\fm))$ are admissible (Definition \ref{2de:admissible}). Then
\begin{enumerate}
  \item $f_*$ restricts to a functor
      \[
      f_{*\ra}\colon\cD(Y,\Lambda_\bullet)_{\ra,\rc}\to\cD(X,\Lambda_\bullet)_{\ra,\rc}
      \]
      if $\dS$ is locally finite-dimensional and $f$ is quasi-compact and quasi-separated and $0$-Artin;

  \item $f^!$ restricts to a functor
      \[
      f^{!\ra}\colon\cD(X,\Lambda_\bullet)_{\ra,\rc}\to\cD(Y,\Lambda_\bullet)_{\ra,\rc}
      \]
      if $\dS$ is locally finite-dimensional;

  \item $\HOM_X$ restricts to a functor
      \[
      \HOM^\ra_X\colon(\cD(X,\Lambda_\bullet)^{(-)}_{\ra,\rc})^{op}\times\cD(X,\Lambda_\bullet)^{(+)}_{\ra,\rc}
      \to\cD(X,\Lambda_\bullet)^{(+)}_{\ra,\rc}.
      \]
\end{enumerate}
\end{proposition}

Let $X$ be a scheme in $\Schqcs$. Recall that a complex $\sfK\in\cD(X,\Lambda_\bullet)$ is a \emph{$\lambda$-complex} \cite{LO2}*{3.0.6} if $\rH^n\sfK$ is constructible and almost adic. In particular, we have $\sfK\in\cD_\cons(X,\Lambda_\bullet)$. The proofs of the following
statements are similar to \cite{Zh}.

\begin{lem}\label{2le:lambda_complex}
Let $X$ be a scheme in $\Schqcs$ such that Condition (3) in Proposition \ref{2pr:five_condition} holds for the pair $(X,(\Lambda,\fm))$. Let
$\cD(X,\Lambda_{\bullet})_{\r{en},\rc}$ be the full subcategory of $\cD(X,\Lambda_\bullet)_{\r{en}}$ spanned by $\lambda$-complexes. We have
\begin{enumerate}
   \item $\cD(X,\Lambda_{\bullet})_{\r{en},\rc}$ is closed under the truncation functors $\tau^{\geq n}$ and $\tau^{\leq n}$.

   \item The essential image of $\rL\pi^*\rR\pi_*\cD(X,\Lambda_{\bullet})_{\r{en},\rc}$ coincides with $\cD(X,\Lambda_\bullet)_{\ra,\rc}$.
\end{enumerate}
\end{lem}

\begin{proposition}
Let the assumptions be as in the above lemma.
\begin{enumerate}
  \item Put $\cD^{\leq n}(X,\Lambda_\bullet)_{\ra,\rc}\coloneqq\cD^{\leq n}(X,\Lambda_\bullet)_\ra\cap\cD_{\cons}(X,\Lambda_\bullet)$. Then the right perpendicular full subcategory $\cD^{\geq n}(X,\Lambda_{\bullet})_{\ra,\rc}$ of $\cD^{\leq n-1}(X,\Lambda_{\bullet})_{\ra,\rc}$ in $\cD(X,\Lambda_{\bullet})_{\ra,\rc}$ coincides with the essential image of
      $\rL\pi^*\rR\pi_*(\cD(X,\Lambda_{\bullet})_{\r{en},\rc}\cap\cD^{\geq n}(X,\Lambda_\bullet)_{\r{en}})$.

  \item For truncation functors on $\cD(X,\Lambda_\bullet)_\ra$, we have
      \[
      \tau_\ra^{\leq n}\simeq\rL\pi^*\circ\rR\pi_*\circ\tau^{\leq n},\quad\tau_\ra^{\geq n}\simeq\rL\pi^*\circ\rR\pi_*\circ\tau^{\geq n}.
      \]
\end{enumerate}
\end{proposition}

\begin{corollary}\label{2co:constructible_adic_t}
Let $X\in\Chpar{}$ be a higher Artin stack that is locally $(\Lambda/\fm)$-bounded. Then the full subcategory $\cD(X,\Lambda_\bullet)_{\ra,\rc}$ is preserved under the truncation functors $\tau_\ra^{\leq n}$ hence $\tau_\ra^{\geq n}$ on $\cD(X,\Lambda_\bullet)_\ra$.
\end{corollary}

\subsection{Compatibility with Laszlo--Olsson}
\label{2ss:compatibility}

We prove the compatibility between our adic formalism and Laszlo--Olsson's \cite{LO2}, under their assumptions.

Put $\Box=\{\ell\}$ where $\ell$ is a rational prime. Let $\dS$ be a $\Box$-coprime scheme satisfying that
\begin{enumerate}
  \item it is affine excellent and finite-dimensional;

  \item for every scheme $X$ of finite type over $\dS$, there exists an \'{e}tale cover $X'\to X$ such that $\r{cd}_{\ell}(Y)<\infty$\footnote{According to our notation, $\r{cd}_{\ell}$ is nothing but $\r{cd}_{\d{F}_{\ell}}$.} for every
      scheme $Y$ \'{e}tale and of finite type over $X'$;

  \item it admits a global dimension function and we fix such a function (see \cite{LZ1}*{6.5.1}).
\end{enumerate}

Recall from \cite{LZ1}*{\Sec6.5} that we denote $\Chplmb{\dS}$ the full subcategory of $\Chplft{\dS}$ spanned by ($1$-)Artin stacks locally of finite type over $S$, with quasi-compact and separated diagonal.

For the coefficient, we fix a complete discrete valuation ring $\Lambda$ with the maximal ideal $\fm$ and residue characteristic $\ell$ such that $\Lambda=\lim_n\Lambda_n$, where $\Lambda_n=\Lambda/\f{m}^{n+1}$, as in \cite{LO2}. In particular, $(\Lambda,\fm)$ is an object of $\PRing$ in our notation. For every stack $\cX\in\Chplmb{\dS}$, the pair $\cX$ is locally $(\Lambda/\fm)$-bounded.

From the definition of $\cD(\cX,\Lambda_\bullet)_{\ra,\rc}$, which is the full subcategory of $\cD(\cX,\Lambda_\bullet)$ spanned by
constructible adic complexes, \cite{LO2}*{3.0.10, 3.0.14, 3.0.18}, and \cite{LZ1}*{5.3.5}, we have a canonical equivalence between categories
\begin{align}\label{2eq:compatibility}
\rh\cD(\cX,\Lambda_\bullet)_{\ra,\rc}\simeq\bD_c(\cX,\Lambda),
\end{align}
where the latter one is defined in \cite{LO2}*{3.0.6}.

\begin{proposition}
For a morphism $f\colon\cY\to\cX$ of finite type in $\Chplmb{\dS}$, there are natural isomorphisms of functors:
\begin{align*}
\rh f^{*\ra}\simeq\rL f^*&\colon\bD_c(\cX,\Lambda)\to\bD_c(\cY,\Lambda),\\
\rh f_{*\ra}\simeq\rR f_*&\colon\bD_c^{(+)}(\cY,\Lambda)\to\bD_c^{(+)}(\cX,\Lambda),\\
\rh f_{!\ra}\simeq\rR f_!&\colon\bD_c^{(-)}(\cY,\Lambda)\to\bD_c^{(-)}(\cX,\Lambda),\\
\rh f^{!\ra}\simeq\rR f^!&\colon\bD_c(\cX,\Lambda)\to\bD_c(\cY,\Lambda),\\
\rh(-\atimes_\cX-)\simeq(-)\overset{\bL}{\otimes}(-)&\colon\bD_c^{(-)}(\cX,\Lambda)\times\bD_c^{(-)}(\cX,\Lambda)\to\bD_c^{(-)}(\cX,\Lambda),\\
\rh\HOM^\ra_\cX\simeq\b{\s{R}hom}_\Lambda&\colon\bD_c^{(-)}(\cX,\Lambda)^{\r{opp}}\times\bD_c^{(+)}(\cX,\Lambda)\to\bD_c^{(+)}(\cX,\Lambda)
\end{align*}
that are compatible with \eqref{2eq:compatibility}. Here, on the right side of the equivalences, we adopt notation from \cite{LO2}*{\Sec1}.
\end{proposition}

By Lemma \ref{2le:constructible_adic} and Proposition \ref{2pr:constructible_adic}, the six operations on the left side in the above proposition do have the correct range.

\begin{proof}
The isomorphisms for tensor product, internal Hom and $f^*$ simply follow from the same definitions here and in \cite{LO2}*{\Sec4, \Sec6}. The
isomorphism for $f_*$ follows from the adjunction and that for $f^*$ (\cite{LZ1}*{6.3.2}). The isomorphism for $f_!$ will follows from the
adjunction and that for $f^!$ which will be proved below.

By the compatibility of dualizing complexes and the isomorphisms for internal Hom, we have natural isomorphisms $\rD^\ra_\cX\simeq\rD_\cX$ and $\rD^\ra_\cY\simeq\rD_\cY$ (Definition \ref{1de:dualizing}). Therefore, by \cite{LO2}*{9.1}, to show the isomorphism for $f^!$, we only need to show that our functors satisfy
\[
\rh f^{!\ra}\simeq\rD^\ra_\cY\circ\rh f^{*\ra}\circ\rD^\ra_\cX.
\]
Note that for every $\sfK\in\bD_c(\cX,\Lambda)$, the biduality map $\delta^\ra_{\Omega_\cX}(\sfK)\colon\sfK\to\rD^\ra_\cX(\rD^\ra_\cX(\sfK))$ is an isomorphism by \cite{LO2}*{Theorem 7.3.1}. Thus, we have
\begin{align*}
\rh f^{!\ra}\sfK&\simeq\rh f^{!\ra}(\rD^\ra_\cX(\rD^\ra_\cX(\sfK)))\\
&=\rh f^{!\ra}(\rh\HOM^\ra_\cX(\rh\HOM^\ra_\cX(\sfK,\Omega_\cX),\Omega_\cX)) \\
&\simeq \rh\fR_\cY(\rh f^!(\rh\HOM_\cX(\rh\HOM^\ra_\cX(\sfK,\Omega_\cX),\Omega_\cX)))  \\
&\simeq \rh\fR_\cY(\rh\HOM_\cY(\rh f^*(\rh\HOM^\ra_\cX(\sfK,\Omega_\cX),f^!\Omega_\cX)))  \\
&\simeq \rh\HOM^\ra_\cY(\rh f^{*\ra}(\rh\HOM^\ra_\cX(\sfK,\Omega_\cX),\Omega_\cY))\\
&=\rD^\ra_\cY(\rh f^{*\ra}(\rD^\ra_\cX(\sfK))).
\end{align*}
The proposition is proved.
\end{proof}

\begin{remark}
In view of the above compatibility, we prove all the expected properties of the six operations, in particular the Base Change Theorem, in the adic case of Laszlo--Olsson \cite{LO2}.
\end{remark}

\section{Perverse t-structures}
\label{3}

In this chapters, we study perverse t-structures for stacks. In \Sec\ref{3ss:perversity_evaluation}, we define the notion of perversity evaluations on stacks, which we will associate t-structures. In \Sec\ref{3ss:perverse_t}, we construct the perverse t-structure with respect to a perverse evaluation. In \Sec\ref{3ss:adic_perverse_t}, we construct perverse t-structures for in the adic case. In \Sec\ref{3ss:constructible_adic_perverse_t}, we study constructibility under perverse truncations in the adic case.

\subsection{Perversity evaluations}
\label{3ss:perversity_evaluation}

We first recall various notion of perversity functions on schemes, introduced by Gabber.

\begin{definition}
Let $X$ be a scheme in $\Schqcs$. Denote by $|X|$ the underlying topological space of $X$.
\begin{enumerate}
  \item Following \cite{Gabber}*{\Sec 1}, a \emph{weak perversity function} on $X$ is a function
      \[
      p\colon|X|\to\dZ\cup\{+\infty\}
      \]
      such that for every $n\in\dZ$, the set $\{x\in|X|\res p(x)\geq n\}$ is ind-constructible.

  \item An \emph{admissible perversity function} on $X$ is a weak perversity function $p$ such that for every $x\in|X|$, there is an open dense subset $U\subseteq\overline{\{x\}}$ satisfying the condition that for every $x'\in U$, $p(x')\leq p(x)+2\codim(x',x)$.

  \item A \emph{codimension perversity function} on $X$ is a function $p\colon|X|\to\dZ\cup\{+\infty\}$ such that for every immediate \'etale specialization $x'$ of $x$, $p(x')=p(x)+1$.
\end{enumerate}
\end{definition}

\begin{remark}
We have the following remarks concerning perversity functions.
\begin{enumerate}
  \item A weak perversity function on a locally Noetherian scheme is locally bounded from below.

  \item An admissible perversity function on a scheme that is locally Noetherian and of finite dimension is locally bounded from above.

  \item A codimension perversity function on a scheme is \emph{not} necessarily a weak perversity function.

  \item A codimension perversity function that is also a weak perversity function is an admissible perversity function. If $X$ is locally
      Noetherian, then a codimension perversity function is a weak perversity function and hence an admissible perversity function.

  \item A codimension perversity function is the opposite of a dimension function in the sense of \cite{PS}*{2.1.8}. If $X$ is locally
      Noetherian and admits a dimension function, then $X$ is universally catenary by \cite{PS}*{2.2.6}. In this case, immediate \'etale
      specializations coincide with immediate Zariski specializations \cite{PS}*{2.1.4}.

  \item If $p$ is a weak (resp.\ admissible, resp.\ codimension) perversity function on $X$ and $d\colon|X|\to\dZ\cup\{+\infty\}$ is a locally constant function, then $p+d$ is a weak (resp.\ admissible, resp.\ codimension) perversity function on $X$.
\end{enumerate}
\end{remark}

\begin{definition}\label{3de:moderate}
A function $q\colon\dN\to\dZ$ or $q\colon\dZ\to\dZ$ is called \emph{moderate} if $q$ and $\b{2}-q$ are both increasing. Here, $\b{2}$ is the function $\b{2}(x)=2x$ and similarly for $\b{0}$ and $\b{1}$, which will be used below.
\end{definition}

\begin{notation}\label{3no:pullback}
Let $f\colon Y\to X$ be a smooth morphism of schemes in $\Schqcs$. For functions $p\colon|X|\to\dZ\cup\{+\infty\}$ and $q\colon\dN\to\dZ$, we define the $q$-weighted pullback $f^*_qp\colon|Y|\to\dZ\cup\{+\infty\}$ by
\[
(f^*_qp)(y)=p(f(y))-q(\trdeg[k(y):k(f(y))])
\]
for every point $y\in|Y|$. In particular, we have $f_{\b{0}}^*p=p\circ f$.
\end{notation}

The following two lemmas study some stability properties of weighted pullbacks of perversity functions.

\begin{lem}\label{3le:perverse_zero}
Let $f\colon Y\to X$ be a morphism (resp.\ \'etale morphism, resp.\ \'etale morphism) of schemes in $\Schqcs$. If $p$ is a weak (resp.\ admissible, resp.\ codimension) perversity function on $X$, then $f^*_{\b{0}}p$ is a weak (resp.\ admissible, resp.\ codimension) perversity function on $Y$.
\end{lem}

\begin{proof}
We have $f^*_{\b{0}}p=p\circ f$. If $p$ is a weak perversity function, then
\[
\{y\in|Y| \mid f^*_{\b{0}}p(y)\geq n\}=f^{-1}(\{x\in|X| \mid p(x) \geq n\})
\]
is ind-constructible by \cite{EGAIV}*{1.9.5 (vi)}. The other two cases follow from the trivial fact that $\codim(y',y)=\codim(f(y'),f(y))$ for
every specialization $y'$ of $y$ on $Y$.
\end{proof}

\begin{lem}\label{3le:perverse_pullback}
Let $f\colon Y\to X$ be a morphism of locally Noetherian schemes in $\Schqcs$, locally of finite type.
\begin{enumerate}
  \item Let $p$ be a weak perversity function on $X$ and $q\colon\dN\to\dZ$ an increasing function. Then $f^*_qp$ is a weak perversity function on $Y$.

  \item Let $p$ be an admissible perversity function on $X$ an $q\colon\dN\to\dZ$ a moderate function (Definition \ref{3de:moderate}). Then $f^*_qp$ is an admissible perversity function on $Y$.

  \item Let $p$ be a codimension perversity function on $X$. Then $f^*_{\b{1}}p$ is a codimension perversity function on $Y$.
\end{enumerate}
\end{lem}

\begin{proof}
For a locally closed subset $Z$ of a scheme $X$, we endow it with the reduced induced subscheme structure. For every point $y\in|Y|$, let
$U_y\subset \overline{\{y\}}$ be a nonempty open subset such that the induced morphism $f_y\colon\overline{\{y\}}\to\overline{\{f(y)\}}$ is
flat. Such an open subset exists by \cite{EGAIV}*{6.9.1}. For $y'\in U_y$, we have
\begin{align*}
\delta(y',y)&\coloneqq\trdeg[k(y):k(f(y))]-\trdeg[k(y'):k(f(y'))]\\
&=\codim(y',U_y\times_{\overline{f(y)}}\{f(y')\})\geq 0
\end{align*}
by \cite{EGAIV}*{14.3.13} since $f_y$ is universally open \cite{EGAIV}*{2.4.6}.

For (1), we know that for every $n\in\dZ$,
\[
\{y\in|Y|\res f^*_qp(y)\geq n\}=\bigcup_{y\in|Y|}f^{-1}\left\{x\in|X|\mid p(x)\geq n+q(\trdeg[k(y):k(f(y))])\right\}\cap U_y
\]
is a union of ind-constructible subsets, and hence is itself ind-constructible. In other words, $f_q^*p$ is a weak perversity function.

For (2), let $y\in|Y|$ be a point; put $x=f(y)$; and let $U_x\subset\overline{\{x\}}$ be a dense open subset such that $p(x')\leq p(x)+2\codim(x',x)$ for every $x'\in U_x$. We prove that for $y'\in U_y\cap f^{-1}(U_x)$,
\[
f^*_qp(y')\leq f^*_qp(y)+2\codim(y',y)
\]
holds. We may assume $p(x)\in\dZ$. Put $x'=f(y')$. We have
\[
f^*_qp(y)=p(x)-q(\trdeg[k(y):k(x)])
\]
and
\[
f^*_qp(y')=p(x')-q(\trdeg[k(y'):k(x')]).
\]
 Moreover, by \cite{EGAIV}*{6.1.2}, we have
\[
\delta(y',y)=\codim(y',y)-\codim(x',x).
\]
Therefore, we have
\begin{align*}
f^*_qp(y')-f^*_qp(y)&=p(x')-p(x)+q(\trdeg[k(y):k(x)])-q(\trdeg[k(y'):k(x')])\\
&\leq2\codim(x',x)+2\delta(y',y)=2\codim(y',y)
\end{align*}
since $q$ is moderate. In other words, $f_q^*p$ is an admissible perversity function on $Y$.

For (3), it is essentially proved in \cite{PS}*{2.5.2}.
\end{proof}

Now we generalize the notion of perversity functions from schemes to stacks, by starting from the following definition.

\begin{definition}[Pointed schematic neighborhood]
Let $X$ be a higher Artin (resp.\ Deligne--Mumford) stack. A \emph{pointed smooth (resp.\ \'etale) schematic neighborhood} of $X$ is a triple
$(X_0,u_0,x_0)$ where $u_0\colon X_0\to X$ is a smooth (resp.\ an \'{e}tale) morphism with $X_0\in\Schqcs$ and $x_0\in|X_0|$ a scheme-theoretical point. A morphism $v\colon (X_1,u_1,x_1)\to(X_0,u_0,x_0)$ of pointed smooth (resp.\ \'etale) schematic neighborhoods is a smooth (resp.\ an \'{e}tale) morphism $v\colon X_1\to X_0$ such that there is a triangle
\begin{align}\label{3eq:schematic_neighborhood}
\xymatrix{
X_1 \ar[rd]_-{u_1} \ar[rr]^-{v} & & X_0 \ar[ld]^-{u_0} \\
& X }
\end{align}
with $v(x_1)=x_0$. We say that $(X_1,u_1,x_1)$ \emph{dominates} $(X_0,u_0,x_0)$ if there is such a morphism. The category of pointed smooth (resp.\ \'etale) schematic neighborhoods of $X$ is denoted by $\Vosm(X)$ (resp.\ $\Voet(X)$).
\end{definition}

\begin{lem}\label{3le:codimension}
Let $X$ be a higher Artin stack, and let $v\colon(X_1,u_1,x_1)\to(X_0,u_0,x_0)$ be a morphism of pointed smooth schematic neighborhoods of $X$. Then the codimension of $x_1$ in the base change scheme $X_{1,x_0}=X_1\times_{X_0}\{x_0\}$ depends only on the source and the target of $v$.
\end{lem}

\begin{proof}
Note that $\codim(x_1,X_{1,x_0})=\dim_{x_1}(v)-\trdeg[k(x_1):k(x_0)]$. It is clear that the term $\dim_{x_1}(v)=\dim_{x_1}(u_1)-\dim_{x_0}(u_0)$ does not depend on $v$. We will show that the other term $\trdeg[k(x_1):k(x_0)]$ does not depend on $v$ either.

Let $f\colon Y\to X$ be an atlas of $X$ with $Y$ a scheme in $\Schqcs$. Let
\[
\xymatrix{
Y_1 \ar[rr]^-{v'} \ar[dr]_-{u'_1} && Y_0 \ar[dl]^-{u'_0} \\ & Y }
\]
be the base change of \eqref{3eq:schematic_neighborhood}, and $f_0\colon Y_0\to X_0$, $f_1\colon Y_1\to X_1$ the induced morphisms. Let $w_0\colon Y'_0\to Y_0$ be an atlas with $Y'_0$ a scheme in $\Schqcs$, and let
\[
\xymatrix{
Y'_1 \ar[r]^-{v''} \ar[d]_-{w_1} & Y'_0 \ar[d]^-{w_0} \\
Y_1 \ar[r]^-{v'} & Y_0 }
\]
be the base change. Then $v''$ is a smooth morphism of schemes in $\Schqcs$. Since $f_0\circ w_0\colon Y'_0\to X_0$ is smooth and surjective, the base change scheme $Y'_{0,x_0}=Y'_0\times_{X_0}\{x_0\}$ is nonempty and smooth over the residue field $k(x_0)$ of $x_0$. Similarly, we have a nonempty scheme $Y'_{1,x_1}$, smooth over $k(x_1)$. Choose a generic point $y'_1$ of $Y'_{1,x_1}$. Then its image $y'_0$ in $Y'_{0,x_0}$ is a generic point. Let $y$ be the image of $y'_0$ in $Y$. Then we have
\[
\trdeg[k(x_1):k(x_0)]=\r{tr.deg}[k(y'_1):k(y)]-\trdeg[k(y'_0):k(y)]
\]
which does \emph{not} depend on $v$. The lemma follows.
\end{proof}

\begin{notation}\label{3no:codimension}
Let $X$ be a higher Artin stack, and let $v\colon(X_1,u_1,x_1)\to(X_0,u_0,x_0)$ be a morphism of pointed smooth schematic neighborhoods of $X$. We will denote by $\delta^{(X_1,u_1,x_1)}_{(X_0,u_0,x_0)}$ the codimension appeared in Lemma \ref{3le:codimension}. It is clear that
\[
\delta^{(X_2,u_2,x_2)}_{(X_0,u_0,x_0)}=\delta^{(X_2,u_2,x_2)}_{(X_1,u_1,x_1)}+\delta^{(X_1,u_1,x_1)}_{(X_0,u_0,x_0)}
\]
if $(X_2,u_2,x_2)$ dominates $(X_1,u_1,x_1)$. Moreover, if $v$ is \'etale, then we have $\delta^{(X_1,u_1,x_1)}_{(X_0,u_0,x_0)}=0$.
\end{notation}

\begin{notation}
For a higher Artin (resp.\ Deligne--Mumford) stack $X$ and a function $\sfp\colon\Ob(\Vosm(X))\to\dZ\cup\{+\infty\}$ (resp.\ $\sfp\colon\Ob(\Voet(X))\to\dZ\cup\{+\infty\}$), we have, by restriction, the function $\sfp_{u_0}\colon|X_0|\to\dZ\cup\{+\infty\}$ for every smooth (resp.\ \'{e}tale) morphism $u_0\colon X_0\to X$ with $X_0$ in $\Schqcs$.

If $f\colon Y\to X$ is a smooth (resp.\ an \'{e}tale) morphism of higher Artin (resp.\ Deligne--Mumford) stacks, then composition with $f$ induces a functor $f\colon \Vosm(Y)\to\Vosm(X)$ (resp.\ $f\colon \Voet(Y)\to\Voet(X)$), and we put $f^*\sfp=\sfp\circ f$.
\end{notation}

\begin{definition}[(admissible/codimension) perversity evaluations]\label{3de:perversity_evaluation}
Let $X$ be a higher Artin stack. A \emph{smooth evaluation} on $X$ is a function
\[
\sfp\colon\Ob(\Vosm(X))\to\dZ\cup\{+\infty\}
\]
such that for $(X_1,u_1,x_1)$ dominating $(X_0,u_0,x_0)$, we have
\[
\sfp(X_0,u_0,x_0) \le\sfp(X_1,u_1,x_1)\le\sfp(X_0,u_0,x_0)+2\delta^{(X_1,u_1,x_1)}_{(X_0,u_0,x_0)}.
\]

A \emph{perversity smooth evaluation} (resp.\ \emph{admissible perversity smooth evaluation}, \emph{codimension perversity smooth evaluation}) on $X$ is a smooth evaluation $\sfp$ such that for every $(X_0,u_0,x_0)\in\Ob(\Vosm(X))$, $\sfp_{u_0}$ is a weak perversity function (resp.\ admissible perversity function, codimension perversity function) on $X_0$.

Similarly, we define \'etale evaluations and (admissible/codimension) perversity \'etale evaluations on a higher Deligne--Mumford stack $X$ using $\Voet(X)$.

We say that a smooth (resp.\ \'etale) evaluation $\sfp$ is \emph{locally bounded} if for every smooth (resp.\ \'etale) morphism $u_0\colon X_0\to X$ with $X_0$ a quasi-compact separated scheme, $\sfp_{u_0}$ is bounded.
\end{definition}

\begin{remark}
If $X$ is a scheme in $\Schqcs$, then the map from the set of \'etale evaluations on $X$ to the set of functions $|X|\to\dZ\cup\{+\infty\}$, carrying $\sfp$ to $\sfp_{\id_X}$, is bijective. Under such bijection, the notions of (weak) perversity, admissible perversity, and
codimension perversity coincide. If $f\colon Y\to X$ is a morphism of schemes in $\Schqcs$, then $f^*$ for \'etale evaluations coincide with $f^*_{\b{0}}$ for functions.
\end{remark}

\begin{example}
We have the following examples of perversity smooth/\'etale evaluations.
\begin{enumerate}
  \item Let $X$ be a higher Artin (resp.\ Deligne--Mumford) stack. Then every constant smooth (resp.\ \'etale) evaluation is an admissible
      perversity smooth (resp.\ \'etale) evaluation.

  \item Let $f\colon Y\to X$ be a morphism of higher Deligne--Mumford stacks, locally of finite type. Let $\sfp$ be an \'etale evaluation on $X$, and $q\colon\dN\to\dZ$ a function. We define an \'etale evaluation $f^*_q\sfp$ on $Y$ as follows. For any object $(Y_0,v_0,y_0)$ of $\Voet(Y)$, there exists a morphism $(Y_1,v_1,y_1)\to(Y_0,v_0,y_0)$ in $\Voet(Y)$ such that there exists a diagram
      \[
      \xymatrix{Y_1\ar[r]^{v_1}\ar[d]_{f_0} & Y\ar[d]^f\\
      X_0\ar[r]^{u_0} & X,}
      \]
      where $X_0$ is in $\Schqcs$ and $u_0$ is \'etale. We put
      \[
      f^*_q\sfp(Y_0,v_0,y_0)=\sfp(X_0,u_0,f_0(y_1))-q(\trdeg[k(y_1):k(f_0(y_1))]).
      \]
      This clearly does not depend on choices. If $\sfp$ is a perversity \'etale evaluation, then $f^*_{\b{0}}\sfp$ is a perversity \'etale
      evaluation by Lemma \ref{3le:perverse_zero}.

  \item Let $f\colon Y\to X$ be a morphism of higher Artin stacks, locally of finite type, with $X$ being a higher Deligne--Mumford stack. Let $\sfp$ be an \'etale evaluation on $X$, and $q\colon\dZ\to\dZ$ a moderate function (Definition \ref{3de:moderate}). We define a smooth evaluation $f^*_q\sfp$ on $Y$ by the formula
      \[
      (f^*_q\sfp)(Y_0,v_0,y_0)=(v_0\circ f)^*_{q'}(y_0)
      \]
      for every object $(Y_0,v_0,y_0)$ of $\Vosm(Y)$, where $q'\colon\dN\to\dZ$ is the function $q'(n)=q(n-\dim_{y_0}(v_0))$. If $\sfp$ is a perversity \'etale evaluation, then $f^*_{\b{0}}\sfp$ is a perversity smooth evaluation. If $X$ is locally Noetherian and $\sfp$ is a perversity (resp.\ admissible perversity, resp.\ codimension perversity) \'etale evaluation, then $f^*_q\sfp$ (resp.\ $f^*_q\sfp$, resp.\ $f^*_{\b{1}}\sfp$) is a perversity (resp.\ admissible perversity, resp.\ codimension) smooth evaluation by Lemma \ref{3le:perverse_pullback}.
\end{enumerate}
\end{example}

\subsection{Perverse t-structures}
\label{3ss:perverse_t}

In this section, we define t-structures associated to perversity evaluations.

\begin{definition}\label{3de:complete}
Let $\cC$ be a stable $\infty$-category equipped with a t-structure. We say that $\cC$ is \emph{weakly left complete} (resp.\ \emph{weakly right complete}) if $\cC^{\le -\infty}=\bigcap_n \cC^{\le -n}$ (resp.\ $\cC^{\ge\infty}=\bigcap_n \cC^{\ge n}$) consists of zero objects.
\end{definition}

The family $(\rH^i)_{i\in \dZ}$ is conservative if and only if $\cC$ is both weakly left complete and weakly right complete (cf.\ \cite{BBD}*{1.3.7}). The following lemma slightly extends \cite{Lu2}*{1.2.1.19}.

\begin{lem}\label{3le:wlc}
Let $\cC$ be a stable $\infty$-category equipped with a t-structure. Consider the following conditions
\begin{enumerate}
  \item The $\infty$-category $\cC$ is left complete.

  \item The $\infty$-category $\cC$ is weakly left complete.
\end{enumerate}
Then (1) implies (2). Moreover, if $\cC$ admits countable products and there exists an integer $a$ such that countable products of objects of $\cC^{\le0}$ belong to $\cC^{\le a}$, then (2) implies (1).
\end{lem}

\begin{proof}
The first assertion is obvious since the image of $\cC^{\le -\infty}$ under the functor $\cC\to\widehat\cC$ consists of zero objects, where $\widehat\cC$ is defined above \cite{Lu2}*{1.2.1.17}.

To show the second assertion, it suffices to replace $f(n-1)$ by $f(n-a-1)$ in the proof of \cite{Lu2}*{1.2.1.19}.
\end{proof}

Let $X$ be a scheme in $\Schqcs$, let $p\colon|X|\to\dZ\cup\{+\infty\}$ be a function, and let $\lambda=(\Xi,\Lambda)$ be an object of $\Rind$. Following Gabber \cite{Gabber}*{\Sec 2}, we define full subcategories $\TS{p}{\cD}{\leq0}(X,\lambda),\TS{p}{\cD}{\geq0}(X,\lambda)\subseteq\cD(X,\lambda)$ as follows: For $\sfK\in\cD(X,\lambda)$,
\begin{itemize}
  \item $\sfK$ belongs to $\TS{p}{\cD}{\le 0}(X,\lambda)$ if and only if
      \[
      i^*_{\overline{x}}j^*_{\overline{x}}\sfK\in\cD^{\le p(x)}(\overline{x},\lambda)
      \]
      for every $x\in|X|$.

  \item $\sfK$ belongs to $\TS{p}{\cD}{\ge 0}(X,\lambda)$ if and only if $\sfK\in \cD^{(+)}(X,\lambda)$ and
      \[
      i^!_{\overline{x}}j^*_{\overline{x}}\sfK\in\cD^{\ge p(x)}(\overline{x},\lambda)
      \]
      for every $x\in|X|$.
\end{itemize}
Here $\overline{x}$ is a geometric point above $x$, and we have natural morphisms
\[
i_{\overline{x}}\colon\overline{x}\to X_{(\overline{x})},\quad j_{\overline{x}}\colon X_{(\overline{x})}\to X.
\]
We will omit $j_{\overline{x}}^*$ from the notation when no confusion arises.

\begin{lem}\label{3le:gabber}
If $p$ is a weak perversity function, then $(\TS{p}{\cD}{\leq 0}(X,\lambda),\TS{p}{\cD}{\geq 0}(X,\lambda))$ is a t-structure on $\cD(X,\lambda)$. Moreover,
\begin{enumerate}
  \item such t-structure is accessible;

  \item such t-structure is weakly left complete if $p$ takes values in $\dZ$;

  \item such t-structure is right complete;

  \item such t-structure is left complete if is locally bounded and every quasi-compact closed open subscheme of $X$ is $\lambda$-cohomologically finite. Here, we say that a scheme $Y$ is \emph{$\lambda$-cohomologically finite} if there exists an integer $n$ such that, for every $\xi\in\Xi$, the $\Lambda(\xi)$-cohomological dimension of the \'etale topos of $Y$ is at most $n$.
\end{enumerate}
\end{lem}

\begin{proof}
The proof of the t-structure is shown by Gabber \cite{Gabber} when $\Xi$ is a singleton. This generalizes
easily to the case of general $\Xi$ as follows. By \cite{Lu2}*{1.4.4.11}, there exists a t-structure $(\TS{p}{\cD}{\leq 0}(X,\lambda),\cD')$ on
$\cD(X,\lambda)$. For $\sfK\in\TS{p}{\cD}{\leq 0}(X,\lambda)$ and $\sfL\in\TS{p}{\cD}{\geq 0}(X,\lambda)$, we have $a_*\HOM(\sfK,\sfL[1])\in\cD^{\ge1}(\ast,\lambda)$, hence $\Hom(\sfK,\sfL[1])=\rH^0(\Xi,a_*\HOM(\sfK,\sfL[1]))=0$, where $a\colon X_{\et}\to\ast$ is the morphism of topoi. Thus we have $\TS{p}{\cD}{\geq 0}(X,\lambda)\subseteq \cD'$. For every $\xi\in\Xi$, the functor $\rL e_{\xi!}\colon\cD(X,\Lambda(\xi))\to\cD(X,\lambda)$ is left t-exact for the t-structures $(\TS{p}{\cD}{\leq0}(X,\Lambda(\xi)),\TS{p}{\cD}{\geq0}(X,\Lambda(\xi)))$ and $(\TS{p}{\cD}{\leq 0}(X,\lambda),\cD')$. It follows that $e_\xi^*$ is right t-exact for the same t-structures. Thus, we have $\cD'\subseteq\TS{p}{\cD}{\geq 0}(X,\lambda)$ as well.

For the properties, (1) and (2) follow from the definition directly; (3) follows from \cite{Gabber}*{3.1}; and (4) follows from Lemma \ref{3le:wlc}.
\end{proof}

Now we define t-structures for stacks associated to perversity evaluations. Let $X$ be a $\Box$-coprime higher Artin (resp.\ a higher Deligne--Mumford) stack equipped with a perversity smooth (resp.\ \'{e}tale) evaluation $\sfp$ (Definition \ref{3de:perversity_evaluation}), and let $\lambda$ be an object of $\Rind_{\ltor}$ (resp.\ $\Rind$). For an atlas (resp.\ \'{e}tale atlas) $u\colon X_0\to X$ with $X_0$ a scheme in $\Schqcs$, we denote by $\TS{\sfp}{\cD}{\leq0}_u(X,\lambda)\subseteq\cD(X,\lambda)$ (resp.\ $\TS{\sfp}{\cD}{\geq0}_u(X,\lambda)\subseteq\cD(X,\lambda)$) the full subcategory spanned by complexes $\sfK$ such that $u^*\sfK$ is in $\TS{\sfp_u}{\cD}{\leq0}(X_0,\lambda)$ (resp.\ $\TS{\sfp_u}{\cD}{\geq0}(X_0,\lambda)$).

\begin{proposition}\label{3pr:perverse_t}
Let $X$ be a $\Box$-coprime higher Artin (resp.\ a higher Deligne--Mumford) stack equipped with a perversity smooth (resp.\ \'{e}tale)
evaluation $\sfp$, and let $\lambda$ be an object of $\Rind_{\ltor}$ (resp.\ $\Rind$). Then
\begin{enumerate}
  \item The pair of subcategories $(\TS{\sfp}{\cD}{\leq0}_u(X,\lambda),\TS{\sfp}{\cD}{\geq0}_u(X,\lambda))$ do not depend on the choice of $u$. We will denote them by $(\TS{\sfp}{\cD}{\leq0}(X,\lambda),\TS{\sfp}{\cD}{\geq0}(X,\lambda))$.

  \item The pair of subcategories $(\TS{\sfp}{\cD}{\leq0}(X,\lambda),\TS{\sfp}{\cD}{\geq0}(X,\lambda))$ determine a right complete accessible t-structure on $\cD(X,\lambda)$, which is weakly left complete if $\sfp$ takes values in $\dZ$. Such t-structure is left complete if $\sfp$ is locally bounded and if for every smooth (resp.\ \'etale) morphism $X_0\to X$ with $X_0$ a quasi-compact separated scheme, $X_0$ is $\lambda$-cohomologically finite.

  \item If $f\colon Y\to X$ is a smooth (resp.\ \'etale) morphism, then $f^*\colon \cD(X,\lambda)\to\cD(Y,\lambda)$ is t-exact with respect to the t-structures associated to $\sfp$ and $f^*\sfp$.
\end{enumerate}
\end{proposition}

\begin{proof}
There exists $k\ge 2$ such that $X$ and $Y$ are in $\Chpar{k}$ (resp.\ $\Chpdm{k}$). We proceed by induction on $k$. The case $k=-2$ follows from Lemma \ref{3le:gabber} and Lemma \ref{3le:perverse_smooth_pullback} below. The induction step follows from the way as in \cite{LZ1}*{4.3.8, 4.3.9}.
\end{proof}

\begin{lem}\label{3le:perverse_smooth_pullback}
Let $f\colon Y\to X$ be a smooth morphism of schemes in $\Schqcs_\Box$, let $\lambda$ be an object of $\Rind_{\ltor}$, and let $p\colon|X|\to\dZ\cup\{+\infty\}$ be a function. Then $f^!$ carries $\TS{p}{\cD}{\ge 0}(X,\lambda)$ to $\TS{f^*_{\b{2}}p}{\cD}{\ge 0}(Y,\lambda)$. Moreover, if $p$ is a weak perversity function on $X$ and $q$ is a weak perversity function on $Y$ satisfying $f^*_{\b{0}}p\le q \le f^*_{\b{2}}p+2\dim f$, then $f^!\colon\cD(X,\lambda)\to\cD(Y,\lambda)$ is t-exact with respect to the t-structures associated to $p$ and $q$.
\end{lem}

\begin{proof}
The first assertion follows from Lemma \ref{3le:perverse_smooth_pullback0} below. The second assertion follows from the first assertion and the Poincar\'e duality $f^!\simeq f^*\langle\dim f\rangle$.
\end{proof}

\begin{lem}\label{3le:perverse_smooth_pullback0}
Let $f\colon Y\to X$ be a smooth morphism in $\Schqcs_\Box$, and $\lambda$ an object of $\Rind_{\ltor}$. Let $\overline{y}$ be a geometric point of $Y$ above $y$; put $\overline{x}=f(\overline{y})$ and $x=f(y)$. Then there is an equivalence of functors
\[
i_{\overline{y}}^!\circ f^!\simeq g^*\circ i_{\overline{x}}^!\langle d\rangle\colon\cD^{(+)}(X,\lambda)\to\cD^+(\overline{y},\lambda),
\]
where $g\colon\overline{y}\to\overline{x}$ is the induced morphism and $d=\trdeg[k(y):k(x)]$.
\end{lem}

\begin{proof}
Consider the diagram with Cartesian squares
\[
\xymatrix{
\overline{y}\ar[rrd]_g\ar[r]^{i_{\overline{y}}} & V\ar[r]^{j} & Y_{\overline x}\ar[d]_{f_{\overline
x}}\ar[r]^{i'_{\overline x}} & Y_{\overline{\{x\}}} \ar[r]^{i'}\ar[d]^{f_{\overline{\{x\}}}}
& Y\ar[d]^f \\
&&\overline x \ar[r]^{i_{\overline x}} & \overline{\{x\}}\ar[r]^i & X}
\]
where $V$ is a regular integral subscheme of $Y_{\overline{x}}$ such that the image of $\overline y$ in $V$ is a generic point. We have a sequence of equivalences of functors
\[
i_{\overline y}^!\circ f^! \simeq i_{\overline y}^*\circ j^! \circ i_{\overline x}^{\prime*} \circ {i'}^!\circ f^!
\simeq i_{\overline y}^*\circ j^! \circ i_{\overline x}^{\prime*} \circ f_{\overline{\{x\}}}^! \circ i^!
\simeq i_{\overline y}^*\circ j^! \circ f_{\overline x}^! \circ i_{\overline x}^* \circ i^!
\]
which, by the Poincar\'{e} duality, is equivalent to
\[
i_{\overline y}^*\circ (f_{\overline x}\circ j)^! \circ i_{\overline x}^!
\simeq i_{\overline y}^*\circ (f_{\overline x}\circ j)^* \circ i_{\overline x}^! \langle d\rangle
\simeq g^*\circ i_{\overline x}^!\langle d\rangle.
\]
The lemma follows.
\end{proof}

\begin{remark}\label{3re:stalk}
We call the t-structure in Proposition \ref{3pr:perverse_t} the \emph{perverse t-structure} with respect to $\sfp$ and denote by $\TS{\sfp}{\tau}{\leq0}$ and $\TS{\sfp}{\tau}{\geq0}$ the corresponding truncation functors, respectively.
\begin{enumerate}
  \item For every (\'{e}tale) atlas $u\colon X_0\to X$ with $X_0$ a scheme in $\Schqcs$, we have $u^*\circ\TS{\sfp}{\tau}{\leq0}\simeq \TS{\sfp_u}{\tau}{\leq0}\circ u$ and $u^*\circ\TS{\sfp}{\tau}{\geq0}\simeq \TS{\sfp_u}{\tau}{\geq0}\circ u$.

  \item If $\sfp=0$, then we recover the usual t-structure. If $X$ is a higher Deligne-Mumford stack and $\sfp$ is a perversity smooth evaluation, then the t-structure associated to $\sfp$ coincides with the t-structure associated to $\sfp\res\Voet(X)$. If $X$ is in $\Schqcs$, then the t-structure associated to $\sfp$ coincides with the t-structure defined by Gabber (as in Lemma \ref{3le:gabber}) associated to the function $\sfp_{\id_X}$.

  \item Let $\sfK$ be a complex in $\cD(X,\lambda)$. Then by definition,
      \begin{itemize}
        \item $\sfK$ belongs to $\TS{\sfp}{\cD}{\leq n}(X,\lambda)$ if and only if for every pointed smooth (resp.\ \'etale) schematic neighborhood $(X_0,u_0,x_0)$ of $X$ and a geometric point $\overline{x_0}$ lying over $x_0$, we have $i_{\overline{x_0}}^*u_0^*\sfK\in\cD^{\leq\sfp(X_0,u_0,x_0)+n}(\overline{x_0},\lambda)$.

        \item $\sfK$ belongs to $\TS{\sfp}{\cD}{\geq n}(X,\lambda)$ if and only if $\sfK\in\cD^{(+)}(X,\lambda)$, and for every pointed smooth (resp.\ \'etale) schematic neighborhood $(X_0,u_0,x_0)$ of $X$ and a geometric point $\overline{x_0}$ lying over $x_0$, we have $i_{\overline{x_0}}^!u_0^*\sfK\in\cD^{\geq\sfp(X_0,u_0,x_0)+n}(\overline{x_0},\lambda)$.
      \end{itemize}
\end{enumerate}
\end{remark}

At the end of the section, we study the restriction of perverse t-structures constructed above to various subcategory of constructible complexes. We fix a $\Box$-coprime base scheme $\dS$ that is a disjoint union of excellent schemes, endowed with a global dimension function.

\begin{proposition}
Let $\lambda=(\Xi,\Lambda)$ be an object of $\Rind_{\Box\text{-}\r{dual}}$. Let $f\colon X\to\dS$ be an object of $\Chpars$ equipped with an \emph{admissible} perversity smooth evaluation $\sfp$ (Definition \ref{3de:perversity_evaluation}). Then the truncation functors $\TS{\sfp}{\tau}{\leq0}$, $\TS{\sfp}{\tau}{\geq0}$ preserve the full subcategory $\cD_{\cons}^{(\rb)}(X,\lambda)$. Moreover, if $\sfp$ is locally bounded, then $\TS{\sfp}{\tau}{\leq0}$, $\TS{\sfp}{\tau}{\geq0}$ preserve $\cD_{\cons}^{?}(X,\lambda)$ for $?=(+),(-)$ or empty.
\end{proposition}

\begin{proof}
We reduce easily to the case of a scheme. In this case, the result is essentially \cite{Gabber}*{8.2}.
\end{proof}

\subsection{Adic perverse t-structures}
\label{3ss:adic_perverse_t}

For perverse t-structures in the adic formalism, we define
\[
\TS{\sfp}{\cD}{\leq n}(X,\lambda)_\ra=\TS{\sfp}{\cD}{\leq n}(X,\lambda)\cap\cD(X,\lambda)_\ra,\quad
\TS{\sfp}{\cD}{\geq n}(X,\lambda)_\ra=\TS{\sfp}{\cD}{\leq n-1}(X,\lambda)_\ra^{\perp}
\]
both as full subcategories of $\cD(X,\lambda)_\ra$. Then the pair $(\TS{\sfp}{\cD}{\leq 0}(X,\lambda)_\ra,\TS{\sfp}{\cD}{\geq 0}(X,\lambda)_\ra)$ define a t-structure, called the \emph{adic perverse t-structure} with respect to $\sfp$, on $\cD(X,\lambda)_\ra$. Denote
$\TS{\sfp}{\tau}{\leq0}_\ra$ and $\TS{\sfp}{\tau}{\geq0}_\ra$ the corresponding truncation functors respectively. We have the
following results.

\begin{lem}\label{3le:adic_perverse_t_indep}
Let $X$ be a $\Box$-coprime higher Artin stack (resp.\ a higher Deligne--Mumford stack) equipped with a perversity smooth (resp.\ \'etale)
evaluation $\sfp$, and $\lambda$ an object of $\Rind_{\ltor}$ (resp.\ $\Rind$). Let $\sfK\in\cD(X,\lambda)_\ra$ be an (adic) complex. Let
$u\colon X_0\to X$ be an atlas (resp.\ \'{e}tale atlas) with $X_0$ a scheme in $\Schqcs$. Then $\sfK$ belongs to $\TS{\sfp}{\cD}{\leq n}(X,\lambda)_\ra$ (resp.\ $\TS{\sfp}{\cD}{\geq n}(X,\lambda)_\ra$) if and only if $u^{*\ra}\sfK$ belongs to $\TS{\sfp_u}{\cD}{\leq
n}(X_0,\lambda)_\ra$ (resp.\ $\TS{\sfp_u}{\cD}{\geq n}(X_0,\lambda)_\ra$).
\end{lem}

\begin{proof}
We only need to show that $u^{*\ra}$ is t-exact. By definition, we obviously have $u^{*\ra}\TS{\sfp}{\cD}{\leq n}(X,\lambda)_\ra\subseteq\TS{\sfp_u}{\cD}{\leq n}(X_0,\lambda)_\ra$. For the other direction, assume $\sfK\in\TS{\sfp}{\cD}{>n}(X,\lambda)_\ra$, that is, $\Hom(\sfL,\sfK)=0$ for every $\sfL\in\cD(X,\lambda)_\ra\cap\TS{\sfp}{\cD}{\leq n}(X,\lambda)$. By the Poincar\'{e} duality, it suffices to show that for every $\sfL'\in\cD(X_0,\lambda)_\ra\cap\TS{\sfp_u}{\cD}{\leq n-2\dim
u}(X_0,\lambda)$, we have $\Hom(\sfL',u^{!\ra}\sfK)=0$, or equivalently, $\Hom(u_{!\ra}\sfL',\sfK)=0$. This follows from the fact that $u_!$ preserves adic complexes and we have $u_!\sfL'\in\TS{\sfp}{\cD}{\leq n}(X,\lambda)$.
\end{proof}

\begin{proposition}\label{3pr:adic_perverse_stalk}
Let $X$ be a $\Box$-coprime higher Artin stack (resp.\ a higher Deligne--Mumford stack) equipped with a perversity smooth (resp.\ \'etale)
evaluation $\sfp$, and $\lambda$ an object of $\Rind_{\ltor}$ (resp.\ $\Rind$). Let $\sfK\in\cD(X,\lambda)_\ra$ be an (adic) complex.
\begin{enumerate}
  \item Then $\sfK$ belongs to $\TS{\sfp}{\cD}{\leq n}(X,\lambda)_\ra$ if and only if for every pointed smooth (resp.\ \'etale) schematic
      neighborhood $(X_0,u_0,x_0)$ of $X$ and a geometric point $\overline{x_0}$ lying over $x_0$, we have
      $i_{\overline{x_0}}^{*\ra}u_0^{*\ra}\sfK\in\cD^{\leq\sfp(X_0,u_0,x_0)+n}(\overline{x_0},\lambda)_\ra$.

  \item Assume that $\sfp$ is \emph{locally bounded}. Then $\sfK$ belongs to $\TS{\sfp}{\cD}{\geq n}(X,\lambda)_\ra$ if and only if
      $\sfK\in\cD^{(+)}(X,\lambda)_\ra$, and for every pointed smooth (resp.\ \'etale) schematic neighborhood $(X_0,u_0,x_0)$ of
      $X$ and a geometric point $\overline{x_0}$ lying over $x_0$, we have
      $i_{\overline{x_0}}^{!\ra}u_0^{*\ra}\sfK\in\cD^{\geq\sfp(X_0,u_0,x_0)+n}(\overline{x_0},\lambda)_\ra$.
\end{enumerate}
\end{proposition}

\begin{proof}
Part (1) is a consequence of the definition and Remark \ref{3re:stalk} (3).

For (2), by Lemma \ref{3le:adic_perverse_t_indep}, we may assume that $X\in\Schqcs$ is quasi-compact and $\sfp=p$ is a bounded weak perversity
function. Then $\sfK\in\TS{p}{\cD}{\geq n}(X,\lambda)_\ra$ is equivalent to that for every $\sfL\in\TS{p}{\cD}{<n}(X,\lambda)_\ra$, $\HOM(\sfL,\sfK)\in\cD^{>0}(X,\lambda)$, which is then equivalent to $\HOM(\sfL,\sfK)\in\cD^+(X,\lambda)$ and $i_{\overline{x}}^!\HOM(\sfL,\sfK)\in\cD^{>0}(\overline{x},\lambda)$ for every geometric point $\overline{x}$ of $X$. By Proposition \ref{1th:properties} (4), we have isomorphisms
\[
i_{\overline{x}}^!\HOM(\sfL,\sfK)\simeq\HOM(i_{\overline{x}}^*\sfL,i_{\overline{x}}^!\sfK)
\simeq\HOM(i_{\overline{x}}^{*\ra}\sfL,i_{\overline{x}}^{!\ra}\sfK).
\]
Now we may assume $\alpha<p<\beta$ for some $\alpha,\beta\in\dZ$ since $p$ is bounded. Then $\TS{p}{\cD}{<n}(X,\lambda)_\ra$ contains $\cD^{<\alpha+n}(X,\lambda)_\ra$.

Now for $\sfK\in\TS{p}{\cD}{\geq n}(X,\lambda)_\ra$, we have $\sfK\in\cD^{\geq\alpha+n}(X,\lambda)_\ra\subseteq\cD^+(X,\lambda)_\ra$ and
$i_{\overline{x}}^{!\ra}\sfK\in\cD^{\geq p(x)+n}(\overline{x},\lambda)_\ra$ for every geometric point $\overline{x}$ of $X$ lying over $x$.

Conversely, assume $\sfK\in\cD^+(X,\lambda)_\ra$, say in $\cD^{\geq\gamma}(X,\lambda)_\ra$, and $i_{\overline{x}}^{!\ra}\sfK\in\cD^{\geq p(x)+n}(\overline{x},\lambda)_\ra$ for every geometric point $\overline{x}$ of $X$ lying over $x$. We have $\HOM(\sfL,\sfK)\in\cD^{\geq\gamma-\beta-n}(X,\lambda)\subseteq\cD^+(X,\lambda)$ and $\HOM(i_{\overline{x}}^{*\ra}\sfL,i_{\overline{x}}^{!\ra}\sfK)\in\cD^{>0}(\overline{x},\lambda)$. Thus, we have $\sfK\in\TS{p}{\cD}{\geq n}(X,\lambda)_\ra$.
\end{proof}

\begin{remark}\label{3re:perverse_truncation}
Let $\sfp,\sfq$ be two perversity smooth (resp.\ \'etale) evaluations on a $\Box$-coprime higher Artin stack (resp.\ a higher Deligne--Mumford stack) $X$. Let $\lambda$ be an object of $\Rind_{\ltor}$ (resp.\ $\Rind$). Let the subscript $?$ be either ``$\ra$'' or empty.
\begin{enumerate}
  \item The intersection of the pair of subcategories $(\TS{\sfp}{\cD}{\leq0}(X,\lambda)_?,\TS{\sfp}{\cD}{\geq0}(X,\lambda)_?)$ with $\cD^{(+)}(X,\lambda)_?$ induces a t-structure on the latter stable $\infty$-category.

  \item If $\sfp\leq\sfq$, then
      \begin{enumerate}
        \item $\TS{\sfp}{\tau}{\leq0}_?$ preserves $\TS{\sfq}{\cD}{\leq0}(X,\lambda)_?$;

        \item $\TS{\sfq}{\tau}{\geq0}_?$ preserves $\TS{\sfp}{\cD}{\geq0}(X,\lambda)_?$;

        \item $\TS{\sfp}{\tau}{\geq0}_?$ is equivalent to the identity functor when restricted to $\TS{\sfq}{\cD}{\geq0}(X,\lambda)_?$;

        \item $\TS{\sfq}{\tau}{\leq0}_?$ is equivalent to the identity functor when restricted to $\TS{\sfp}{\cD}{\leq0}(X,\lambda)_?$;

        \item $\TS{\sfp}{\tau}{<0}_?$ is equivalent to the null functor when restricted to $\TS{\sfq}{\cD}{\geq0}(X,\lambda)_?$;

        \item $\TS{\sfq}{\tau}{>0}_?$ is equivalent to the null functor when restricted to $\TS{\sfp}{\cD}{\leq0}(X,\lambda)_?$.
      \end{enumerate}

  \item By (2a), if $\sfp$ is locally bounded, then the intersection of the pair of subcategories $(\TS{\sfp}{\cD}{\leq0}(X,\lambda)_?,\TS{\sfp}{\cD}{\geq0}(X,\lambda)_?)$ with $\cD^{(-)}(X,\lambda)_?$ or $\cD^{(\rb)}(X,\lambda)_?$ induces a t-structure on the latter stable $\infty$-category.

  \item By (2e) and (2f), if $X$ is quasi-compact and $\sfp$ is bounded, then there exist constant integers $\alpha<\beta$ such that
      $\TS{\sfp}{\rH}{0}_?=\TS{\sfp}{\rH}{0}_?\circ\tau^{[\alpha,\beta]}_?$, where $\TS{\sfp}{\rH}{0}_?=\TS{\sfp}{\tau}{\geq0}_?\circ\TS{\sfp}{\tau}{\leq0}_?$ is the cohomology functor.
\end{enumerate}
\end{remark}

\subsection{Constructible adic perverse t-structures}
\label{3ss:constructible_adic_perverse_t}

We fix a $\Box$-coprime base scheme $\dS$ that is a disjoint union of schemes that are excellent, quasi-compact, finite dimensional, and admit a global dimension function for which we fix one. We fix also an object $(\Lambda,\fm)$ of $\PRing_{\ltor}$ such that $\Lambda/\fm^{n+1}$ is a
($\Box$-torsion) Gorenstein ring of dimension $0$ for every $n\in\dN$ and $\dS$ is locally $(\Lambda/\fm)$-bounded.

\begin{proposition}
For an object $f\colon X\to\dS$ of $\Chpars$ equipped with an \emph{admissible} perversity evaluation $\sfp$, the truncation functors
$\TS{\sfp}{\tau}{\leq0}_\ra$ and $\TS{\sfp}{\tau}{\geq0}_\ra$ preserve $\cD^{?}(X,\Lambda_\bullet)_{\ra,\rc}$ for $?=(+),(-),(\r b)$ or empty.
\end{proposition}

\begin{proof}
By Lemma \ref{3le:adic_perverse_t_indep}, we may assume that $X$ is a quasi-compact, separated (and excellent, finite dimensional) scheme that is $(\Lambda/\fm)$-bounded, and $\sfp=p$ is an admissible perversity function on $X$. In particular, $p$ is bounded. We prove by Noetherian induction. We may further assume $X$ is irreducible. For a complex $\sfK\in\cD(X,\Lambda_{\bullet})_{\ra,\rc}$, we may assume $\sfK\in\cD^{\rb}(X,\Lambda_{\bullet})_{\ra,\rc}\subseteq\cD^{\rb}(X,\Lambda_{\bullet})$ by Remark \ref{3re:perverse_truncation} (3). Choose
a dense open subset $U$ of $X$ such that
\begin{itemize}
  \item $U$ is essentially smooth;

  \item $p(x)\leq p(\eta)+\codim(x,X)$ for $x\in|U|$, where $\eta$ is the unique generic point of $X$;

  \item $p(x)\geq p(\eta)$ for $x\in |U|$;

  \item the complex $\sfK_U\coloneqq\sfK\res U$, viewed as an element of $\cD^{\rb}(U,\Lambda_{\bullet})$, has smooth almost adic cohomology sheaves.
\end{itemize}
Then the perverse truncation for $\sfK_U$ is simply the usual truncation (up to a shift by $p(\eta)$), which preserves constructibility by Corollary \ref{2co:constructible_adic_t}.
\end{proof}

Our definition of the constructible adic perverse t-structure coincides with Laszlo--Olsson \cite{LO3} under their restrictions, where in particular $X$ is a locally Noetherian ($1$-)Artin stack over a field $k$ (that is, $\dS=\Spec k$) with $\r{cd}_{\ell}(k)<\infty$, and $\sfp$ is the middle perversity smooth evaluation, that is, the unique perverse smooth evaluation such that for every atlas $u\colon X_0\to X$ with $X_0$ a scheme in $\Schqcs$, we have $\sfp_u=(f\circ u)_{\b{1}}^*p_0$, where $f\colon X\to\dS$ is the structure morphism and $p_0$ is the zero perverse function on $\dS=\Spec k$.

\section{Hyperdescent properties}
\label{4}

In this chapter, we study hyperdescent properties for certain operations on stacks. In \Sec\ref{4ss:hyperdescent}, we study some general facts for hyperdescent. In \Sec\ref{4ss:smooth}, \Sec\ref{4ss:proper} and \Sec\ref{4ss:flat}, we study smooth, proper and flat hyperdescent, respectively.

\subsection{Hyperdescent}
\label{4ss:hyperdescent}

In this section, we study hyperdescent properties in the general setup.

\begin{definition}
Let $\cC$, $\cD$ be $\infty$-categories, let $F\colon\cC^{op}\to \cD$ be a functor, and let $X_\bullet^+\colon\rN(\del_+)^{op}\to\cC$ be an augmented simplicial object of $\cC$.
\begin{enumerate}
  \item We say that $X_\bullet^+$ is an \emph{augmentation of $F$-descent} if $F\circ(X_\bullet^+)^{op}$ is a limit diagram in $\cD$.

  \item Assume that $\cC$ admits pullbacks. We say that $X_\bullet^+$ is a \emph{hypercovering for universal $F$-descent} if $X^+_q\to(\cosk_{q-1}(X^+_{\bullet}/X^+_{-1}))_q$ is a morphism of universal $F$-descent for all $q\ge 0$.
\end{enumerate}
\end{definition}

By definition, a morphism of $\cC$ is of $F$-descent \cite{LZ1}*{3.1.1} if and only if its \v{C}ech nerve is an augmentation of $F$-descent. We now give several criteria for $(2)\Rightarrow(1)$.

\begin{proposition}\label{4pr:n_category}
Let $\cC$ be an $\infty$-category admitting pullbacks, let $\cD$ be an $n$-category admitting finite limits for an integer $n\ge 0$, and let $F\colon\cC^{op}\to\cD$ be a functor. Then every hypercovering $X_\bullet^+$ for universal $F$-descent is an augmentation of $F$-descent.
\end{proposition}

To prove Proposition \ref{4pr:n_category}, we need a few lemmas.

\begin{lem}\label{4le:induction}
Let $\cC$, $\cD$ be $\infty$-categories such that $\cC$ admits finite limits, let $F\colon\cC^{op}\to\cD$ be a functor, and let $e$ be a final object of $\cC$. Let $f_\bullet\colon U_\bullet\to V_\bullet$ be a morphism of simplicial objects of $\cC$ such that $V_\bullet \to e$ is an augmentation of $F$-descent and $f_q$ is a morphism of $F$-descent for all $q$. Assume that there exists an integer $n\ge 0$ such that $U_\bullet$ is $n$-coskeletal, $V_\bullet$ is $(n-1)$-coskeletal, and $f_q$ is an equivalence for $q<n$. Then $U_\bullet\to e$ is an augmentation of $F$-descent.
\end{lem}

\begin{proof}
With out lost of generality, we may assume that $F(e)$ is an initial object of $\cD$. Let $W_+\colon\rN(\del_+\times\del)^{op}\to\Fun(\Delta^1,\cC)$ be a \v{C}ech nerve of $f_\bullet$, and put $W\coloneqq W_+\res\rN(\del\times\del)^{op}$. For every $q\ge 0$, $W_+\res\rN(\del_+\times\{[q]\})^{op}$ is a \v{C}ech nerve of $f_q$, which is a morphism of $F$-descent by assumption. It follows that $F\circ W_+^{op}\res\rN(\del_+\times\{[q]\})$ is a limit diagram. We may thus identify the limit of $F\circ W^{op}$ with the limit $F\circ W_+^{op}\res\rN(\{[-1]\}\times\del_s)$. Since $W_+\res\rN(\{[-1]\}\times\del)^{op}$ can be identified with $V_\bullet$, the limit of $F\circ W^{op}$ can be identified with $F(e)$. Put $D_\bullet\coloneqq W\circ \delta$, where $\delta\colon\rN(\del)^{op}\to\rN(\del\times\del)^{op}$ is the diagonal map. Since $\rN(\del)^{op}$ is sifted \cite{Lu1}*{5.5.8.4}, the limit of $F\circ D_\bullet^{op}$ can be identified with $F(e)$. The proof of \cite{Lu1}*{6.5.3.9} exhibits $U_\bullet\res\rN(\del_s)^{op}$ as a retract of $D_\bullet\res\rN(\del_s)^{op}$. It follows that the limit of $F\circ U_\bullet^{op}$ is a retract of $F(e)$, hence is $F(e)$. The lemma follows.
\end{proof}

\begin{lem}\label{4le:coskeletal}
Let $\cC$, $\cD$ be $\infty$-categories such that $\cC$ admits pullbacks, let $F\colon\cC^{op}\to\cD$ be a functor, and let $X_\bullet^+$ be an
$n$-coskeletal hypercovering for universal $F$-descent for an integer $n\ge-1$. Then $X_\bullet^+$ is an augmentation of $F$-descent.
\end{lem}

\begin{proof}
Since morphisms of universal $F$-descent are stable under pullbacks and compositions, the morphism $\cosk_m(X_\bullet^+/X_{-1}^+)\to\cosk_{m-1}(X_\bullet^+/X_{-1}^+)$ satisfies the assumptions of Lemma \ref{4le:induction}. It follows by induction that $\cosk_n(X^+_\bullet/X^+_{-1})$ is an augmentation of $F$-descent.
\end{proof}

\begin{lem}\label{4le:n_category}
Let $n\ge -1$ be an integer, let $\cD$ be an $n$-category admitting finite colimits, and let $f_\bullet \colon Y_\bullet \to X_\bullet$ be a morphism of semisimplicial (resp.\ simplicial) objects of $\cD$ such that $Y_q\to X_q$ is an equivalence for $q\le n$. Then the induced morphism between geometric realizations $|f_\bullet|\colon|Y_\bullet|\to|X_\bullet|$ is an equivalence in $\cD$.
\end{lem}

\begin{proof}
The existence of the geometric realizations is guaranteed by \cite{Lu2}*{1.3.3.10}. The semisimplicial case follows from the simplicial case by taking left Kan extensions. The simplicial case follows from the proof of \cite{Lu2}*{1.3.3.10}.
\end{proof}

\begin{proof}[Proof of Proposition \ref{4pr:n_category}]
It suffices to apply the dual version of Lemma \ref{4le:n_category} to the morphism $h\colon X^+_\bullet \to \cosk_n (X^+_\bullet/X_{-1}^+)$ and Lemma \ref{4le:coskeletal}.
\end{proof}

The following proposition can be used to deduce Gabber's hyper base change theorem \cite{Org}*{Th\'eor\`eme 2.2.5} (see \cite{Zhhyper}*{Remark 2.3}).

\begin{proposition}\label{4pr:t_descent}
Let $\cC$ be an $\infty$-category admitting pullbacks, let $\cD$ be a stable $\infty$-category endowed with a weakly right complete t-structure that either admits countable limits or is right complete, let $F\colon\cC^{op}\to\cD$ be a functor, and let $X_\bullet^+\colon\rN(\del_+)^{op}\to\cC$ be a hypercovering for universal $F$-descent such that $F\circ(X_\bullet^+)^{op}$ factorizes through $\cD^{\ge 0}$. Then $X_\bullet^+$ is an augmentation of $F$-descent.
\end{proposition}

\begin{proof}
Let $n\ge 0$. By Lemma \ref{4le:coskeletal}, $Y_\bullet^+=\cosk_n(X_\bullet^+/X_{-1}^+)$ is an augmentation of $F$-descent, so that it suffices to show that the morphism
\[
c\colon K\coloneqq\lim_{p\in\del}F(X_p)\to L\coloneqq\lim_{p\in\del}F(Y_p)
\]
induced by $h_\bullet\colon X_\bullet^+\to Y_\bullet^+$ is an isomorphism. By \cite{Lu2}*{1.2.4.4, 1.2.4.5}, we have a morphism of converging spectral sequences
\[
\xymatrix{
\rE_1^{p,q}=\rH^q F(X_p)\ar@{=>}[r]\ar[d]_-{c_1^{p,q}} & \rH^{p+q} K\ar[d]^-{\rH^{p+q} c}\\
\TS{\prime}{\rE}{p,q}_1=\rH^q F(Y_p)\ar@{=>}[r] & \rH^{p+q} L,}
\]
concentrated in the first quadrant. For $p\le n$, since $h_p$ is an equivalence, $c_1^{p,q}$ is an isomorphism for all $q$. It follows that
$c_r^{p,q}$ is an isomorphism for $p+q\le n-1$, and $\tau^{\le n-1}c$ is an equivalence. Since $n$ is arbitrary and $\cD$ is weakly right complete, $c$ is an equivalence.
\end{proof}

We denote by $\PSLt$ (resp.\ $\PSRt$) the $\infty$-category defined as follows:
\begin{itemize}
  \item Objects of $\PSLt$ (resp.\ $\PSRt$) are presentable stable $\infty$-categories equipped with a t-structure.

  \item Morphisms of $\PSLt$ (resp.\ $\PSRt$) are t-exact functors admitting right (resp.\ left) adjoints.
\end{itemize}
The $\infty$-categories $\PSLt$ (resp.\ $\PSRt$) admit small limits, and those limits are preserved by the forgetful functor $\PSLt\to\PSL$ (resp.\ $\PSRt\to \PSR$). For a diagram $K\to\PSLt$ or $K\to\PSRt$, $(\lim\cC_k)^{\le 0}$ (resp.\ $(\lim\cC_k)^{\ge 0}$) is the full subcategory of $\lim\cC_k$ spanned by objects whose image in $\cC_k$ is in $\cC_k^{\le 0}$ (resp.\ $\cC_k^{\ge 0}$). For an interval $I\subseteq\dZ$, we have an equivalence $(\lim\cC_k)^{\in I}\to \lim \cC_k^{\in I}$.

We denote by $\PSLt[,\r{wrc}]$ (resp.\ $\PSRt[,\r{rc},\r{wlc}]$) the full subcategory of $\PSLt$ (resp.\ $\PSRt$) spanned by those $\cC$ that are weakly right complete (resp.\ right complete and weakly left complete). This full subcategory is stable under small limits in $\PSLt$ (resp.\ $\PSRt$).

\begin{proposition}\label{4pr:hyperdescent}
Consider a diagram
\[
\xymatrix{
\cD'^{op}\ar[d]_-{j^{op}}\ar[r]^-F & \PSRt[,\r{rc},\r{wlc}]\ar[d]^-{P}\\
\cD^{op}\ar[r]^-{G} & \Cat}
\]
of $\infty$-categories, in which $\cD$ admits pullbacks, $j$ is an inclusion satisfying the right lifting property with respect to $\partial\Delta^n\subseteq\Delta^n$ for $n\ge 2$, and $P$ is the forgetful functor. Assume that the arrows in $\cD'$ are stable under pullbacks in $\cD$ by arrows in $\cD'$. Let $X_\bullet^+\colon\rN(\del_+)^{op}\to\cD$ be a hypercovering for universal $G$-descent such that $X_\bullet^+\res\rN(\del_{s+})^{op}$ factorizes through $j$. Then $X_\bullet^+$ is an augmentation of $G$-descent.
\end{proposition}

\begin{proof}
By the right completeness of $F(X^+_p)$ for $p\ge -1$, it suffices to show that $(F\circ(X_\bullet^+)^{op}\res\rN(\del_{s+}))^{\le 0}$ is a limit diagram. Put $\cC=\lim(F\circ (X_\bullet^+)^{op}\res\rN(\del_s))$ for simplicity. We then have the induced t-exact functor $f^*\colon F(X^+_{-1})\to\cC$. Let $f_!\colon\cC\to F(X^+_{-1})$ be a left adjoint of $f^*$. The restrictions of these provide adjoint functors
\[
(f_!)^{\le 0}\colon\cC^{\le 0}\to F(X^+_{-1})^{\le 0},\quad (f^*)^{\le 0}\colon F(X^+_{-1})^{\le 0}\to\cC^{\le 0}.
\]
Let us first show that $a\colon f_! f^*K \to K$ is an equivalence for all $K\in F(X^+_{-1})^{\le 0}$, namely, that $(f^*)^{\le 0}$ is fully faithful. This is similar to Proposition \ref{4pr:t_descent}. Take $n\ge 0$. The morphism $h_\bullet\colon X^+_\bullet\to\cosk_n(X^+_\bullet/X^+_{-1})=Y^+_\bullet$ induces a diagram
\[
\xymatrix{f_!f^* K \ar[rr]^-{c}\ar[rd]_-{a} && g_!g^* K\ar[dl]^-{b}\\
&K}
\]
where $g_!$ is a left adjoint of the t-exact functor $g^*\colon F(X^+_{-1})\to \lim F\circ (Y^+_\bullet)^{op}\res\rN(\del_s)$. By Lemma
\ref{4le:coskeletal}, $Y^+_\bullet$ is an augmentation of $G$-descent, so that $b$ is an equivalence. Moreover, we have $c=\colim (f_{p!} f_p^* K\to g_{p!}g_p^*K)$, where $f_{p!}$ is a left adjoint of $f_p^*\colon F(X^+_{-1})\to F(X^+_{p})$, $g_{p!}$ is a left adjoint of $g_p^*\colon
F(Y^+_{-1})\to F(Y^+_{p})$, and $f_{p!}f_p^*\to g_{p!}g_p^*$ is induced by $h_p$. By \cite{Lu2}*{1.2.4.4, 1.2.4.5}, we have a morphism of converging spectral sequences
\[
\xymatrix{
\rE_1^{p,q}=\rH^q(f_{-p!}f_{-p}^* K)\ar@{=>}[r]\ar[d]_-{c_1^{p,q}} & \rH^{p+q} f_! f^* K\ar[d]^-{\rH^{p+q} c}\\
\TS{\prime}{\rE}{p,q}_1=\rH^q(g_{-p!}g_{-p}^* K)\ar@{=>}[r] & \rH^{p+q} g_! g^* K,}
\]
concentrated in the third quadrant. For $p\ge -n$, since $h_p$ is an equivalence, $c_1^{p,q}$ is an isomorphism for all $q$. It follows that
$c_r^{p,q}$ is an isomorphism for $p+q\ge 1-n$, and $\tau^{\ge 1-n}c$ is an equivalence. Therefore, $\tau^{\ge 1-n} a$ is an equivalence. Since $n$ is arbitrary and $F(X^+_{-1})$ is weakly left complete, $a$ is an equivalence.

It remains to show that $d\colon L\to f^*f_! L$ is an equivalence for every $L\in \cC^{\le 0}$. Since $\cC$ is weakly left complete, it suffices to show that $\tau^{\ge 1-n}d$ is an equivalence for every $n \ge 1$. For this, we may assume $L\in \cC^{[1-n,0]}$. We will show that $L$ is in the essential image of $(f^*)^{\le 0}$. Since $(f^*)^{\le 0}$ is fully faithful, this proves that $d$ is an equivalence. Let $H\colon\PSRt[,\r{rc},\r{wlc}]\to\c{C}\r{at}_n$ be the functor sending $\cF$ to $\cF^{[1-n,0]}$, where $\c{C}\r{at}_n$ is the $\infty$-category of $n$-categories. It suffices to show that $H\circ F\circ (X^+_\bullet)^{op}\res\rN(\del_{s+})$ is a limit diagram. Since $\c{C}\r{at}_n$ is an $(n+1)$-category, we may assume that $X^+_\bullet/X^+_{-1}$ is $(n+1)$-coskeletal by Lemma \ref{4le:n_category} applied to $X^+_\bullet\to\cosk_{n+1}(X^+_\bullet/X^+_{-1})$. In this case, $F\circ (X^+_\bullet)^{op}\res\rN(\del_{s+})$ is a limit diagram by Lemma \ref{4le:coskeletal}.
\end{proof}

The following variant of Proposition \ref{4pr:hyperdescent} will be used to establish proper hyperdescent. To state it conveniently, we introduce a bit of terminology. Let $\cC$ be an $\infty$-category admitting pullbacks, and $F\colon\cC^{op}\to \Cat$ a functor. We say that a morphism $f$ of $\cC$ is \emph{$F$-conservative} if $F(f)$ is conservative. We say that $f$ is \emph{universally $F$-conservative} if every pullback of $f$ in $\cC$ is $F$-conservative. We say that an augmented simplicial object $X_\bullet^+$ of $\cC$ is a \emph{hypercovering for universal $F$-conservativeness} if $X^+_n\to (\cosk_{n-1}(X^+_\bullet/X^+_{-1}))_n$ is universally $F$-conservative for every $n\ge 0$.

\begin{proposition}\label{4pr:hyperdescent2}
Let $\cC$ be an $\infty$-category admitting pullbacks, let $F\colon\cC^{op}\to\PSLt[,\r{wrc}]$ be a functor, and let $a$ be an integer.
\begin{enumerate}
  \item Let $G\colon\PSLt[,\r{wrc}]\to\Cat$ be the functor sending $\cC$ to $\cC^{\ge a}$. If $X_\bullet^+$ is a hypercovering for universal $(G\circ F)$-descent, then it is an augmentation of $(G\circ F)$-descent.

  \item Let $G\colon\PSLt[,\r{wrc}]\to\Cat$ be the functor sending $\cC$ to $\cC^+\coloneqq\bigcup_n \cC^{\ge n}$. If $X_\bullet^+$ is a hypercovering for universal $(G\circ F)$-descent and for universal $(P\circ F)$-conservativeness, where $P\colon \PSLt[,\r{wrc}]\to \Cat$ is the forgetful functor, then it is an augmentation of $(G\circ F)$-descent.
\end{enumerate}
\end{proposition}

\begin{proof}
The proof for (1) is similar to the proof of Proposition \ref{4pr:hyperdescent}. For (2), the conservativeness implies that $G(\lim F\circ (X^+_\bullet)^{op})\to \lim G\circ F\circ(X_\bullet^+)^{op}$ is an equivalence. The rest of the proof is similar.
\end{proof}

\subsection{Smooth hyperdescent}
\label{4ss:smooth}

The \'etale $\infty$-topos of an affine scheme is not hypercomplete (see \cite{Lu1}*{\Sec6.5} for the definition) in general. By contrast, the stable $\infty$-categories we constructed satisfy smooth hyperdescent.

We regard the map
\begin{align*}
\EO{}{\Chpar{}}\otimes{}\coloneqq(\EO{}{\Chpar{}}{\r{I}}{})^\otimes\colon\rN(\Chpar{})^{op}\times\rN(\Rind)^{op}\to\PSLM
\end{align*}
and the map
\begin{align*}
\EO{}{\Chpar{}_\Box}{}{!}\colon \rN(\Chpar{}_\Box)_F\times\rN(\Rind_{\ltor})^{op}\to\PSL
\end{align*}
from \cite{LZ1} as functors
\begin{align*}
\EO{}{\Chpar{}}\otimes{}&\colon\rN(\Chpar{})^{op}\to\Fun(\rN(\Rind)^{op},\PSLM),\\
\EO{}{\Chpar{}_\Box}{}{!}&\colon\rN(\Chpar{}_\Box)_F\to\Fun(\rN(\Rind_{\ltor})^{op},\PSL).
\end{align*}
In the adic case, we have similar functors
\begin{align*}
\EO{\ra}{\Chpar{}}\otimes{}&\colon\rN(\Chpar{})^{op}\to\Fun(\rN(\Rind)^{op},\PSLM),\\
\EO{\ra}{\Chpar{}_\Box}{}{!}&\colon\rN(\Chpar{}_\Box)_F\to\Fun(\rN(\Rind_{\ltor})^{op},\PSL).
\end{align*}
from Proposition \ref{1pr:monoidal} and \eqref{1eq:lowersh_adic}, respectively.

\begin{definition}\label{4de:covering}
We say that an augmented simplicial object $X_\bullet^+$ in $\Chpar{}$ (or similar $\infty$-categories) is a \emph{(P) hypercovering} for a property (P) on morphisms if $X^+_q \to (\cosk_{q-1}(X^+_{\bullet}/X^+_{-1}))_q$ is \emph{surjective} and satisfies (P) for every $q\ge 0$.
\end{definition}

\begin{proposition}\label{4pr:hyperdescent_stack}
Every smooth hypercovering in $\Chpar{}$ (resp.\ $\Chpar{}_\Box$) is an augmentation of both $\EO{}{\Chpar{}}\otimes{}$-descent (resp.\
$\EO{}{\Chpar{}_\Box}{op}{!}$-descent) and $\EO{\ra}{\Chpar{}}\otimes{}$-descent (resp.\ $\EO{\ra}{\Chpar{}_\Box}{op}{!}$-descent).
\end{proposition}

\begin{proof}
Let $X_\bullet^+$ be an augmented simplicial object of $\Chpar{}$ (resp.\ $\Chpar{}_\Box$). It suffices to apply Proposition \ref{4pr:hyperdescent} to the full subcategory $\Chpar{}_{\r{sm}/X_{-1}}\subseteq \Chpar{}_{/X_{-1}}$ spanned by higher Artin stacks smooth over $X_{-1}$. In the notation of Proposition \ref{4pr:hyperdescent}, $F$ associates the usual t-structure (resp.\ the usual t-structure shifted by twice the relative dimension over $X_{-1}$). This proof applies to both the non-adic case and the adic case. The adic case can also be deduced from the non-adic case by taking limits.
\end{proof}

\subsection{Proper hyperdescent}
\label{4ss:proper}

In this section, we study hyperdescent properties for proper morphisms. We start from some lemmas for preparation.

\begin{lem}\label{4le:real}
Let $\cC$ and $\cD$ be stable $\infty$-categories equipped with left complete t-structures. Let $F\colon \cC\to \cD$ be a t-exact functor. Then
$\cC^{\le 0}$ admits geometric realizations, and geometric realizations are preserved by $F$.
\end{lem}

\begin{proof}
By \cite{Lu2}*{1.2.4.5}, for any simplicial object $X_\bullet$ of $\cC$, there exist a geometric realization $X=\lvert X_\bullet \rvert$ in $\cC$ and a geometric realization $Y=\lvert FX_\bullet\rvert$ in $\cD$, and $\rH^n(f)$ is an isomorphism for all $n$,  where $f$ is the morphism $Y\to FX$. It follows that $f$ is an equivalence.
\end{proof}

\begin{lem}\label{4le:split}
Let $\cC$, $\cD$, $\cE$ be stable $\infty$-categories equipped with t-structures such that $\cC$ and $\cD$ are both left and right complete. Let $F\colon \cC\to \cD$ and $G\colon \cC\to \cE$ be t-exact functors. Assume $G$ conservative. Then $\cC$ admits $G$-split \cite{Lu2}*{4.7.3.2} geometric realizations, and those geometric realizations are preserved by $F$.
\end{lem}

\begin{proof}
Let $X_\bullet$ be a $G$-split simplicial object of $\cC$, and $Y_\bullet\colon\rN(\del_+)^{op}\to\cD$ a split augmentation of $G\circ
X_\bullet$. Then the unnormalized cochain complex
\[
\dots \to \rH^q Y_2\to \rH^q Y_1 \to \rH^q Y_0\to \rH^q Y_{-1} \to 0
\]
is acyclic. Since $G$ is conservative, it follows that the unnormalized cochain complex
\[
\dots \to \rH^q X_2\to \rH^q X_1 \xrightarrow{\theta^q} \rH^q X_0
\]
is an acyclic resolution of the object $A^q=\r{coker}(\theta^q)$ in the heart of $\cC$ and the same holds after applying the functor $F$. By \cite{Lu2}*{1.2.4.12}, $X_\bullet$ admits a geometric realization $X$, $FX_\bullet$ admits a geometric realization $Z$, and $\rH^n(f)$ is an
isomorphism for all $n$, where $f$ is the morphism $Z\to FX$. It follows that $f$ is an equivalence.
\end{proof}

The functor $\EO{}{\Chpar{}}\otimes{}$ restricts to a functor
\[
\EO{\geq 0}{\Chpar{}_\Box}{*}{}\colon\rN(\Chpar{}_\Box)^{op}\to\Fun(\rN(\Rind_{\ltor})^{op},\Cat)
\]
sending $X$ to the assignment $\lambda\mapsto\cD^{\geq 0}(X,\lambda)$.

\begin{proposition}\label{4pr:proper_descent}
Let $\dS$ be a $\Box$-coprime (resp.\ $\Box$-coprime \emph{locally Noetherian}, that is, there exists an atlas $S\to\dS$ where $S$ is a locally
Noetherian scheme) higher Artin stack.
\begin{enumerate}
  \item For every object $\lambda$ of $\Rind_{\ltor}$ and every Cartesian square
      \[
      \xymatrix{
      W \ar[r]^-{g} \ar[d]_-{q} & Z \ar[d]^-{p} \\
      Y \ar[r]^-{f} & X }
      \]
      in $\Chpar{}_\Box$ (resp.\ $\Chpars$) with $p$ proper of finite diagonal (resp.\ proper and $1$-Artin), the induced square
      \[
      \xymatrix{
      \cD^{\ge 0}(Z,\lambda)  \ar[d]_-{g^*} & \cD^{\ge 0}(X,\lambda) \ar[l]_-{p^*} \ar[d]^-{f^*}\\
      \cD^{\ge 0}(W,\lambda) & \cD^{\ge 0}(Y,\lambda) \ar[l]_-{q^*} }
      \]
      is right adjointable.

  \item Every proper finite-diagonal hypercovering in $\Chpar{}_\Box$ (resp.\ proper and $1$-Artin hypercovering in $\Chpars$) is an augmentation of $\EO{\geq 0}{\Chpar{}_\Box}{*}{}$-descent.
\end{enumerate}
\end{proposition}

\begin{proof}
Let us first show that (1) implies (2). By Proposition \ref{4pr:hyperdescent2}, to show (2), it suffices to show that every surjective morphism proper of finite diagonal (resp.\ proper and $1$-Artin) is of $\EO{\geq 0}{\Chpar{}_\Box}{*}{}$-descent. For this, we apply \cite{LZ1}*{3.3.6}: condition (1) follows from the dual of Lemma \ref{4le:real}; the Beck--Chevalley condition (2) is simply part (1); condition (3) is clear.

To show (1), applying \cite{LZ1}*{4.3.6} and the smooth base change, we are reduced to the case where $X$ and $Y$ are in $\Schqcs$. In this case, there exists a finite \cite{Rydh}*{Theorem~B} (resp.\ proper \cite{OlChow}*{1.1}) surjective morphism $r_0\colon Z_0\to Z$ with $Z_0$ a scheme. Since (1) is known in the case where $p$ is proper and $0$-Artin, $r_0$ is $\EO{\geq 0}{\Chpar{}_\Box}{*}{}$-descent by the above proof of (2). Thus every object of $\cD^{\ge 0}(Z,\lambda)$ has the form $\lim_{n\in \del}r_{n*}r_n^*\sfK$, where $r_\bullet$ is a \v{C}ech nerve of $r_0$. By Lemma \ref{4le:real}, the functors $f^*$ and $g^*$ preserve limits indexed by $\del$. Thus it suffices to check that the natural transformation $f^*\circ p_*\circ r_{n*}\to q_*\circ g^*\circ r_{n*}$ is a natural equivalence. This follows from the known cases of (1) with $p$ replaced by the proper $0$-Artin morphisms $r_n$ and $p\circ r_n$.
\end{proof}

The above result can be extended to $\cD(X,\lambda)^\otimes$ under cohomological finiteness conditions. We fix an object $\lambda$ of $\Rind_{\ltor}$. The functors $\EO{}{\Chpar{}}\otimes{}$ and $\EO{\ra}{\Chpar{}}\otimes{}$ restrict to functors
\begin{align*}
\EO{}{\Chpar{}_\Box}{\otimes}{\lambda}&\colon\rN(\Chpar{}_\Box)^{op}\to\PSLM,\\
\EO{\ra}{\Chpar{}_\Box}{\otimes}{\lambda}&\colon\rN(\Chpar{}_\Box)^{op}\to\PSLM
\end{align*}
sending $X$ to $\cD(X,\lambda)^\otimes$ and $\cD(X,\lambda)_\ra^\otimes$, respectively.

\begin{proposition}\label{4pr:proper_descent2}
Let $\dS$ be a $\Box$-coprime (resp.\ $\Box$-coprime \emph{locally
Noetherian}) higher Artin stack. Let $\lambda$ be an object of
$\Rind_{\ltor}$.
\begin{enumerate}
  \item Consider a Cartesian square
      \[
      \xymatrix{
      W \ar[r]^-{g} \ar[d]_-{q} & Z \ar[d]^-{p} \\
      Y \ar[r]^-{f} & X }
      \]
      in $\Chpar{}_\Box$ (resp.\ $\Chpars$) with $p$ proper of finite
      diagonal (resp.\ proper and $1$-Artin). Assume that for every
      morphism $U\to X$ locally of finite type with $U$ an affine scheme,
      $X_0$ is $\lambda$-cohomologically finite. Then the induced square
      \[
      \xymatrix{
      \cD(Z,\lambda)  \ar[d]_-{g^*} & \cD(X,\lambda) \ar[l]_-{p^*} \ar[d]^-{f^*}\\
      \cD(W,\lambda) & \cD(Y,\lambda) \ar[l]_-{q^*} }
      \]
      is right adjointable.

  \item Let $X^+_\bullet$ be a proper finite-diagonal hypercovering in
      $\Chpar{}_\Box$ (resp.\ proper and $1$-Artin hypercovering in
      $\Chpars$). Assume that for every morphism $U\to X^+_{-1}$ locally
      of finite type with $U$ an affine scheme, $X_0$ is
      $\lambda$-cohomologically finite. Then $X^+_\bullet$ is an
      augmentation of both
      $\EO{}{\Chpar{}_\Box}{\otimes}{\lambda}$-descent and
      $\EO{\ra}{\Chpar{}_\Box}{\otimes}{\lambda}$-descent.
\end{enumerate}
\end{proposition}

\begin{proof}
We first show that (1) implies (2) for $\EO{}{\Chpar{}_\Box}{\otimes}{\lambda}$-descent. One only needs to repeat the proof of Proposition \ref{4pr:proper_descent} with Proposition \ref{4pr:hyperdescent2} replaced by Proposition \ref{4pr:hyperdescent} and Lemma \ref{4le:real} replaced by Lemma \ref{4le:split}. Note that the case for $\EO{}{\Chpar{}_\Box}{\otimes}{\lambda}$-descent implies the case for $\EO{\ra}{\Chpar{}_\Box}{\otimes}{\lambda}$-descent by \cite{LZ1}*{3.1.4}.

The proof for (1) is similar to Proposition \ref{4pr:proper_descent} since $r_0$ is of $\EO{}{\Chpar{}_\Box}{\otimes}{}$-descent as well.
\end{proof}

\subsection{Flat hyperdescent}
\label{4ss:flat}

The following proposition is an analogue of flat cohomological descent \cite{SGA4}*{Vbis Proposition 4.3.3 c)}.

\begin{proposition}\label{4pr:flat}
Every flat and locally finitely presented hypercovering of higher Artin stacks is an augmentation of $\EO{\geq 0}{\Chpar{}_\Box}{*}{}$-descent.
\end{proposition}

\begin{proof}
By Proposition \ref{4pr:hyperdescent2}, we are reduced to show that every surjective flat and locally finitely presented morphism $f\colon Y\to X$ in $\Chpar{}_\Box$ is of $\EO{\geq 0}{\Chpar{}_\Box}{*}{}$-descent. By \cite{LZ1}*{Lemma 3.1.2} and the smooth descent, we are reduced to the case of schemes. Let $X'$ be a disjoint union of strict localizations of $X$, such that the morphism  is surjective. By \cite{EGAIV}*{17.16.2, 18.5.11}, there exists a surjective \'etale morphism of schemes $g\colon X'\to X$ and a finite surjective morphism of schemes $g'\colon Z\to X'$ in $\Schqcs$ such that the composite morphism $Z\to X$ factorizes through $f$. By \cite{LZ1}*{Lemma 3.1.2} and \'etale descent, it suffices to show that $g'$ is of universal $\EO{\geq 0}{\Schqcs}{*}{}$-descent. For this, we apply \cite{LZ1}*{Lemma 3.3.6}. Condition (1) follows from the dual of Lemma \ref{4le:real}. Condition (2) follows from finite base change. Condition (3) is clear.
\end{proof}

The above proposition can be extended to $\cD(X,\lambda)^\otimes$ under cohomological finiteness conditions, similar to the case of proper hyperdescent. We leave details to the reader.

\begin{remark}\label{4re:etale}
We define the $\infty$-category of \emph{$\infty$-DM stacks} $\Chpdm{\infty}$ to be the $\infty$-category $\r{Sch}(\cG_{\et}(\dZ))$ of $\cG_{\et}(\dZ)$-schemes in the sense of \cite{LuV}*{2.3.9, 2.6.11}. Using Proposition \ref{4pr:hyperdescent}, we can adapt the DESCENT program
\cite{LZ1}*{\Sec 4} to define the first and the second enhanced operation maps for $\infty$-DM stacks, namely, a functor
\[
\EO{}{\Chpdm{\infty}}{\r{I}}{}\colon((\Chpdm{\infty})^{op}\times\rN(\Rind)^{op})^\amalg\to\Cat
\]
that is a weak Cartesian structure, and a map
\[
\EO{}{\Chpdm{\infty}}{\r{II}}{}\colon\delta^*_{2,\{2\}}(((\Chpdm{\infty})^{op}\times\rN(\Rind_\tor)^{op})^{\amalg,op})^\cart_{F,\all}\to\Cat.
\]
Applying the construction in \Sec\ref{1ss:limit}, we obtain the first and the second enhanced adic operation maps for $\infty$-DM stacks, namely, a functor
\[
\EO{\ra}{\Chpdm{\infty}}{\r{I}}{}\colon((\Chpdm{\infty})^{op}\times\rN(\Rind)^{op})^\amalg\to\Cat
\]
that is a weak Cartesian structure, and a map
\[
\EO{\ra}{\Chpdm{\infty}}{\r{II}}{}\colon\delta^*_{2,\{2\}}(((\Chpdm{\infty})^{op}\times\rN(\Rind_\tor)^{op})^{\amalg,op})^\cart_{F,\all}\to\Cat.
\]
By restriction, we have similar functors $\EO{}{\Chpdm{\infty}}{}{!}$ and $\EO{\ra}{\Chpdm{\infty}}{}{!}$. Parallel to Proposition \ref{4pr:hyperdescent_stack}, we have that every smooth hypercovering in $\Chpdm{\infty}$ is an augmentation of both $\EO{}{\Chpdm{\infty}}\otimes{}$-descent (resp.\ $\EO{}{\Chpdm{\infty}}{op}{!}$-descent) and $\EO{\ra}{\Chpdm{\infty}}\otimes{}$-descent (resp.\ $\EO{\ra}{\Chpdm{\infty}}{op}{!}$-descent). We have similar results for proper and flat hyperdescent.
\end{remark}

\end{document}